\documentclass[11pt]{amsart}
\usepackage{amsfonts, amsmath, amssymb, amsthm, stmaryrd, color, enumerate, relsize}
\usepackage{mathrsfs,array}
\usepackage{xy}
\usepackage{hyperref}
\usepackage{tikz-cd}
\input xy
\xyoption{all}

\DeclareMathOperator{\intHom}{\!{\mathscr{H}\! om}}
\newcommand{\dotimes}{\otimes^{\mathbb L}}

\numberwithin{equation}{section}
\newtheorem{theorem}{Theorem}
\numberwithin{theorem}{section}
\newtheorem{lemma}[theorem]{Lemma}
\newtheorem{corollary}[theorem]{Corollary}
\newtheorem{proposition}[theorem]{Proposition}

\newtheorem{definition}[theorem]{Definition}

\theoremstyle{definition}
\newtheorem{remark}[theorem]{Remark}
\newtheorem{example}[theorem]{Example}

\date{\today}

\title{Berkovich Motives}
\author{Peter Scholze}
\begin{document}

\begin{abstract} We construct a theory of (\'etale) Berkovich motives. This is closely related to Ayoub's theory of rigid-analytic motives, but works uniformly in the archimedean and nonarchimedean setting. We aim for a self-contained treatment, not relying on previous work on algebraic or analytic motives. Applying the theory to discrete fields, one still recovers the \'etale version of Voevodsky's theory. Two notable features of our setting which do not hold in other settings are that over any base, the cancellation theorem holds true, and under only minor assumptions on the base, the stable $\infty$-category of motivic sheaves is rigid dualizable.
\end{abstract}

\maketitle

\tableofcontents

\section{Introduction}

This paper is motivated by an application to the geometrization of the local Langlands correspondence \cite{FarguesScholze}, namely to \cite[Conjecture I.9.5]{FarguesScholze} on the independence of $\ell$ of the Langlands parameters constructed there. This application is the subject of the companion paper \cite{MotivicGeometrization}.

In \cite{FarguesScholze} we used the theory of \'etale cohomology of rigid-analytic varieties \cite{Huber} and their generalization, diamonds \cite{ECoD}. \'Etale cohomology only works well with torsion coefficients prime to the characteristic, naturally leading to $\ell$-adic coefficients for some auxiliary prime $\ell\neq p$. Grothendieck's vision was that these different $\ell$-adic cohomologies should naturally be shadows of an underlying ``motive''. While a full theory of motives remains elusive, much progress was made by Voevodsky et.~al.~through the introduction of a stable $\infty$-category $\mathcal D_{\mathrm{mot}}(S)$ of motivic sheaves on $S$, for any scheme $S$, equipped with all six functors \cite{AyoubAlgebraic}. Having six functors with all usual properties is sufficient for many applications, although usually one eventually runs into the issue that the categories $\mathcal D_{\mathrm{mot}}(S)$ remain inexplicit, even when $S$ is a geometric point. Namely, rationally morphisms in this category are governed by rational $K$-theory groups and these remain challenging to compute, being the subject of the Hodge and Tate conjectures, and also the Beilinson--Soul\'e vanishing conjecture. Notably, endowing $\mathcal D_{\mathrm{mot}}(S)$ with the ``motivic $t$-structure'' remains unknown.

For the applications to local Langlands, one however needs a variant of this theory that applies to some generalization of rigid-analytic varieties, namely diamonds: Quotients of perfectoid spaces by pro-\'etale equivalence relations. Even more generally, the formalism should apply to v-stacks. Such a theory of motives has been constructed through the works of Ayoub \cite{AyoubRigid}, Vezzani \cite{Vezzani} and Ayoub--Gallauer--Vezzani \cite{AyoubGallauerVezzani}. Their approach is to handle rigid-analytic varieties via formal models, eventually reducing to algebraic motives. General v-stacks need to be reduced to rigid-analytic varieties in turn.

The purpose of this paper is to observe that one can also build the theory from scratch, in a language that is very close to the one used to discuss \'etale cohomology of diamonds \cite{ECoD} -- in particular, making use of very fine Grothendieck topologies like the v-topology, and strictly totally disconnected spaces. We decided, however, to write the results in the Berkovich setting instead of the adic setting, for the following reasons:
\begin{enumerate}
\item[{\rm (1)}] The theory works similarly in the archimedean and non-archimedean settings, which can be nicely combined in the Berkovich setting.
\item[{\rm (2)}] The cancellation theorem holds true over any base, making the category of motivic sheaves more concrete.
\item[{\rm (3)}] For many $X$, the stable $\infty$-category $\mathcal D_{\mathrm{mot}}(X)$ of motivic sheaves on $X$ is rigid dualizable.
\item[{\rm (4)}] We believe that the symmetric monoidal presentable $(\infty,2)$-category naturally built from the $6$-functor formalism classifies certain ``normed ring stacks'' in a suitable sense, and yields a nontrivial example of a $2$-rigid symmetric monoidal presentable $(\infty,2)$-category.
\end{enumerate}

However, as shown in \cite{AyoubGallauerVezzani}, there is a parallel theory for adic spaces.

Let us now explain our results in some more detail. We work in the category of Banach rings, which are rings $A$ equipped with a norm $|\cdot|_A: A\to \mathbb R_{\geq 0}$ such that $A$ is Cauchy-complete. Any ring $A$ admits the discrete norm $|a|_A = 1$ for $a\neq 0$, embedding discrete rings into Banach rings. For any Banach ring $A$, Berkovich has defined the Berkovich spectrum $\mathcal M(A)$, which is a compact Hausdorff space, classifying maps from $A$ to Banach fields. Here, a Banach field is a Banach ring $(K,|\cdot|_K)$ such that $K$ is a field and $|\cdot|_K$ is multiplicative. There are three types of Banach fields:
\begin{enumerate}
\item[{\rm (1)}] Discrete fields.
\item[{\rm (2)}] Archimedean Banach fields, which are by Gelfand--Mazur just $\mathbb R$ or $\mathbb C$, with norm a power $|\cdot|^\alpha$, $\alpha\in (0,1]$, of the standard norm $|\cdot|$.
\item[{\rm (3)}] Nonarchimedean Banach fields $(K,|\cdot|_K)$. These have as ring of integers a rank $1$ valuation ring $\mathcal O_K$.
\end{enumerate}

We sidestep the problem of defining a geometric category of Berkovich spaces and instead directly endow (the opposite of) the category of Banach rings with a very fine Grothendieck topology, and will proceed by looking at the corresponding category of stacks. The following notion is analogous to the notion of v-covers used in \cite{ECoD}, but as in Berkovich spaces one only records rank $1$ valuations it is closer to the arc-topology defined for schemes in \cite{BhattMathew}.

\begin{definition} A collection of maps $\{A\to B_i\}_{i\in I}$ of Banach rings defines an arc-cover if there is a finite subset $J\subset I$ such that the induced map
\[
\bigsqcup_{i\in J} \mathcal M(B_i)\to \mathcal M(A)
\]
is surjective.
\end{definition}

This topology is far from subcanonical, but it turns out that invariants of ``topological'' nature (such as the underlying Berkovich space, or the \'etale site, or \'etale cohomology) satisfy arc-descent. For any Banach ring $A$, we denote by $\mathcal M_{\mathrm{arc}}(A)$ the arc-sheafification of the presheaf $B\mapsto \mathrm{Hom}(A,B)$ represented by $A$.

There are the usual set-theoretic issues that the category of Banach rings is not small. For any uncountable regular cardinal $\kappa$, one can restrict to $\kappa$-small Banach rings, meaning those whose underlying set has cardinality less than $\kappa$; and it is easy to see that when increasing $\kappa$, the corresponding pullback functors on categories of arc-sheaves are fully faithful. We will work with the union of these categories over all $\kappa$; this is the category of small arc-sheaves, i.e.~those that are a small colimit of $\mathcal M_{\mathrm{arc}}(A)$'s. Equivalently, these are those arc-sheaves that also commute with $\kappa$-filtered colimits, for some $\kappa$.

Working locally in the arc-topology, one can always assume that there is some topologically nilpotent unit $T$ with $|T^{-1}|=|T|^{-1}=2$. For example, for a discrete ring $R$, the map $R\to R((T))_{1/2}$ is an arc-cover. Here the subscript $\tfrac 12$ refers to the norm that we give the element $T$. We will generally be working in the ``analytic'' setting where such a $T$ exists; arc-descent will then allow us to descend back to the algebraic setting of discrete rings. This way, our theory of Berkovich motives will recover Voevodsky's theory of (\'etale) motives over a field.

Let us now cut straight to the main definitions. For any small arc-stack $X$, we have the arc-site $X_{\mathrm{arc}}$ over $X$ (just given by a slice), and a resulting stable $\infty$-category $\mathcal D(X_{\mathrm{arc}},\mathbb Z)$ of arc-hypersheaves with values in $\mathcal D(\mathbb Z)$.

\begin{definition}\label{def:effectivemotivesintro} The stable $\infty$-category of effective motivic sheaves over $X$ is the full subcategory
\[
\mathcal D_{\mathrm{mot}}^{\mathrm{eff}}(X)\subset \mathcal D(X_{\mathrm{arc}},\mathbb Z)
\]
of all finitary and ball-invariant sheaves.
\end{definition}

Here, a sheaf is ball-invariant if for any $Y$ over $X$, the values on $Y$ and $\mathbb B_Y$ agree, where $\mathbb B_Y$ denotes a $1$-dimensional ball (of radius $1$) over $Y$. We would like to say that a sheaf is finitary if it commutes with filtered colimits of Banach rings. However, this requirement is slightly incompatible with the requirement of being an arc-sheaf, so the correct condition is to ask for commutation with filtered colimits only on the class of strictly totally disconnected Banach rings.

\begin{definition} A Banach ring $A$ is strictly totally disconnected if $A$ is uniform\footnote{The notion of uniformity is recalled in Definition~\ref{def:uniform}.}, $\mathcal M(A)$ is a profinite set, and for all $x\in \mathcal M(A)$ the corresponding residue field $K(x)$ is algebraically closed and non-discrete.
\end{definition}

In other words, $K(x)$ is either $\mathbb C$, or an algebraically closed nonarchimedean field.

\begin{remark} Definition~\ref{def:effectivemotivesintro} is very close to the definition of $\mathcal D_{\mathrm{et}}$ used in \cite{ECoD}. There, we also isolate a full subcategory in the derived category of v-sheaves; it is the subcategory of finitary sheaves that are invariant under change of algebraically closed field. We will see that with torsion coefficients, the condition of invariance under change of algebraically closed field is equivalent to ball-invariance, but the latter works better with rational coefficients.
\end{remark}

One can describe various parts of $\mathcal D_{\mathrm{mot}}^{\mathrm{eff}}(X)$ explicitly. By arc-descent, it suffices to do this when $X$ is strictly totally disconnected, and we restrict here for simplicity to the case that $X$ is a point.

\begin{proposition} Let $X=\mathcal M_{\mathrm{arc}}(C)$ for some algebraically closed non-discrete Banach field $C$.
\begin{enumerate}
\item[{\rm (i)}] If $C=\mathbb C$, then
\[
\mathcal D(\mathbb Z)\cong \mathcal D_{\mathrm{mot}}^{\mathrm{eff}}(X)\subset \mathcal D(X_{\mathrm{arc}},\mathbb Z)
\]
and the essential image agrees with the subcategory of finitary sheaves.
\item[{\rm (ii)}] If $C$ is nonarchimedean and $n\in \mathbb Z$ is invertible in the residue field of $C$, then
\[
\mathcal D(\mathbb Z/n\mathbb Z)\cong \mathcal D_{\mathrm{mot}}^{\mathrm{eff}}(X,\mathbb Z/n\mathbb Z)\subset \mathcal D(X_{\mathrm{arc}},\mathbb Z/n\mathbb Z).
\]
For finitary sheaves of $\mathbb Z/n\mathbb Z$-modules, being ball-invariant is equivalent to being invariant under change of algebraically closed field.
\item[{\rm (iii)}] If $C$ is nonarchimedean and of residue characteristic $p$, then $p$ is invertible in $\mathcal D_{\mathrm{mot}}^{\mathrm{eff}}(X)$.
\item[{\rm (iv)}] If $C$ has mixed characteristic and tilt $C^\flat$, then tilting induces an equivalence
\[
\mathcal D_{\mathrm{mot}}^{\mathrm{eff}}(C)\cong \mathcal D_{\mathrm{mot}}^{\mathrm{eff}}(C^\flat).
\]
\end{enumerate}
\end{proposition}

Part (i) shows that over $\mathbb C$, the theory is equivalent to the theory of Betti sheaves, while (ii) shows that with torsion coefficients it is equivalent to the theory of \'etale sheaves (necessarily of order prime to the residue characteristic, by (iii)). Part (iv) has (in slightly different language) first been obtained by Vezzani \cite{Vezzani}.

Thus, the part not covered by classical theories is the part with $\mathbb Q$-coefficients for nonarchimedean fields.

There is a natural tensor product $-\otimes_{\mathrm{mot}}-$ on $\mathcal D_{\mathrm{mot}}^{\mathrm{eff}}(X)$. More precisely, the tensor product on $\mathcal D(X_{\mathrm{arc}},\mathbb Z)$ restricts to a tensor product on the subcategory of finitary sheaves; and ball-invariance naturally defines a Verdier quotient $\mathcal D_{\mathrm{mot}}^{\mathrm{eff}}(X)$ to which the tensor product descends. Explicitly, one has to make the arc-tensor product ball-invariant again.

The Tate twist can be described explicitly, as usual. First, for any smooth map $f: Y\to X$ -- for simplicity, assume $Y$ is an open subset of some finite-dimensional affine space $\mathbb A^n_X$ --, the pullback map
\[
f^\ast: \mathcal D_{\mathrm{mot}}^{\mathrm{eff}}(X)\to \mathcal D_{\mathrm{mot}}^{\mathrm{eff}}(Y)
\]
admits a left adjoint
\[
f_\sharp: \mathcal D_{\mathrm{mot}}^{\mathrm{eff}}(Y)\to \mathcal D_{\mathrm{mot}}^{\mathrm{eff}}(X)
\]
satisfying base change and the projection formula. In particular, applied to the unit we get the ``free motivic sheaf on $Y$''
\[
\mathbb Z_{\mathrm{mot}}[Y]\in \mathcal D_{\mathrm{mot}}^{\mathrm{eff}}(X).
\]
Explicitly, it is obtained by starting with the free arc-sheaf on $Y$ and making it ball-invariant; this turns out to already yield a finitary sheaf.

In particular, we can define the first Tate twist
\[
\mathbb Z(1) := (\mathbb Z_{\mathrm{mot}}[\mathbb G_m]/\mathbb Z_{\mathrm{mot}}[\ast])[-1].
\]
Over $\mathbb C$, this is just isomorphic to $\mathbb Z$.

\begin{proposition} Over nonarchimedean $X$, let
\[
\overline{\mathbb G}_m = \mathbb G_m/(1+\mathcal O_{<1})
\]
where $\mathcal O_{<1}\subset \mathcal O$ is the sheaf of elements of norm $<1$. Then the natural map
\[
\mathbb Z_{\mathrm{mot}}[\mathbb G_m]/\mathbb Z_{\mathrm{mot}}[\ast]\to \overline{\mathbb G}_m
\]
is an isomorphism, and hence
\[
\mathbb Z(1)\cong \overline{\mathbb G}_m[-1].
\]
\end{proposition}

Again as usual in the motivic formalism, this is not yet invertible. Still, the $\Sigma_n$-action on $\mathbb Z(n) = \mathbb Z(1)^{\otimes_{\mathrm{mot}} n}$ is trivial. These higher weight motivic complexes can also be related to an appropriate version of algebraic $K$-theory.

\begin{theorem} Let $K$ denote the functor $A\mapsto K_{\mathrm{cn}}(A)$ sending any Banach ring $A$ to the connective $K$-theory spectrum $K_{\mathrm{cn}}(A)$ of the abstract ring $A$. Let $\overline{K}$ denote the universal ball-invariant arc-sheaf under $K$. Then
\begin{enumerate}
\item[{\rm (i)}] The arc-sheaf $\overline{K}$ is finitary.
\item[{\rm (ii)}] There is a canonical isomorphism
\[
\overline{K}\otimes \mathbb Q\cong \bigoplus_{n\geq 0} \mathbb Q(n)[2n]
\]
identifying rational motivic cohomology with Adams-eigenspaces on rational $K$-theory.
\item[{\rm (iii)}] If $C$ is an algebraically closed nonarchimedean field of residue characteristic $p$, then
\[
\overline{K}(C)\cong \mathrm{cofib}(K(C_{<1})\to K(C))[\tfrac 1p]
\]
where $K(C_{<1})$ denotes the algebraic $K$-theory of the non-unital, but Tor-unital, ring $C_{<1}\subset C$.\footnote{Recall that a non-unital ring $A$ is Tor-unital if its unitalization $A^+ = A\oplus\mathbb Z$ satisfies $\mathbb Z\dotimes_{A^+} \mathbb Z = \mathbb Z$. For Tor-unital rings, one defines $K(A)$ as the fibre of $K(A^+)\to K(\mathbb Z)$, and it is a theorem of Suslin \cite{SuslinExcision} that this agrees with the fibre of $K(A')\to K(A'/A)$ for any unital ring $A'$ containing $A$ as an ideal.}
\item[{\rm (iv)}] More generally, if $A$ is a Banach-$C$-algebra such that $A_{\leq 1}$ is the completion of a smooth $\mathcal O_C$-algebra, then
\[
\overline{K}(A)\otimes \mathbb Q\cong \mathrm{cofib}(K(A_{<1})\to K(A))\otimes \mathbb Q.
\]
\end{enumerate}
\end{theorem}

\begin{remark} Part (iii) has the curious consequence that the cofiber of $K(C_{<1})\to K(C)$ has a natural $K(C)$-$E_\infty$-algebra structure. The categorical reasons for this are not clear.
\end{remark}

Our proof of (ii) is very simple, and proves the similar result also in the classical motivic context; we are not aware that this argument is recorded in the literature. The key point is to show that the natural map
\[
\mathbb S[\ast/\mathbb G_m]\to K_{\mathrm{cn}}
\]
coming from line bundles becomes an equivalence after passing to ball-invariant arc-sheaves, and rationalizing. Now $K_{\mathrm{cn}}$ is the group completion of $\bigsqcup_{n\geq 0} \ast/\mathrm{GL}_n$, while the left-hand side is the group completion of
\[
\bigsqcup_{n\geq 0}\ast/(\mathbb G_m^n\rtimes \Sigma_n),
\]
where each $\mathbb G_m^n\rtimes \Sigma_n\subset \mathrm{GL}_n$ can be identified with the normalizer of the diagonal torus. Using Breen--Deligne resolutions, it is then sufficient to show that the natural map
\[
\mathbb Q_{\mathrm{mot}}[\ast/(\mathbb G_m^n\rtimes \Sigma_n)]\to \mathbb Q_{\mathrm{mot}}[\ast/\mathrm{GL}_n]
\]
is an isomorphism. But both sides can be understood explicitly; everything only involves spaces of simple motivic nature such as flag varieties.

The proofs of (iii) and (iv) make use of known descent and homotopy-invariance results for $K$-theory, which enable one to control the process of making sheaves ball-invariant that is otherwise very inexplicit. Specifically, we use $\mathbb A^1$-invariance over valuation rings, and (pro-)cdh-descent.

Using part (iv) of the theorem above, one can prove the cancellation theorem for Tate twists directly. But in fact, the full form of the cancellation theorem holds:

\begin{theorem} For any small arc-stack $X$, the functor
\[
-\otimes_{\mathrm{mot}} \mathbb Z(1): \mathcal D_{\mathrm{mot}}^{\mathrm{eff}}(X)\to \mathcal D_{\mathrm{mot}}^{\mathrm{eff}}(X)
\]
is fully faithful.
\end{theorem}

The proof of the cancellation theorem is the most difficult part of this paper. While all the other arguments use only soft arguments that are very natural within this analytic framework, the proof of the cancellation theorem uses some of the techniques introduced by Voevodsky such as a version of presheaves with transfer. In outline, we follow the argument of \cite{VoevodskyCancellation}.

\begin{remark} In our setting, the result for any $X$ reduces formally to strictly totally disconnected spaces and then even to geometric points. This is in contrast to other motivic settings, either algebraic or rigid-analytic, where the cancellation theorem is only true over fields.
\end{remark}

We can now define the full category of motivic sheaves by formally inverting the Tate twist.

\begin{definition} Let $X$ be a small arc-stack. The presentable stable $\infty$-category of motivic sheaves on $X$ is
\[
\mathcal D_{\mathrm{mot}}(X) = \mathcal D_{\mathrm{mot}}^{\mathrm{eff}}(X)[\mathbb Z(1)^{\otimes -1}] = \varprojlim_{\intHom(\mathbb Z(1),-)} \mathcal D_{\mathrm{mot}}^{\mathrm{eff}}(X).
\]
\end{definition}

We prove that this defines a $6$-functor formalism for which smooth maps are cohomologically smooth. In fact, the key point for both of these assertions is to prove that for $f: \mathbb P^1_X\to X$, the left and right adjoints of $f^\ast$ agree up to twist. For this, one constructs unit and counit of the adjunction explicitly.

The resulting category $\mathcal D_{\mathrm{mot}}(X)$ is very well-behaved.

\begin{theorem} If $X$ is a geometric point, then $\mathcal D_{\mathrm{mot}}(X)$ is compactly generated and the compact objects agree with the dualizable objects. If $X$ has equal characteristic, $\mathcal D_{\mathrm{mot}}(X)$ is generated by the motives of smooth projective varieties with good reduction.

If $X=\mathcal M_{\mathrm{arc}}(A)$ for an analytic Banach ring $A$, and $X$ has finite cohomological dimension, then $\mathcal D_{\mathrm{mot}}(X)$ is rigid.
\end{theorem}

Here, we use the notion of rigidity of Gaitsgory--Rozenblyum \cite[Chapter I, Definition 9.1.2]{GaitsgoryRozenblyum}; it is an adaptation to the non-compactly generated case of the idea that the compact and dualizable objects agree. Again, this rigidity in the relative case is something that does not usually hold true, not even for \'etale sheaves on schemes or rigid-analytic varieties. Namely, in those settings the category is compactly generated, but not all compact objects are dualizable (as there are constructible sheaves that are not locally constant). The Berkovich setting yields categories closer to the category of sheaves on the compact Hausdorff space $\mathcal M(A)$, and the latter is rigid.

Finally, we discuss the relation to the algebraic theory. Let $k$ be a discrete algebraically closed field. We get the small arc-sheaf $\mathcal M_{\mathrm{arc}}(k)$, and hence $\mathcal D_{\mathrm{mot}}(k)$.

\begin{proposition} The symmetric monoidal presentable stable $\infty$-category $\mathcal D_{\mathrm{mot}}(k)$ agrees with the one of \'etale motives over $k$ defined by Voevodsky.
\end{proposition}

For the proof, the key point is to compute maps in our version, for which we use the description of (rational) Tate twists in terms of algebraic $K$-theory. As maps in Voevodsky's category are also (rationally) controlled by $K$-theory, this yields the comparison.

Finally, in our framework it is easy to relate motives over an algebraically closed nonarchimedean field $C$ with motives over its residue field $k$, recovering results of Binda--Gallauer--Vezzani \cite{BindaGallauerVezzani}. We focus on the simplest case $k=\overline{\mathbb F}_p$ and $C$ being the completed algebraic closure of $k((T))$.

\begin{theorem} Fix a pseudouniformizer $\pi\in C$ together with compatible roots $\pi^{1/n}$ for all $n$ prime to the residue characteristic of $C$. There is a symmetric monoidal ``nearby cycles functor''
\[
\Psi = \Psi_{(\pi^{1/n})_n}: \mathcal D_{\mathrm{mot}}(C)\to \mathcal D_{\mathrm{mot}}(k)
\]
which upgrades to a symmetric monoidal equivalence
\[
\mathcal D_{\mathrm{mot}}(C)\cong \{(A,N)\mid A\in \mathcal D_{\mathrm{mot}}(k), N: A\to A\otimes \mathbb Q(-1)\ \mathrm{locally}\ \mathrm{nilpotent}\}.
\]
\end{theorem}

{\bf Acknowledgments.} The idea of using rigid-analytic motives to study independence of $\ell$ of $L$-parameters was found while giving the lecture course on $6$-functor formalisms in the winter term 2022/23. This was of course inspired by the work of Richarz and Scholbach on motivic geometric Satake, and their work in the direction of a motivic version of the results of V.~Lafforgue. I want to thank Johannes Ansch\"utz, Joseph Ayoub, Dustin Clausen, David Hansen, Hyungseop Kim and Dmitry Kubrak for very helpful feedback on a draft, and Mura Yakerson for some discussions on the cancellation theorem. The results of this paper were the basis for an ARGOS seminar in the winter term 2024/25, and I heartily thank all the speakers and participants for their feedback. Moreover, I am very grateful to the referee for a very careful reading, and a large number of suggestions for improvements.

\section{Banach rings and Berkovich spectrum}

We begin by recalling some basic definitions and facts about Banach rings and Berkovich spectra that we will use. We follow closely Berkovich's work \cite[Chapter 1]{Berkovich}.

\begin{definition} A seminormed ring is a commutative unital ring $R$ equipped with a map
\[
|\cdot|_R: R\to \mathbb R_{\geq 0}
\]
satisfying the following properties.
\begin{enumerate}
\item[{\rm (i)}] One has $|0|_R=0$ and $|\text-1|_R\leq 1$.
\item[{\rm (ii)}] For all $x,y\in R$, one has $|xy|_R\leq |x|_R |y|_R$.
\item[{\rm (iii)}] For all $x,y\in R$, one has $|x+y|_R\leq |x|_R+|y|_R$.
\end{enumerate}

A map of seminormed rings $(R,|\cdot|_R)\to (S,|\cdot|_S)$ is a map of commutative unital rings $f: R\to S$ such that $|f(x)|_S\leq |x|_R$ for all $x\in R$.
\end{definition}

Any seminorm on $R$ induces a ring topology on $R$, and we will often implicitly endow $R$ with this topology.

\begin{definition} A Banach ring is a seminormed ring that is complete with respect to the induced topology, i.e.~any Cauchy sequence has a unique limit.
\end{definition}

In particular, in a Banach ring $(R,|\cdot|_R)$, the seminorm $|\cdot|_R$ is actually a norm, i.e.~if $|x|_R=0$ then $x=0$. As usual, the inclusion of Banach rings into seminormed rings has a left adjoint, the completion.

\begin{remark} If one applies (ii) to $y=1$ one gets $|x|_R\leq |1|_R |x|_R$, which shows that if $|1|_R<1$ then in fact $|x|_R=0$ for all $x\in R$, and so the completion of $R$ is zero. However, one cannot ask $|1|_R=1$ as this would exclude $R=0$ from being a Banach ring.
\end{remark}

\begin{proposition} The category of seminormed rings has all colimits. In particular, the category of Banach rings has all colimits.
\end{proposition}

\begin{proof} The case of Banach rings follows from the case of seminormed rings by taking colimits in seminormed rings and then completing.

For the case of seminormed rings, it suffices to handle the initial object, and the case of filtered colimits and of pushouts. The initial object is given by $\mathbb Z$ with its standard norm $|n|_{\mathbb Z}=\pm n$. (Here we use the condition $|\text-1|_R\leq 1$ on a seminorm, which implies $|1|_R\leq 1$ and then the bound follows from the triangle inequality.) Filtered colimits $\mathrm{colim}_i R_i$ are computed naively on the underlying rings, and the norm of an element $r\in \mathrm{colim}_i R_i$ that is the image of some $r_{i_0}\in R_{i_0}$ is given as the infimum of $|r_i|_{R_i}$, where $r_i$ is the image of $r_{i_0}$ in $R_i$ for $i\geq i_0$. For pushouts, take a diagram $B\leftarrow A\to C$ of seminormed rings. On $D=B\otimes_A C$, we can define the projective norm: Namely, we set
\[
|d|_D = \mathrm{inf}\{\sum_i |b_i|_B|c_i|_C : d = \sum_i b_i\otimes c_i\},
\]
the infimum of $\sum_i |b_i|_B|c_i|_C$ over all possible presentations of $d$ as a tensor $\sum_i b_i\otimes c_i$. It is easy to see that $(D,|\cdot|_D)$ defines a seminormed ring, and in fact the pushout in the category of seminormed rings.
\end{proof}

\begin{definition} A seminormed ring $(R,|\cdot|_R)$ is non-archimedean if the ultrametric triangle inequality
\[
|x+y|_R\leq \mathrm{max}(|x|_R,|y|_R)
\]
holds true for all $x,y\in R$.
\end{definition}

\begin{proposition} If $(R,|\cdot|_R)$ is power-multiplicative (i.e., $|x^n|_R=|x|^n_R$), then $R$ is nonarchimedean if and only if $|2|_R\leq 1$.
\end{proposition}

\begin{proof} One direction is clear, so assume $|2|_R\leq 1$. First, we prove $|n|_R\leq 1$ for all integers $n$. Write $n$ in binary expansion and use $|2^k|_R\leq 1$ for all $k\geq 0$ to see that $|n|_R$ is bounded by $1+\mathrm{log}_2(n)$. Applying this result for $|n^i|_R= |n|_R^i$ and letting $i\to \infty$ shows that $|n|_R\leq 1$. Now write
\[
(x+y)^n = \sum_{i=0}^n \binom{n}{i} x^i y^{n-i}
\]
to get $|x+y|_R^n\leq (n+1)\mathrm{max}(|x|_R,|y|_R)^n$ for all $n\geq 1$, which in the limit $n\to \infty$ gives $|x+y|_R\leq \mathrm{max}(|x|_R,|y|_R)$.
\end{proof}

\subsection{Free seminormed rings} Given a seminormed ring $(R,|\cdot|_R)$ and real numbers $r_1,\ldots,r_n\geq 0$, there is a free seminormed $R$-algebra
\[
R[T_1,\ldots,T_n]_{r_1,\ldots,r_n}
\]
with new elements $T_1,\ldots,T_n$ of norm $|T_i|\leq r_i$. This is given by the polynomial algebra $R[T_1,\ldots,T_n]$ with norm
\[
|\sum_{i_1,\ldots,i_n} a_{i_1,\ldots,i_n} T_1^{i_1}\cdots T_n^{i_n}| = \sum_{i_1,\ldots,i_n} |a_{i_1,\ldots,i_n}| r_1^{i_1}\cdots r_n^{i_n},
\]
where $a_{i_1,\ldots,i_n}\in R$. Beware however that there is not a free seminormed $R$-algebra on a variable $T$: One has to bound the norm of $T$. Beware also that unlike in algebraic contexts, these free seminormed $R$-algebras are not compact objects in the category of seminormed rings, and in fact the category of seminormed rings is not compactly generated. (There are still some compact objects, for example $\mathbb F_p$ with the quotient norm from $\mathbb Z$.)

Still, the category of seminormed rings has enough compact morphisms, i.e.~it is compactly assembled. Let us recall the relevant background in its natural generality of $\infty$-categories.

\begin{definition} Let $C$ be a presentable $\infty$-category.
\begin{enumerate}
\item[{\rm (i)}] A morphism $f: X\to Y$ is weakly compact if for any filtered colimit $Z=\mathrm{colim}_i Z_i$ with a map $Y\to Z$, there is some $i$ and a map $X\to Z_i$ making
\[\xymatrix{
X\ar[r]\ar[d]^f & Z_i\ar[d]\\
Y\ar[r] & Z
}\]
commute.
\item[{\rm (ii)}] The $\infty$-category $C$ is compactly assembled if filtered colimits commute with finite limits in $C$, and it is generated under colimits by objects of the form
\[
\mathrm{colim}(X_0\to X_1\to X_2\to \ldots)
\]
where all maps $X_i\to X_{i+1}$ are weakly compact.
\end{enumerate}
\end{definition}

\begin{remark} A more abstract characterization is that $C$ is compactly assembled if and only if the colimit functor
\[
\mathrm{Ind}(C)\to C
\]
has a left adjoint, see \cite[Theorem 2.39]{Ramzi}. The terminology ``weakly compact'' follows \cite[Definition 2.1]{Ramzi}. As it is the only notion of compactness of morphisms we will use, we will occasionally drop the ``weakly''. As in a compactly assembled $\infty$-category, the different notions of compact morphisms notions agree (cf.~\cite[Example 2.20]{Ramzi}), we hope this will not lead to confusion.
\end{remark}

\begin{proposition} The category of seminormed rings is compactly assembled.
\end{proposition}

We warn the reader that this assertion does not pass to the subcategory of Banach rings. Still, the argument below, of writing a disc as a filtered colimit of slightly larger discs (``overconvergence''), will be a recurring theme.

\begin{proof} Let $\mathbb Z[T]_r$ be the free seminormed ring with an element $T$ with $|T|\leq r$. Then for $r'>r$, the map $\mathbb Z[T]_{r'}\to \mathbb Z[T]_r$ is compact. This means that for any filtered colimit of seminormed rings $R=\mathrm{colim}_i R_i$ and an element $T\in R$ with $|T|_R\leq r$, there is some $i$ such that $T$ lifts to $R_i$ with $|T|_{R_i}\leq r'$. This follows from $|T|_R$ being the infimum of the $|T|_{R_i}$.

It follows that $\mathbb Z[T]_r = \mathrm{colim}_{r'>r} \mathbb Z[T]_{r'}$ is a presentation as a sequential colimit along weakly compact transition maps. The whole category is generated by this example (for varying $r\in \mathbb R_{>0}$) under colimits.
\end{proof}

\subsection{Banach fields} A special class of Banach rings are the Banach fields.

\begin{definition} A Banach field is a Banach ring $(K,|\cdot|_K)$ such that $K$ is a field and $|xy|_K=|x|_K|y|_K$ for all $x,y\in K$.
\end{definition}

\begin{remark} In particular, $0\neq 1$ in $K$ and hence $|1|_K=1$.
\end{remark}

There are three flavours of Banach fields: The discrete Banach fields in which $K$ has the discrete topology and $|\cdot|_K$ takes values in $\{0,1\}$, and the nondiscrete fields in which there exists some $\pi\in K$ with $0<|\pi|_K<1$, which can further be divided into the archimedean and nonarchimedean ones. In the following, by a nonarchimedean Banach field we also mean that in particular it is not discrete.

\begin{theorem}[Gelfand--Mazur] Any Banach field $K$ with $|2|_K>1$ is isomorphic to $\mathbb R$ or $\mathbb C$, endowed with some power $|\cdot|^\alpha$ of the usual norm, with $0<\alpha\leq 1$.
\end{theorem}

Now let $(R,|\cdot|_R)$ be a general seminormed ring.

\begin{definition} The Berkovich spectrum $\mathcal M(R)$ is the closed subspace
\[
\mathcal M(R)\subset \prod_{r\in R} [0,|r|_R]
\]
of all maps $||\cdot||: R\to \mathbb R_{\geq 0}$ satisfying the following properties:
\begin{enumerate}
\item[{\rm (i)}] One has $||0||=0$ and $||1||=1$.
\item[{\rm (ii)}] For all $x,y\in R$, one has $||xy||=||x||||y||$.
\item[{\rm (iii)}] For all $x,y\in R$, one has $||x+y||\leq ||x||+||y||$.
\end{enumerate}
\end{definition}

By its definition, $\mathcal M(R)$ is a compact Hausdorff space. It is easy to see that $\mathcal M(R)$ is unchanged under replacing $R$ by its completion. For any $s\in \mathcal M(R)$, one can endow $R$ with a new seminormed ring structure by using $||\cdot||=||\cdot||_s$; we let $R(s)$ be the corresponding completion of $R$. Then $R(s)$ is a Banach ring with multiplicative norm. The norm then extends to its fraction field, and we let $K(s)$ be the corresponding completion: This is a Banach field. This construction shows that for any $s\in \mathcal M(R)$ corresponding to $||\cdot||: R\to \mathbb R_{\geq 0}$, the norm can be written as the composite of a map $R\to K(s)$ to a Banach field, and the norm on the Banach field $K(s)$. The map $R\to K(s)$ has the property that the fraction field of the image is dense in $K(s)$, and this gives an equivalent description of the points of $\mathcal M(R)$ as maps $f: R\to K$ to Banach fields (in the category of Banach rings) such that the fraction field of the image is dense.

The first nontrivial theorem is the following.

\begin{theorem}\label{thm:berkovichnonempty} Let $R$ be a seminormed ring with $|1|_R=1$. Then $\mathcal M(R)\neq \emptyset$.
\end{theorem}

\begin{proof} We are free to replace $R$ by its completion. We can also replace $R$ by the quotient by any proper closed ideal. Note that any maximal ideal $\mathfrak m$ of $R$ is already closed. Otherwise, its closure is all of $R$, and in particular contains $1$. Thus, there exists $x\in \mathfrak m$ with $|1-x|_R<1$. Thus, $x=1-y$ with $|y|_R<1$. But then $1+y+y^2+\ldots$ converges in $R$ and defines $(1-y)^{-1}$. Thus, $x\in \mathfrak m$ is invertible, which is a contradiction.

In other words, we can assume that $R=K$ is a field. If we consider all possible seminorms on $K$ bounded by the given one and with $|1|_K=1$, then this category forms a partially ordered set (where maps are morphisms of $K$-Banach rings) and has all filtered colimits; we can thus pick a maximal element. We can also pass to the completion. Our task now is to prove that $K$ is a Banach field. First note that $x\mapsto \mathrm{lim}_{n\to \infty} |x^n|_K^{1/n}$ defines another norm on $K$ which is smaller; by choice of $K$, we thus have $|x|_K^n=|x^n|_K$ for $n>0$. Now for any nonzero $x\in K$ we can define a new norm by sending $y$ to the infimum over all $n>0$ of $|x|_K^{-n} |x^ny|_K$. This still sends $1\in K$ to $1\in \mathbb R_{\geq 0}$ as $|x|_K^n=|x^n|_K$. Again, the choice of $K$ implies that we must thus have $|x|_K^{-n} |x^ny|_K = |y|_K$, and in particular $|xy|_K=|x|_K|y|_K$.
\end{proof}

\begin{corollary} Let $R$ be a Banach ring and $f\in R$. Then $f$ is invertible if and only if $||f||: \mathcal M(R)\to \mathbb R_{\geq 0}$ is everywhere strictly positive.
\end{corollary}

\begin{remark} This yields Gelfand's famous proof of Wiener's theorem.
\end{remark}

\begin{proof} If $f$ is not invertible, then the closure of the ideal generated by $f$ is not all of $R$ (as maximal ideals are closed, as established in the previous proof), and hence the quotient $S$ is a nonzero Banach ring where $f$ maps to $0$. Any point in the image of $\mathcal M(S)\to \mathcal M(R)$ is then a point where $||f||$ is zero.
\end{proof}

The following notion was introduced by Kedlaya--Liu \cite[Remark 2.3.9 (d)]{KedlayaLiu1} under the name ``free of trivial spectrum''; later, Kedlaya \cite[Definition 1.1.2]{KedlayaAWS} introduced the term ``analytic''. While this clashes with the notion of ``analytic ring'' in condensed mathematics, in the present paper it does not lead to problems.

\begin{definition} A Banach ring $A$ is analytic if for all $x\in \mathcal M(A)$, the Banach field $K(x)$ is non-discrete.
\end{definition}

It is not possible to define a sheaf of Banach rings on $\mathcal M(R)$ localizing $R$. Using condensed mathematics (and in particular the discussion in \cite[Lecture 5, Corollary 5.10]{Complex} generalized from $\mathbb C$-algebras to general analytic Banach algebras), one can still produce a sheaf of animated gaseous $R$-algebras, at least when $R$ is analytic (which can always be ensured by passing to $R[T^{\pm 1}]_{1/2}$). We will only use the following seemingly elementary fact, which can be proved directly.

\begin{theorem}\label{thm:idempotents} Let $R$ be a Banach ring. Then idempotent elements in $R$ are in bijection with open and closed subsets of $\mathcal M(R)$, via the map sending an idempotent $e\in R$ to the vanishing locus of $e$.
\end{theorem}

\begin{proof} The previous corollary implies injectivity of the map from idempotent elements to open and closed subsets. The high-tech proof of the converse is the following. We have to see that any open and closed decomposition $\mathcal M(R) = U_1\sqcup U_2$ yields an idempotent element in $R$. Passing to the completion of $R[T^{\pm 1}]_{1/2}$ we may assume that $R$ is analytic (the idempotent will necessarily only have a constant term, as it is invariant under rescaling the variable $T$). Then we have a sheaf $\mathcal O$ of animated gaseous $R$-algebras on $\mathcal M(R)$, and the open and closed decomposition yields a global section $e\in R$ that is $0$ on one part and $1$ on the other.

A direct proof goes as follows; this uses the notion of uniform completions and the Gelfand transform introduced just below. Assume $f\in R$ and $r>0$. We have maps $R\to R\langle F\rangle_r$ and $R\to R\langle f^{-1}\rangle_{r}$, whose Berkovich spaces are given by the subsets $\{|f|\leq r\}$ and $\{|f|\geq r\}$ of $\mathcal M(R)$. We claim that if one has idempotents $e_1\in R\langle f\rangle_r$ and $e_2\in R\langle f^{-1}\rangle_{r}$ that map to the same idempotent in $R\langle f^{\pm 1}\rangle_r$, then they come from a unique idempotent $e\in R$. Once one knows this, the result follows by a standard localization argument.

One first checks that idempotents lift uniquely from the uniform completion of $R$ back to $R$: Given any sufficiently good approximation, the Newton method will produce an idempotent (as the equation $x^2-x=0$ is finite \'etale). Thus, we can assume that $R$ is uniform. By a similar argument, we can also replace $R\langle f\rangle_r$ by the colimit of $R\langle f\rangle_{r'}$ over $r'>r$; we denote this by $R\langle f\rangle_r^\dagger$. Then the map
\[
R\to \mathrm{eq}(R\langle f\rangle_r^\dagger\times R\langle f^{-1}\rangle_{r}^\dagger\rightrightarrows R\langle f^{\pm 1}\rangle_r^\dagger)
\]
is actually an isomorphism, and in particular the gluing of idempotents follows. Injectivity follows from the Gelfand transform being injective, and surjectivity from an argument of Kedlaya--Liu \cite[Corollary 2.8.9]{KedlayaLiu1}. More precisely, the main point is that
\[
R\langle f\rangle_r^\dagger = R\langle U\rangle_r^\dagger / (U-f)
\]
where $U$ is a free variable, and the ideal generated by $U-f$ is closed. When $R$ is uniform, this assertion can be reduced to the case of Banach fields, where it is standard. The same applies to the other two overconvergent algebras, and then the result follows by a direct diagram chase.
\end{proof}

\subsection{Gelfand transform}

\begin{definition} The Gelfand transform is the map
\[
R\to \prod_{s\in \mathcal M(R)}^{\mathrm{Ban}} K(s)
\]
where the product is taken in the category of Banach rings; equivalently, it is the subring of the naive product $\prod_s K(s)$ of tuples $(x_s)_s$ of elements $x_s\in K(s)$ with $||x_s||_{K(s)}$ bounded independently of $s$, equipped with the supremum norm.
\end{definition}

\begin{definition}\label{def:uniform} A Banach ring $R$ is uniform if the Gelfand transform is an isometric embedding. Equivalently,
\[
|x|_R = \mathrm{sup}_{s\in \mathcal M(R)} ||x_s||_{K_s}
\]
for all $x\in R$.
\end{definition}

In general, one can define the spectral norm on $R$ as $\mathrm{sup}_{s\in \mathcal M(R)} ||x_s||_{K_s}$.

\begin{proposition} Let $R$ be a Banach ring. The following conditions are equivalent.
\begin{enumerate}
\item[{\rm (i)}] The Banach ring $R$ is uniform.
\item[{\rm (ii)}] The map $x\mapsto |x|_R$ is power-multiplicative, i.e.~$|x^n|_R = |x|^n_R$ for all $n\geq 1$.
\end{enumerate}
\end{proposition}

In general, the map
\[
x\mapsto \mathrm{lim}_{n\to \infty} |x^n|^{1/n}_R
\]
defines a power-multiplicative norm bounded by the given norm. It follows that it is equal to the spectral norm. Completion with respect to this norm defines the uniform completion.

\begin{proof} It is clear that (i) implies (ii) as the spectral norm is power-multiplicative. For the converse, take any $x\in R$ with $|x|_R\neq 0$. Consider the free seminormed ring $S$ under $R$ in which $x$ is invertible and with $|x^{-1}|_S\leq |x|_R^{-1}$. Concretely, this is $S=R[x^{-1}]$ with $|s|_S$ being the infimum of $|r|_R |x|_R^{-n}$ over all presentations $s=rx^{-n}$. In particular, by power-multiplicativity, one has $|1|_S = 1$. It follows that $\mathcal M(S)\neq \emptyset$ by Theorem~\ref{thm:berkovichnonempty}. Any map $S\to K$ to a Banach field yields a map $f: R\to S\to K$ such that $|x|_R = ||f(x)||_K$.
\end{proof}

\subsection{The nonarchimedean unit disc} We end this section by recalling the structure of a disc over an algebraically closed nonarchimedean Banach field $C$. Let $C\langle T\rangle_1$ be the usual Tate algebra of functions on the disc $|T|\leq 1$; this is the uniform completion of the seminormed ring $C[T]_1$. Note that all elements of $C[T]$ factor into an element of $C$ and products of elements of the form $T-a$ with $a\in C$, so any point is determined by the collection of norms $||T-a||$ with $a\in C$. For $|a|>1$, one necessarily has $||T-a||=|a|$ as $|T|\leq 1$, so can restrict to elements $T-a$ with $a\in \mathcal O_C$, the ring of integers of $C$. Thus
\[
\mathcal M(C\langle T\rangle_1)\subset \Big\{(x_a)_{a\in \mathcal O_C}\mid \forall a,b: x_a\leq \mathrm{max}(x_b,|a-b|_C), |a-b|_C\leq \mathrm{max}(x_a,x_b)\Big\}\subset \prod_{|a|\leq 1} [0,1].
\]
If we fix a finite set $A\subset \mathcal O_C$, we claim that
\[
\{(x_a)_{a\in A}\mid \forall a,b: x_a\leq \mathrm{max}(x_b,|a-b|_C), |a-b|_C\leq \mathrm{max}(x_a,x_b)\}\subset \prod_{a\in A} [0,1]
\]
is a tree, and that $\mathcal M(C\langle T\rangle_1)$ surjects onto it. Given any point $x=(x_a)_{a\in A}$, we can pick $a\in A$ for which $x_a$ is minimal. Then $x_b = \mathrm{max}(x_a,|a-b|_C)$ for all other $b\in A$. We see that the full subspace
\[
\Big\{(x_a)_{a\in A}\mid \forall a,b: x_a\leq \mathrm{max}(x_b,|a-b|_C), |a-b|_C\leq \mathrm{max}(x_a,x_b)\Big\}\subset \prod_{a\in A} [0,1]
\]
is the union of maps
\[
f_b: [0,1]\to \prod_{a\in A} [0,1]: x\mapsto (\mathrm{max}(x,|a-b|_C))_a.
\]
These functions have the property that if for some $r\in [0,1]$ one has $f_{b_1}(r)=f_{b_2}(r)$, then $f_{b_1}(s)=f_{b_2}(s)$ for all $s\geq r$. This property implies that their union is a tree.

Moreover, we see that $\mathcal M(C\langle T\rangle_1)$ has dense (hence full) image, as for any $b\in A$ and any generic element $t\in C$ with $|t-b|=r$ one can define a point of $\mathcal M(C\langle T\rangle_1)$ (from the map $C\langle T\rangle_1\to C$ sending $T$ to $t$) that will be sent to $f_b(r)$. Indeed, this happens as soon as for the finitely many $a\in A$, one has $|t-a|_C=\mathrm{max}(r,|a-b|_C)$, which happens for generic $t$ in the disc $|t-b|=r$. Passing to the inverse limit and using compactness, we see that
\[
\mathcal M(C\langle T\rangle_1)= \Big\{(x_a)_{a\in \mathcal O_C}\mid \forall a,b: x_a\leq \mathrm{max}(x_b,|a-b|_C), |a-b|_C\leq \mathrm{max}(x_a,x_b)\Big\}\subset \prod_{|a|\leq 1} [0,1],
\]
and is an inverse limit of finite trees.

One can also describe explicitly the residue fields $K(x)$, as points of types (1), (2), (3) and (4).
\begin{enumerate}
\item[(1)] If there is some $a$ with $x_a=0$, then $K(x)=C$, sending $T$ to $a$.
\item[(2)] If there is some $a$ for which $x_a$ is minimal and nonzero and $x_a=|t|$ for some $t\in C$, then $K(x)$ is the completion of $C(T)$ with respect to $x$, which is a Banach field with the same value group as $C$, and residue field a pure transcendental extension of the residue field of $C$. Up to changing $T$ by $(T-a)/t$, this can be described by localizing $\mathcal O_C[T]$ at the generic point of the special fibre, $\pi$-adically completing, and inverting $\pi$, where $\pi\in C$ is a pseudouniformizer, i.e.~$0<|\pi|<1$.
\item[(3)] If there is some $a$ for which $x_a$ is minimal and nonzero and $x_a$ does not lie in $|C|$, then $K(x)$ is the completion of $C(T)$ with respect to $x$, which is a Banach field with the same residue field as the residue field of $C$, but with value group $|C^\times|\times x_a^{\mathbb Z}$. In fact, it can be explicitly described as $C\langle (T-a)^{pm 1}\rangle_{x_a}$, the uniform completion of the free seminormed ring $C[(T-a)^{\pm 1}]_{x_a}$ with $|T-a|=|(T-a)^{-1}|^{-1}=x_a$.
\item[(4)] If there is no $a$ for which $x_a$ is minimal, one can choose a sequence $a_1,a_2,\ldots$ such that $x_{a_i}$ approaches the (necessarily positive) infimum. In that case, $K(x)$ is the completion of $C(T)$ with respect to $x$, which is a Banach field with the same residue field as the residue field of $C$ and the same value group as $C$ (but distinct from $C$), i.e.~an immediate extension. This can be described explicitly as the completed filtered colimit of the free seminormed rings $C[T-a_i]_{x_{a_i}}$.
\end{enumerate}

Under appropriate conditions on $C$ ($|C|=\mathbb R_{\geq 0}$ respectively $C$ is spherically complete), points of type (3) respectively (4) do not occur. In particular, for spherically complete fields with $|C|=\mathbb R_{\geq 0}$ (huge fields indeed), only types (1) and (2) appear; such fields are occasionally useful.

\section{The arc-topology}

We will be working with an extremely fine topology on (the opposite of) the category of Banach rings.

\begin{definition}\label{def:arccover} Consider a family $\{A\to B_i\}_{i\in I}$ of maps of Banach rings with common source. The family is an arc-cover if there is a finite subset $J\subset I$ such that the map
\[
\bigsqcup_{i\in J} \mathcal M(B_i)\to \mathcal M(A)
\]
is surjective. Equivalently, for any Banach field $K$ with a map $A\to K$, there is an extension of Banach fields $K\subset L$ and a lift of $A\to K$ to a map $B_i\to L$, for some $i\in J$.
\end{definition}

In order to see that this is well-behaved (for example, pullbacks of arc-covers are arc-covers), we recall the following proposition on the behaviour of the Berkovich spectrum under colimits.

\begin{proposition}\label{prop:berkovichspectrumcolimit}\leavevmode\begin{enumerate}
\item[{\rm (i)}] Let $B\leftarrow A\to C$ be a pushout diagram of seminormed rings, with pushout $D=B\otimes_A C$. The induced map
\[
\mathcal M(D)\to \mathcal M(B)\times_{\mathcal M(A)} \mathcal M(C)
\]
of compact Hausdorff spaces is surjective.
\item[{\rm (ii)}] Let $(A_i)_i$ be a filtered diagram of seminormed rings with colimit $A=\mathrm{colim}_i A_i$. The induced map
\[
\mathcal M(A)\to \mathrm{lim}_i \mathcal M(A_i)
\]
of compact Hausdorff spaces is an isomorphism.
\end{enumerate}
\end{proposition}

In both assertions, one could replace seminormed rings with Banach rings, as the Berkovich spectrum only depends on the Banach completion.

\begin{proof} In part (i), by interpreting points of the Berkovich spectrum in terms of maps to Banach fields, we can assume that $B\leftarrow A\to C$ is a diagram of Banach fields. In this case, we only need to prove that $|1|_D=1$. Assume first that $A$ is not discrete. If $A$ is archimedean, then all three rings are either $\mathbb R$ or $\mathbb C$ and the claim is clear. If $A$ is nonarchimedean, then so are $B$ and $C$, and we have the diagram $\mathcal O_B\leftarrow \mathcal O_A\to \mathcal O_C$ of their rings of integers, which are faithfully flat. Then we can endow
\[
D = (\mathcal O_B\otimes_{\mathcal O_A} \mathcal O_C)[\tfrac 1\pi]
\]
(where $\pi\in \mathcal O_A$ has $0<|\pi|<1$) with the norm
\[
x\mapsto \mathrm{inf} \{|\pi|^{-m/n}\mid x^n \pi^m\in \mathcal O_B\otimes_{\mathcal O_A} \mathcal O_C\}.
\]
This is an ultrametric uniform norm compatible with the ones on $\mathcal O_B$ and $\mathcal O_C$, and sends $1$ to $1$. Thus, also $|1|_D=1$.

Finally, assume that $A$ is discrete. Then in particular $|2|_A\leq 1$ and so $B$ and $C$ are discrete or non-archimedean. If both are discrete, the claim is clear. Otherwise, without loss of generality the map $A\to B$ factors over $A((T))_r$ where one adjoins a free Laurent series variable $T$ with norm $|T|=|T^{-1}|^{-1}=r$, for some $0<r<1$; we can then reduce to the case $B=A((T))_r$. But then $B$ is the Banach completion of a free $A$-module with basis $T^n$, $n\in \mathbb Z$, and then so is $D$ over $C$, which shows in particular that $|1|_D=1$.

Part (ii) is obvious from the definition (using that any bijective map of compact Hausdorff spaces is an isomorphism).
\end{proof}

\begin{proposition} The opposite of the category of Banach rings, endowed with the notion of arc-covers, forms a site where all objects are quasicompact and quasiseparated.
\end{proposition}

To be precise, one should restrict to $\kappa$-small Banach rings for some uncountable regular cardinal $\kappa$. Increasing $\kappa$ yields a geometric morphism of topoi, and the corresponding pullback functors will be fully faithful and agree with left Kan extension, identifying the essential image with those arc-sheaves commuting with $\kappa$-filtered colimits. In the similar situation of the fpqc topology on the category of commutative rings, this has been worked out by Waterhouse \cite{Waterhouse}, and the arguments adapt easily. In fact, as the property of being an arc-cover is much easier to guarantee than that of being an fpqc-cover, it is even easier.

\begin{proof} Clear.
\end{proof}

For any Banach ring $A$, we denote by $\mathcal{M}_{\mathrm{arc}}(A)$ the associated arc-sheaf, which is the arc-sheafification of the presheaf $B\mapsto \mathrm{Hom}(A,B)$.

\begin{example} Any nonzero Banach ring $A$ admits an arc-cover by the completion of the free seminormed ring $A[T^{\pm 1}]_{1/2}$ where one adjoins a new variable $T$ with $|T|=|T^{-1}|^{-1}=\frac 12$. In particular, locally in the arc-topology we can assume that $A$ admits an element $T\in A$ with $|T|=|T^{-1}|^{-1}=\tfrac 12$, and in particular that $A$ is analytic.
\end{example}

\begin{example} If $A$ is any Banach ring and $B$ is its uniform completion, then the map $A\to B$ is an arc-cover. It follows that $\mathcal{M}_{\mathrm{arc}}(A)=\mathcal M_{\mathrm{arc}}(B)$. Also, the topos is unchanged if we restrict to complete uniform Banach rings (admitting an element $T$ with $|T|=|T^{-1}|^{-1}=\tfrac 12$ if nonzero).
\end{example}

\begin{example} Let $A$ be any Banach ring. The map
\[
A\to B=\prod_{x\in \mathcal M(A)}^{\mathrm{Ban}} K(x)
\]
is an arc-cover. The target is an example of a totally disconnected Banach ring as defined below. As we will see, $\mathcal M(B)$ agrees with the Stone-\v{C}ech compactification of the underlying discrete set of $\mathcal M(A)$.
\end{example}

\begin{example} Let $A$ be a Banach $\mathbb F_p$-algebra. For $0<t\leq 1$, let $A^{(t)}$ denote the Banach ring $A$ with the norm raised to the $t$-th power. Then the map
\[
A\to A_{\mathrm{perf}} = \mathrm{colim}^{\mathrm{Ban}}(A\xrightarrow{\mathrm{Frob}} A^{(p^{-1})}\xrightarrow{\mathrm{Frob}} A^{(p^{-2})}\to\ldots)
\]
from $A$ to its completed perfection is an arc-cover and the uniform completion of $A_{\mathrm{perf}}\otimes_A A_{\mathrm{perf}}$ is $A_{\mathrm{perf}}$. Hence, we get an isomorphism $\mathcal M_{\mathrm{arc}}(A)\cong \mathcal M_{\mathrm{arc}}(A_{\mathrm{perf}})$. In other words, sheafification in the arc-topology includes perfection when working in characteristic $p$.
\end{example}

\begin{example} Let $A$ be any Banach ring over $\mathbb C$. Then
\[
A\to B=\mathrm{Cont}(\mathcal M(A),\mathbb C)
\]
is an arc-cover, and the uniform completion of $B\otimes_A B$ is equal to $B$. Indeed, like for any Banach-$\mathbb C$-algebra, this uniform completion embeds into $\mathrm{Cont}(\mathcal M(B\otimes_A B),\mathbb C)$, but here $\mathcal M(B\otimes_A B)=\mathcal M(B)\times_{\mathcal M(A)} \mathcal M(B) = \mathcal M(A)$, and so it must agree with $B=\mathrm{Cont}(\mathcal M(A),\mathbb C)$. It follows that $\mathcal M_{\mathrm{arc}}(A)\cong \mathcal M_{\mathrm{arc}}(B)$. As an example, arc-sheafification identifies the Banach ring of holomorphic functions on the closed unit disc with the Banach ring of continuous functions thereon.
\end{example}

\begin{proposition} The arc-site on Banach $\mathbb C$-algebras is Morita equivalent to the site of compact Hausdorff spaces, with finite collections of jointly surjective maps as covers.
\end{proposition}

By Morita equivalence, we mean that their associated categories of sheaves are equivalent. In the present case, this holds even for (non-hypercomplete) sheaves of anima.

If one works over $\mathbb R$, one gets instead the site of compact Hausdorff spaces equipped with an involution.

\begin{proof} This is immediate from the previous example.
\end{proof}

Of course, the site of the previous proposition is the one defining condensed sets. A standard technique in that setting is to take covers by profinite sets. The analogue of profinite sets in our setting is the following.

\begin{definition}\label{def:totallydisconnected} Let $A$ be a Banach ring.
\begin{enumerate}
\item[{\rm (i)}] The ring $A$ is totally disconnected if it is analytic, uniform, and $\mathcal M(A)$ is a profinite set.
\item[{\rm (ii)}] The ring $A$ is strictly totally disconnected if it is totally disconnected and for all $x\in \mathcal M(A)$, the Banach field $K(x)$ is algebraically closed.
\end{enumerate}
\end{definition}

We note that if $A$ is totally disconnected, then by Theorem~\ref{thm:idempotents}, any disconnection of $\mathcal M(A)$ yields a decomposition $A=A_1\times A_2$ of Banach rings. This is necessarily an isometry, as $A$ is uniform. This implies that for any $x\in \mathcal M(A)$, the map $A\to K(x)$ is surjective, by writing it as a completed filtered colimit of surjections (corresponding to clopen neighborhoods of $x$).

\begin{proposition} Let $(K_i)_i$ be a collection of non-discrete Banach fields and let
\[
A=\prod_i^{\mathrm{Ban}} K_i.
\]
\begin{enumerate}
\item[{\rm (i)}] The ring $A$ is totally disconnected.
\item[{\rm (ii)}] The natural map $I=\bigsqcup_i \mathcal M(K_i)\to \mathcal M(A)$ induces an homeomorphism
\[
\beta(I)\to \mathcal M(A)
\]
of compact Hausdorff spaces, where $\beta(I)$ is the Stone-\v{C}ech compactification of the discrete set $I$ (i.e., the set of ultrafilters on $I$).
\item[{\rm (iii)}] For any $x\in \mathcal M(A)=\beta(I)$, the Banach field $K(x)$ is a corresponding Banach ultraproduct of the Banach fields $K_i$.
\item[{\rm (iv)}] If all $K_i$ are algebraically closed, then $A$ is strictly totally disconnected.
\end{enumerate}
\end{proposition}

\begin{proof} Part (i) reduces to (ii) and (iii). For (ii), we recall that $\beta(I)$ is the inverse limit $\mathrm{lim}_{I\to \overline{I}} \overline{I}$ over all finite sets $\overline{I}$ to which $I$ maps, and certainly $\mathcal M(A)$ maps there (as any finite disjoint decomposition of $I$ decomposes $A$ into a finite product, and this is taken by $\mathcal M$ to a finite disjoint union). It remains to see that the fibres of this map $\mathcal M(A)\to \beta(I)$ are points, with residue field as in (iii). But those fibres are given by the Banach rings
\[
\mathrm{colim}_{J\subset I, \beta(J)\ni x}^{\mathrm{Ban}} \prod_{i\in J}^{\mathrm{Ban}} K_i,
\]
which is exactly the Banach ultraproduct. This is indeed a Banach field itself: Given any $x$ in the Banach ultraproduct, it lifts to a collection $(x_i)_i$ such that $|x|$ is the ultralimit of $|x_i|_{K_i}$. In particular, when nonzero, then $|x_i|_{K_i}$ is nonzero and bounded away from zero for all $i\in J$ for some $J$ with $x\in \beta(J)$; then the collection of $x_i^{-1}$ defines an inverse of $x$, with norm inverse to the norm of $x$. Moreover, if all $K_i$ are algebraically closed, then so is the Banach ultraproduct: Lift any monic equation over the completed filtered colimit (which is a quotient) to the product $\prod_{i\in J}^{\mathrm{Ban}} K_i$, pick a minimal-norm solution of the equation in each term to get a solution in $\prod_{i\in J}^{\mathrm{Ban}} K_i$ and then specialize back.
\end{proof}

\begin{corollary} There is a basis for the arc-topology consisting of strictly totally disconnected Banach rings.
\end{corollary}

Thus, we can define the same topos if we restrict the objects to strictly totally disconnected Banach rings. Now, finally, the arc-topology feels adequate:

\begin{theorem} On the site of strictly totally disconnected Banach rings, the arc-topology is subcanonical.
\end{theorem}

\begin{proof} It suffices to see that for any arc-cover $A\to B$ of strictly totally disconnected Banach rings, and any arc-cover $B\otimes_A B\to C$ by another strictly totally disconnected Banach ring, the map
\[
A\to \mathrm{eq}(B\rightrightarrows C)
\]
is an isomorphism of Banach rings. Note first that any uniform ring injects into a strictly totally disconnected Banach ring (by the Gelfand transform), and hence this is equivalent to the map
\[
A\to \mathrm{eq}(B\rightrightarrows (B\otimes_A B)_u^\wedge)
\]
being an isomorphism, where $-_u^\wedge$ denotes the uniform completion.

Let $A'=\mathrm{eq}(B\rightrightarrows (B\otimes_A B)_u^\wedge)$. We get the map $\mathcal M(A')\to \mathcal M(A)$. It suffices to see that this is an isomorphism and that the induced maps of residue fields are isomorphisms. Indeed, this implies that $A'$ is itself strictly totally disconnected, and the resulting map $A\to A'$ must be injective with dense image (as the map is an isomorphism on residue fields, so one can locally approximate any element of $A'$, and then globally as the Berkovich space is profinite). It is also an isometry by uniformity; this implies, by completeness, that it is an isomorphism. Passing to fibres, we can assume that $A$ is an algebraically closed non-discrete Banach field $K$. As we are free to pass to refinements of the given cover (as we have a separated presheaf), we can then also assume that $B=L$ is a Banach field. If $K=\mathbb C$ then $B=\mathbb C$ by Gelfand--Mazur and there is nothing to prove. Thus, we can assume that $K$ is non-archimedean.

If $K\to \mathrm{eq}(L\rightrightarrows (L\otimes_K L)_u^\wedge)$ is not an isomorphism, we can pick some $x$ in the target that is not contained in $K$, and replace $L$ by the completion of $K(x)\subset L$, which is then a residue field of the unit ball, and is one of types (2), (3) and (4). In all cases it easy to check that the completion of $L\otimes_K L$ is already uniform, so that the equalizer is just $K$.
\end{proof}

Over $p$-adic rings, it is sufficient to pass to perfectoid rings to get the subcanonical nature.

\begin{theorem} On the site of Banach rings $A$ with $|p|_A<1$ and which are perfectoid, the arc-topology is subcanonical.
\end{theorem}

\begin{proof} For any pushout diagram $B\leftarrow A\to C$ of perfectoid rings, the completed tensor product $B\hat{\otimes}_A C$ is again perfectoid, see \cite[Proposition 6.18]{ScholzePerfectoid}. Thus, one has to see that if $A\to B$ is an arc-cover of perfectoid rings, then the map
\[
A\to \mathrm{eq}(B\rightrightarrows B\hat{\otimes}_A B)
\]
is an isomorphism. Choosing auxiliary rings of integral elements, this is the v-descent for the structure sheaf on perfectoid spaces, see \cite[Proposition 8.8]{ECoD}.
\end{proof}

Thus, on the locus $|p|<1$, the present story is almost the same as the story of v-sheaves considered in \cite{ECoD}; the difference is that in \cite{ECoD}, one also fixes rings of integral elements $A^+\subset A$. In the language used by Huber \cite{Huber} this means that we are restricting to overconvergent v-sheaves.

\section{Finitary arc-sheaves}

Our goal is to define a functor
\[
A\mapsto \mathcal D_{\mathrm{mot}}(A),
\]
taking a Banach ring $A$ to the symmetric monoidal presentable stable $\infty$-category $\mathcal D_{\mathrm{mot}}(A)$ of ``motivic sheaves on $\mathcal M_{\mathrm{arc}}(A)$''. This will contain a full subcategory
\[
\mathcal D_{\mathrm{mot}}^{\mathrm{eff}}(A)\subset \mathcal D_{\mathrm{mot}}(A)
\]
of effective motives, which will be the derived $\infty$-category of ball-invariant finitary arc-sheaves of abelian groups. In this section, we start by looking at the finitary arc-sheaves.

It turns out that finitaryness implies being an arc-sheaf, in the following sense.

\begin{theorem}\label{theorem:finitaryarc} Let $A$ be any Banach ring. Consider a functor $\mathcal F$ from the category of strictly totally disconnected Banach $A$-algebras towards anima that commutes with finite products and with filtered colimits. Then $\mathcal F$ is a hypercomplete arc-sheaf.

In particular, if $\mathcal F$ is a functor to abelian groups, then it is a sheaf of abelian groups, and for all strictly totally disconnected $A$-algebras $B$, the cohomology of $\mathcal F$ on $B$ vanishes in positive degrees.
\end{theorem}

\begin{remark} It follows formally that the result is true more generally for functors with values in compactly generated $\infty$-categories $\mathcal C$ (or even in compactly assembled $\infty$-categories). Indeed, in the compactly generated case, this follows by evaluating on compact objects (which commutes with all limits and all filtered colimits). In the compactly assembled case, there are sufficiently many sequential Pro-objects $(X_n)_n$ with compact transition maps (in the sense that the collection of functors $\mathrm{colim}_n \mathrm{Hom}(X_n,-)$ detect equivalences). To construct those, fix any compact map $X'\to X$. Let $X_0=X$ and pick a factorization $X'\to X_1\to X_0$ where both maps are compact -- this exists in a compactly assembled category. Now inductively construct $X'\to \ldots\to X_2\to X_1\to X_0$. The functors $\mathrm{colim}_n \mathrm{Hom}(X_n,-)$ commute with finite limits and all filtered colimits. Now if $\mathcal F$ takes values in $\mathcal C$ and commutes with finite products and filtered colimits, then for any such $(X_n)_n$, also $\mathrm{colim}_n \mathrm{Hom}(X_n,\mathcal F)$ has this property, and hence is a hypercomplete arc-sheaf by the theorem. This implies that $\mathcal F$ is a sheaf with values in $\mathcal C$, as the collection of functors $\mathrm{colim}_n \mathrm{Hom}(X_n,-)$ also reflects limit diagrams. Indeed, for $\omega_1$-compact $X$, one can write $\mathrm{Hom}(X,-)$ as a limit of such functors.
\end{remark}

For this proof, we will use the quasi-pro-\'etale topology; the terminology follows \cite[Definition 10.1, Lemma 7.19]{ECoD}.

\begin{definition}\label{def:quasiproetale} A map $A\to B$ of Banach rings is quasi-pro-\'etale if after any base change $A\to C$ to an algebraically closed Banach field $C$, the uniform completion of $B\otimes_A C$ is isomorphic to $\mathrm{Cont}(S,C)$ for some profinite set $S$.
\end{definition}

Previously, we constructed arc-covers by strictly totally disconnected Banach rings. Actually, one can find quasi-pro-\'etale covers by strictly totally disconnected Banach rings:

\begin{lemma}\label{lem:quasiproetalecover} Let $A$ be an analytic Banach ring. Then there is a quasi-pro-\'etale cover $A\to B$ such that $B$ is strictly totally disconnected.
\end{lemma}

\begin{proof} Take any profinite set $S$ mapping to $\mathcal M(A)$. We claim that there is an initial uniform $A$-Banach algebra $A_S$ with a lift of $\mathcal M(A_S)\to \mathcal M(A)$ to $\mathcal M(A_S)\to S$. Moreover, this initial $A_S$ satisfies $\mathcal M(A_S)=S$ (while residue fields are unchanged, so $A\to A_S$ is quasi-pro-\'etale). Indeed given any $f\in A$ we have a cover of $\mathcal M(A)$ by the interiors of $\mathcal M(A\langle f\rangle_r)$ and $\mathcal M(A\langle f^{-1}\rangle_s)$ when $s<r$. The pullback of this cover to $S$ splits, and picking such a splitting, the initial $A_S$ necessarily admits a map from $A\langle f\rangle_r\times A\langle f^{-1}\rangle_s$. Continuing this way, we can by transfinite induction replace $A$ by algebras $A_i$ so that $S$ lifts compatibly to $\mathcal M(A_i)$ and in the inverse limit $\mathcal M(A_i)$ becomes totally disconnected. This reduces us to the case $A$ is totally disconnected. Replacing $A$ by a completed filtered limit of finite products of $A$ (more precisely, by $\mathrm{Cont}(S,A)$ which is the completed filtered colimit of $A^{\overline{S}}$ over finite quotients $S\to \overline{S}$), we can assume that $S$ is a subset of $\mathcal M(A)$. In that case $A_S$ is a quotient of $A$, corresponding to a completed filtered colimit of direct summands.

Taking $S\to \mathcal M(A)$ to be surjective, we can thus find a totally disconnected quasi-pro-\'etale cover. Finally we can make it strictly totally disconnected by taking a limit of finite \'etale covers, noting that they always spread from a point to a small neighborhood (as local rings are henselian).
\end{proof}

If $A$ is already strictly totally disconnected, one can completely understand further quasi-pro-\'etale maps:

\begin{lemma}\label{lem:quasiproetaleoverstd} Let $A$ be strictly totally disconnected. Then the category of quasi-pro-\'etale strictly totally disconnected $A$-algebras $B$ is equivalent to the category of profinite sets over $\mathcal M(A)$, by sending $B$ to $\mathcal M(B)$.
\end{lemma}

\begin{proof} This is immediate from the construction in the previous lemma, which yields the inverse functor.
\end{proof}

In particular, we can understand quasi-pro-\'etale sheaves.

\begin{lemma}\label{lem:quasiproetalesheafprofin} Let $S$ be a profinite set and consider the site of profinite sets $S'$ over $S$ equipped with the Grothendieck topology with covers finite families of jointly surjective maps. Any contravariant functor $S'\mapsto \mathcal F(S')$ towards anima that takes finite disjoint unions to finite products and cofiltered inverse limits to filtered colimits defines a hypersheaf, and the category of such is equivalent to the category of sheaves of anima on the topological space $S$.
\end{lemma}

\begin{proof} The category of $S'/S$ is the Pro-category of the category of profinite sets over $S$ that are finite disjoint unions of open and closed subsets. This is probably most easily seen in the dual setting of Boolean algebras, where it says that the category of Boolean $\mathrm{Cont}(S,\mathbb F_2)$-algebras is compactly generated, and the compact objects correspond to finite disjoint unions of open and closed subsets of $S$. The compact generation (by finitely presented Boolean algebras) is formal. Writing $\mathrm{Cont}(S,\mathbb F_2)$ as the filtered colimit of $\mathrm{Cont}(\overline{S},\mathbb F_2)$ over all finite quotients $\overline{S}$ of $S$, the finitely presented objects arise as base change from some finite quotient $\overline{S}$ of $S$, i.e.~we precisely get those profinite sets $S'$ over $S$ that can be written as $\overline{S}'\times_{\overline{S}} S$ for a map of finite sets $\overline{S}'\to \overline{S}$. But this is equivalent to the condition that $S'$ is a finite disjoint union of open and closed subsets of $S$.

The datum of $\mathcal F$ as in the statement is then equivalent to the datum of a functor $\mathcal F_0$ defined on open and closed subsets and commuting with disjoint unions, i.e.~a sheaf of anima on the topological space $S$. To see that $\mathcal F$ is a hypersheaf, we write it as the limit of its truncations, and so can assume that $\mathcal F$ is $n$-truncated for some $n$. Then any cover $S'\to S$ is a cofiltered limit of covers $S'_i\to S$ where $S'_i$ is a finite disjoint union of open and closed subsets of $S$. Then the sheaf property for $S'\to S$ follows from the sheaf property for $S'_i\to S$ upon passing to filtered colimits. In this last step, we use that filtered colimits commute with finite limits, and that in the $n$-truncated case, the limit in the descent diagram can be taken to be a finite limit.
\end{proof}

We can now extend finitary sheaves from strictly totally disconnected Banach rings to all analytic Banach rings, at least when they are $n$-truncated for some $n$.

\begin{lemma}\label{lem:quasiproetalecohomfiltcolim} Let $A$ be an analytic Banach ring, and fix some integer $n$. Consider a functor $\mathcal F$ from strictly totally disconnected $A$-algebras to $n$-truncated anima that commutes with finite products and with filtered colimits. Extend $\mathcal F$ to all uniform Banach $A$-algebras by quasi-pro-\'etale sheafification. Then the extension of $\mathcal F$ also commutes with finite products and filtered colimits.
\end{lemma}

\begin{proof} Note that by the previous two lemmas, $\mathcal F$ is already a sheaf on strictly totally disconnected algebras, and hence the value of $\mathcal F$ is unchanged on them.

The commutation with finite products is clear. For filtered colimits, we make extensive use of the fact that because $\mathcal F$ is $n$-truncated, any descent diagram can be taken to be finite, and the resulting finite limits commute with filtered colimits.

We first consider the case of totally disconnected $A$-algebras. If $B$ is totally disconnected, we can pick a quasi-pro-\'etale cover $B\to \tilde{B}$ with $\tilde{B}$ strictly totally disconnected by choosing a tower of finite \'etale covers. This computes the value on $B$ as a totalization of the values on $\tilde{B}$, the uniform completion of $\tilde{B}\otimes_B \tilde{B}$ etc.~, all of which are strictly totally disconnected. This description also works for all quotients of $B$ and shows commutation with filtered colimits for quotients of $B$. In particular, the value at $B$ are the global sections of a sheaf on $\mathcal M(B)$ whose stalks are the values on the residue fields of $B$. Using this observation, the general commutation with filtered colimits for totally disconnected $A$-algebras reduces to the case of Banach fields. Let $(K_i)_i$ be a filtered diagram of Banach fields over $A$, with completed filtered colimit $K$. To prove commutation with general filtered colimits, it suffices to prove commutation with colimits along an ordinal. In that situation, one can by transfinite induction construct compatible completed algebraic closures $C_i$ of $K_i$, yielding a completed algebraic closure $C$. Then the claim follows as before.

For a general filtered diagram $(A_i)_i$ of uniform Banach $A$-algebras, pick functorial profinite sets $S_i$ surjecting onto $\mathcal M(A_i)$ (for example, $S_i$ is the Stone-\v{C}ech compactification of $\mathcal M(A_i)$ as a discrete set) and set $\tilde{A}_i = A_{i,S_i}$. Then $A_i\to \tilde{A}_i$ is a functorial cover by a totally disconnected $A$-algebra, and also all further completed tensor products $\tilde{A}_i\otimes_{A_i} \tilde{A}_i$ etc.~are totally disconnected. Writing out the descent, this reduces us to the totally disconnected case that was already handled.
\end{proof}

To understand the difference between arc-sheaves and quasi-pro-\'etale sheaves, we use the following splitting theorem.

\begin{lemma}\label{lem:coversplits} Let $A$ be a strictly totally disconnected Banach ring and let $A\to B$ be an arc-cover. Then $B$ is a completed filtered colimit of $A$-algebras $B_i$ such that each $A\to B_i$ admits a splitting.
\end{lemma}

\begin{proof} Using that seminormed $A$-algebras are compactly assembled, we can reduce to the case that $B$ is (the completion of) a finitely presented seminormed $A$-algebra. Let $x_1,\ldots,x_n\in B$ be generators and $f_1,\ldots,f_m\in A[x_1,\ldots,x_n]$ the relations. We can replace $B$ by $A[x_1,\ldots,x_n]$ with the pullback of the norm (so that $|f_i|=0$), as this has the same completion; and then we can in turn write this as a sequential colimit of seminormed rings where we only enforce $|f_i|\leq \epsilon$. This way, we can arrange that $\mathcal M(B_i)$ sits inside some finite-dimensional affine-space over $A$ and contains an open subset of the affine space that surjects onto $\mathcal M(A)$. The next lemma shows that this gives a splitting.
\end{proof}

\begin{lemma}\label{lem:smoothspacessplit} Let $A$ be a strictly totally disconnected Banach ring and let $U\subset \mathbb A^n_A$ be an open subset that surjects onto $\mathcal M(A)$. Then there is an $A$-section of $\mathbb A^n_A$ contained in $U$.
\end{lemma}

\begin{proof} If we can find a section at each point $x\in \mathcal M(A)$, then any way to extend this to a section of $\mathbb A^n_A$ will lie in $U$ in an open neighborhood of $x$; as $A$ is strictly totally disconnected, we can then find a section on all of $A$ by combining local sections on open and closed subsets.

Thus, we can assume that $A=C$ is an algebraically closed non-discrete Banach field. If $C=\mathbb C$, the result follows from Gelfand--Mazur. If $C$ is nonarchimedean, we can find some nonzero Tate algebra $C\langle T_1,\ldots,T_n,X_1,\ldots,X_m\rangle_1/(g_1,\ldots,g_m)$ mapping into $U$ by taking any point of $U$ and describing a small neighborhood of it in terms of  inequalities $|f_i|\leq 1$ or $|f_i|\geq 1$; accordingly, $g_i=X_i-f_i$ or $g_i = X_if_i - 1$. But all maximal ideals of the Tate algebra $C\langle T_1,\ldots,T_n,X_1,\ldots,X_m\rangle_1$ have residue field $C$.
\end{proof}

Finally, we can give the proof of Theorem~\ref{theorem:finitaryarc}.

\begin{proof}[Proof of Theorem~\ref{theorem:finitaryarc}] We can write $\mathcal F$ as the limit of $\tau_{\leq n} \mathcal F$ (naively truncating each value), where each $\tau_{\leq n} \mathcal F$ still commutes with finite products and with filtered colimits. Thus, we can assume that $\mathcal F$ takes values in $n$-truncated anima for some $n$.

Any analytic Banach ring $A$ admits a quasi-pro-\'etale map $A\to B$ such that $B$ is strictly totally disconnected, cf.~Lemma~\ref{lem:quasiproetalecover}. If $A$ is strictly totally disconnected, then the category of quasi-pro-\'etale maps $A\to B$ to strictly totally disconnected algebras is equivalent to the category of profinite sets over $\mathcal M(A)$, cf.~Lemma~\ref{lem:quasiproetaleoverstd}. It follows that $B\mapsto\mathcal F(B)$ is a sheaf in the quasi-pro-\'etale topology on strictly totally disconnected $A$-algebras $B$, cf.~Lemma~\ref{lem:quasiproetalesheafprofin}. Sending any $B$ to the quasi-pro-\'etale cohomology of $\mathcal F$ on $B$ then defines an extension of $\mathcal F$ to all Banach $A$-algebras. This functor still commutes with finite products and with filtered colimits, cf.~Lemma~\ref{lem:quasiproetalecohomfiltcolim}.

Now take any arc-cover $A\to B$ of strictly totally disconnected algebras. We want to show that $\mathcal F$ satisfies descent along $A\to B$. Using Lemma~\ref{lem:coversplits}, we write $B$ as a completed filtered colimit of $A$-algebras $B_i$ that admit a splitting. As $A\to B_i$ has a splitting, it is formal that descent holds along $A\to B_i$. But the descent diagram for $A\to B$ is then the filtered colimit of the descent diagrams for $A\to B$. As $\mathcal F$ is $n$-truncated, this is finite limit diagram that we can commute past the filtered colimit.
\end{proof}

\begin{definition} Let $X$ be a small arc-stack. We consider the stable $\infty$-category of finitary arc-sheaves, defined as the full subcategory
\[
\mathcal D_{\mathrm{fin}}(X,\mathbb Z)\subset \mathcal D(X_{\mathrm{arc}},\mathbb Z)
\]
whose restriction to strictly totally disconnected Banach rings over $X$ commutes with filtered colimits.
\end{definition}

\begin{remark} For objects in $\mathcal D^+$, Lemma~\ref{lem:quasiproetalecohomfiltcolim} shows that the commutation with filtered colimits holds true on all Banach rings. The restriction to strictly totally disconnected Banach rings is thus only a maneuver to avoid problems with Postnikov limits which may not converge sufficiently fast on general Banach rings.

But note that for any $X$, we can define its cohomological dimension (for finitary arc-sheaves) as the minimum $d$ such that for all $M\in \mathcal D(X_{\mathrm{arc}},\mathbb Z)$ concentrated in degree $0$, we have $H^i(X_{\mathrm{arc}},M)=0$ for $i>d$. If $A=\mathrm{colim}_i A_i$ is any filtered colimit of Banach algebras over $X$ such that the cohomological dimension of the $\mathcal M_{\mathrm{arc}}(A_i)$ is uniformly bounded, then for any $M\in \mathcal D(X_{\mathrm{arc}},\mathbb Z)$, the map
\[
\mathrm{colim}_i M(A_i)\to M(A)
\]
is an isomorphism. Indeed, this holds even without the assumption on cohomological dimensions when $M\in \mathcal D^+$ by Lemma~\ref{lem:quasiproetalecohomfiltcolim}. Under the assumption of bounded cohomological dimensions, it can be deduced in general by passing to Postnikov limits. Indeed, the Postnikov limit commutes past the filtered colimit over $i$ as in any cohomological degree, it becomes eventually constant.
\end{remark}

This automatically has a symmetric monoidal tensor structure, as well as symmetric monoidal pullback functors, compatible with the embedding into $\mathcal D(X_{\mathrm{arc}},\mathbb Z)$.

\begin{proposition} The $\infty$-category $\mathcal D_{\mathrm{fin}}(X,\mathbb Z)$ is equivalent to the $\infty$-category of functors from strictly totally disconnected Banach rings over $X$ to $\mathcal D(\mathbb Z)$ that commute with finite products and with filtered colimits.
\end{proposition}

\begin{proof} This is clear from the above.
\end{proof}

\begin{proposition} The subcategory
\[
\mathcal D_{\mathrm{fin}}(X,\mathbb Z)\subset \mathcal D(X_{\mathrm{arc}},\mathbb Z)
\]
is stable under all colimits, truncations, and is left-complete.
\end{proposition}

\begin{proof} This is immediate from the previous proposition.
\end{proof}

Over the complex numbers, finitary arc-sheaves reduce to usual sheaves.

\begin{theorem}\label{thm:archimedeanfinitary} If $A$ is a Banach algebra over $\mathbb C$, then pullback under $\mathcal M_{\mathrm{arc}}(A)_{\mathrm{arc}}\to \mathcal M(A)$ induces an equivalence
\[
\widehat{\mathcal D}(\mathcal M(A),\mathbb Z)\cong \mathcal D_{\mathrm{fin}}(\mathcal M_{\mathrm{arc}}(A),\mathbb Z).
\]
where the source denotes the left-complete derived $\infty$-category of abelian sheaves on the compact Hausdorff space $\mathcal M(A)$.
\end{theorem}

\begin{proof} By descent, we can reduce to the case that $A$ is strictly totally disconnected (noting that the left-completed category does satisfy descent). The arc-site over $A$ is then equivalent to the category of profinite sets over the profinite set $\mathcal M(A)$, i.e.~the pro-\'etale site of the profinite set $\mathcal M(A)$, and the finitary sheaves in there are precisely the \'etale sheaves, i.e.~the usual sheaves.
\end{proof}

We need the following way to compute the cohomology of finitary arc-sheaves. First, we have the following result at points.

\begin{proposition} Let $K$ be a non-discrete Banach field and let $M\in \mathcal D_{\mathrm{fin}}(\mathcal M_{\mathrm{arc}}(K),\mathbb Z)$. Let $C$ be a completed algebraic closure of $K$ and $G_K$ the resulting absolute Galois group. Then $M(C)$ has a natural continuous $G_K$-action and
\[
M(K)\cong M(C)^{hG_K}
\]
is the continuous cohomology of $G_K$ acting on $M(C)$.
\end{proposition}

Implicitly, we are using here the condensed interpretation of continuous group cohomology for profinite groups $G$. Namely, there is the site of profinite $G$-sets, sheaves on which are equivalent to condensed sets equipped with an action of $G$ (considered as a condensed group). This contains the full subcategory of finitary sheaves of abelian groups, which is classically the category of continuous discrete $G$-modules. Under this equivalence, cohomology on the site of profinite $G$-sets, i.e.~condensed group cohomology, agrees with continuous group cohomology.

\begin{proof} This follows from pushforward along the map of sites from $X_{\mathrm{arc}}$ to profinite $G_K$-sets, where $X=\mathcal M_{\mathrm{arc}}(K)$.
\end{proof}

\begin{proposition} Let $A$ be a Banach ring and $M\in \mathcal D_{\mathrm{fin}}^+(\mathcal M_{\mathrm{arc}}(A),\mathbb Z)$. Consider the map of sites
\[
\nu: \mathcal M_{\mathrm{arc}}(A)_{\mathrm{arc}}\to \mathcal M(A)
\]
where the target denotes the site of open subsets of the compact Hausdorff space $\mathcal M(A)$. Then $R\nu_\ast M$ is a complex of sheaves on $\mathcal M(A)$ whose stalk at $x\in \mathcal M(A)$ is given by $M(K(x))$.
\end{proposition}

\begin{proof} This is a consequence of the extension of $M$ to a functor on all Banach $A$-algebras that commutes with filtered colimits.
\end{proof}

\begin{definition} Let $X$ be a small arc-stack. The cohomological dimension $d(X)$ is the cohomological dimension of the functor
\[
\mathcal D_{\mathrm{fin}}(X,\mathbb Z)\to \mathcal D(\mathbb Z): M\mapsto M(X).
\]
\end{definition}

\begin{corollary} If $X=\mathcal M_{\mathrm{arc}}(A)$ for an analytic Banach ring $A$, then $d(X)$ is bounded by the sum of the cohomological dimension of the compact Hausdorff space $\mathcal M(A)$, and the supremum over all $x\in \mathcal M(A)$ of the cohomological dimension of $G_{K(x)}$.
\end{corollary}

We are interested in constructing some six-functor formalism later on. Doing so, there is usually a distinguished class of ``proper'' maps $f$ for which definitionally $f_!=f_\ast$. We will be using the following class.

\begin{definition} A $0$-truncated map $f: Y\to X$ of small arc-stacks is proper if it is qcqs.

The cohomological dimension $d(f)$ is the supremum of $d(Y_x)$ over all $x\in X$. The map has locally finite cohomological dimension if the pullback to any quasicompact substack of $X$ has finite cohomological dimension.
\end{definition}

We get the following version of proper base change.

\begin{theorem}\label{thm:properpushforwardfinitary} Let $f: Y\to X$ be a $0$-truncated proper map of small arc-sheaves with finite cohomological dimension. Then
\[
f_\ast: \mathcal D(Y_{\mathrm{arc}},\mathbb Z)\to \mathcal D(X_{\mathrm{arc}},\mathbb Z)
\]
preserves finitary arc-sheaves, and on such has finite cohomological dimension.

In particular,
\[
f_\ast: \mathcal D_{\mathrm{fin}}(Y,\mathbb Z)\to \mathcal D_{\mathrm{fin}}(X,\mathbb Z)
\]
commutes with any base change. If $f$ is quasi-pro-\'etale, i.e.~all geometric fibres are profinite sets, then $f_\ast$ satisfies the projection formula.
\end{theorem}

\begin{proof} First, one proves preservation of finitary arc-sheaves for $\mathcal D^+$, for which no assumption on finite cohomological dimension is necessary. We can assume that $X=\mathcal M_{\mathrm{arc}}(A)$ for some analytic Banach ring $A$, and we can simplicially resolve $Y$ by affine spaces $\mathcal M_{\mathrm{arc}}(B_i)$ as $Y$ is qcqs. Then the result follows from the extension of any $M\in \mathcal D_{\mathrm{fin}}^+(Y,\mathbb Z)$ to all uniform Banach algebras over $Y$ that commutes with filtered colimits. In particular, as stalks are a special case of filtered colimits, everything commutes with passing to stalks, so one gets the bound on cohomological dimension, and then one can pass to Postnikov limits to get preservation of finitary arc-sheaves on all of $D$. We use that all hypercomplete arc-sheaves are Postnikov complete, as the arc-topos is replete, see \cite{MondalReinecke}.

The base change property for $f_\ast$ as a functor on arc-sheaves is formal, as one always has base change for slices, and in the arc-topos, everything is a slice. In particular, the projection formula can always be checked on geometric fibres. If $f$ is quasi-pro-\'etale, this reduces the projection formula to the case $Y=S\times \mathcal M_{\mathrm{arc}}(C)\to X=\mathcal M_{\mathrm{arc}}(C)$ for a profinite set $S$. Writing $S$ as a cofiltered limit of finite sets and using finitaryness yields the claim in this case.
\end{proof}

\begin{definition} A map of small arc-stacks $j: U\to X$ is an open immersion if after pullback to each $\mathcal M_{\mathrm{arc}}(A)\to X$, the fibre product $U\times_X \mathcal M_{\mathrm{arc}}(A)$ is representable by an open subspace of $\mathcal M_{\mathrm{arc}}(A)$, i.e.~a subfunctor coming from an open subspace of $\mathcal M(A)$.
\end{definition}

More generally, we have the following notion.

\begin{definition} A map $f: Y\to X$ of small arc-stacks is finitary if for any filtered colimit $A=\mathrm{colim}_i A_i$ of strictly totally disconnected Banach rings $A_i$, the map
\[
\mathrm{colim}_i Y(A_i)\to \mathrm{colim}_i Y(A)\times_{X(A)} X(A_i)
\]
is an isomorphism. Equivalently, the diagram
\[\xymatrix{
\mathrm{colim}_i Y(A_i)\ar[r]\ar[d] & Y(A)\ar[d]\\
\mathrm{colim}_i X(A_i)\ar[r] & X(A)
}\]
is cartesian.
\end{definition}

We note that this is equivalently the condition that $Y$ defines a finitary arc-sheaf over $X$.

\begin{lemma}\leavevmode
\begin{enumerate} 
\item[{\rm (i)}] Composites of finitary maps are finitary.
\item[{\rm (ii)}] Pullbacks of finitary maps are finitary.
\item[{\rm (iii)}] If $f: X\to Y$ and $g: Y\to Z$ are maps such that $g$ and $g\circ f$ are finitary, then $f$ is finitary.
\item[{\rm (iv)}] Given $X$, the category of finitary $Y$ over $X$ is stable under all colimits.
\item[{\rm (v)}] Any open immersion is finitary.
\item[{\rm (vi)}] Conversely, a finitary injection is an open immersion.
\end{enumerate}
\end{lemma}

\begin{proof} Parts (i) -- (iv) follow directly from the definition and standard properties of anima. Given these results, it suffices to check (v) for the open immersion $U\subset \mathcal M(B)$ given by asking a condition $|T|<r$ for some $T\in B$. This follows from the norm on a filtered colimit being given by an infimum. For (vi), we can assume that $f: Y\to X$ has strictly totally disconnected target. For any $x\in X$ in the image of $f$, we can write $x$ as a cofiltered limit of open and closed neighborhoods, and by finitaryness, one such neighborhood must lie in $Y$. It follows that $Y$ is open.
\end{proof}

\begin{proposition}\label{prop:finitaryleftadjoint} For any finitary map $f: Y\to X$, the left adjoint $f_\sharp: \mathcal D(Y_{\mathrm{arc}},\mathbb Z)\to \mathcal D(X_{\mathrm{arc}},\mathbb Z)$ to $f^\ast$ (on the level of derived categories of arc-sheaves) preserves the subcategory of finitary sheaves, defining a functor
\[
f_\sharp: \mathcal D_{\mathrm{fin}}(Y,\mathbb Z)\to \mathcal D_{\mathrm{fin}}(X,\mathbb Z)
\]
left adjoint to $f^\ast$, and satisfying the projection formula.
\end{proposition}

\begin{proof} This is even true unstably, and then it is simply the assertion that a composite of two finitary maps is finitary.
\end{proof}

Moreover, excision holds true for finitary sheaves. For open immersions $j$, we write $j_!=j_\sharp$ for the left adjoint of $j^\ast$.

\begin{proposition} Let $X$ be a small arc-stack and let $j: U\subset X$ be an open substack. Let $i: Z\subset X$ be the complementary closed substack, whose $A$-valued points are those points of $X$ that do not meet $U$. Then the triangle
\[
j_! j^\ast M\to M\to i_\ast i^\ast M
\]
is exact for all $M\in \mathcal D_{\mathrm{fin}}(X,\mathbb Z)$.
\end{proposition}

\begin{proof} By descent, we can assume that $X$ is strictly totally disconnected. In that case $U$ is an increasing union of open and closed subsets $U_k\subset U$, with complementary open and closed $Z_k\supset Z$ (for $k$ in some filtered index category $K$). For each $k$, we get an obvious short exact sequence
\[
j_{k!} j_k^\ast M\to M\to i_{k\ast} i_k^\ast M
\]
using just the decomposition into two open and closed subsets. Now pass to the filtered colimit over $k$, and use finitaryness to see that this yields the desired exact triangle.
\end{proof}

In particular, we can apply excision to break up into the archimedean locus $|2|>1$, and the nonarchimedean locus $|2|\leq 1$. As the archimedean locus reduces to usual sheaves on the Berkovich space, we will largely focus on the nonarchimedean locus in the following.

At this point, we can almost get a $6$-functor formalism with respect to the class $P$ of proper maps of finite dimension. The only issue is the failure of the projection formula for general proper morphisms.

\begin{example} Let $X$ be nonarchimedean and consider the projection $f: \mathbb P^1_X\to X$ and the constant sheaf $\mathbb Q$. Then $f_\ast \mathbb Q=\mathbb Q$: Indeed, this can be checked fibrewise, reducing us to geometric points. There is no Galois cohomology, so this reduces to the cohomology of $\mathcal M(\mathbb P^1_C)$ which is a tree. If the projection formula were true, this would mean that for any $M\in \mathcal D_{\mathrm{fin}}(X,\mathbb Q)$, one has $f_\ast f^\ast M=M$ and in particular $M(\mathbb P^1_X)=M(X)$. But the sheaf $\overline{\mathbb G}_m = \mathbb G_m/(1+\mathcal O_{<1})$ studied in the next section yields a counterexample. It is visibly finitary, and concentrated in degree $0$ when evaluated at $C$. But $H^1(\mathbb P^1_C,\overline{\mathbb G}_m)\neq 0$ because $\mathbb P^1_C$ has a nontrivial line bundle. As we will see in the next section, the sheaf $\overline{\mathbb G}_m$ used in this counterexample is even ball-invariant; without ball-invariance, there are many more counterexamples.
\end{example}

\section{Effective motives}

As a next step towards the definition of motives, we look at the subcategory of ball-invariant finitary sheaves.

\begin{definition} Let $X$ be a small arc-stack, and let $Y\in \mathcal D(X_{\mathrm{arc}},\mathbb Z)$, or just an arc-sheaf of anima over $X$. Then $Y$ is ball-invariant if for all Banach rings $A$ over $X$, the map
\[
Y(A)\to Y(A\langle T\rangle_1)
\]
is an isomorphism. Here $A\langle T\rangle_1$ is the free uniform Banach $A$-algebra on a variable $T$ with $|T|\leq 1$.
\end{definition}

\begin{definition}\label{def:effmotives}\footnote{We work only with ``\'etale motives'', enforcing \'etale descent throughout; and only with $\mathbb Z$-coefficients, as this is the main setup we care about for applications. Everything would work without change if one replaces $\mathbb Z$ by the sphere spectrum (or an arbitrary $E_\infty$-ring); in fact the difference between these two settings is torsion, and hence reduces to the standard \'etale theory.} Let $X$ be a small arc-stack. The stable $\infty$-category of effective motivic sheaves over $X$ is the full $\infty$-subcategory
\[
\mathcal D_{\mathrm{mot}}^{\mathrm{eff}}(X)\subset \mathcal D_{\mathrm{fin}}(X,\mathbb Z)
\]
of all $M\in \mathcal D_{\mathrm{fin}}(X,\mathbb Z)\subset \mathcal D(X_{\mathrm{arc}},\mathbb Z)$ that are also ball-invariant.
\end{definition}

\begin{lemma} For all non-discrete algebraically closed Banach fields $C$, the Banach algebra $C\langle T\rangle_1$ has cohomological dimension at most $3$.
\end{lemma}

\begin{proof} If $C=\mathbb C$, then this reduces to the dimension of $\{z\in \mathbb C\mid |z|\leq 1\}$ which is $2$. If $C$ is nonarchimedean, then $\mathcal M(C\langle T\rangle_1)$ is an infinite tree, hence of dimension $1$; and for all $x\in \mathcal M(C\langle T\rangle_1)$, the absolute Galois group $G_{C(x)}$ is either trivial (when $x$ is a $C$-point) or a closed subgroup of the absolute Galois group $G_{C(T)}$ which has cohomological dimension $2$ (on $\mathbb Z$-modules).
\end{proof}

\begin{corollary} Assume that $M\in \mathcal D_{\mathrm{fin}}(X,\mathbb Z)$, and that for any non-discrete algebraically closed Banach field $C$ over $X$, the map
\[
M(C)\to M(C\langle T\rangle_1)
\]
is an isomorphism. Then $M$ is ball-invariant.
\end{corollary}

\begin{proof} We need to see that for all Banach rings $A$ over $X$, the  map
\[
M(A)\to M(A\langle T\rangle_1)
\]
is an isomorphism. By arc-descent, we can assume that $A$ is strictly totally disconnected. We can then treat both sides as sheaves on the profinite set $\mathcal M(A)$. As both sides are finitary (using Theorem~\ref{thm:properpushforwardfinitary}), the stalks agree with the values on the algebraically closed residue fields, where the result holds by assumption.
\end{proof}

\begin{corollary} The subcategory
\[
\mathcal D_{\mathrm{mot}}^{\mathrm{eff}}(X)\subset \mathcal D_{\mathrm{fin}}(X,\mathbb Z)
\]
is stable under all colimits.
\end{corollary}

\begin{proof} It suffices to note that the functor $M\mapsto M(C\langle T\rangle_1)$ commutes with all colimits by finite cohomological dimension.
\end{proof}

In particular, the inclusion has a right adjoint. Actually, it also has a left adjoint (and hence the subcategory $\mathcal D_{\mathrm{mot}}^{\mathrm{eff}}$ is also stable under all limits), which we will construct explicitly using a description involving simplices. As our model for the $n$-simplex, we inspire ourselves from the topological model $\{0\leq x_1\leq \ldots\leq x_n\leq 1\}$, and take
\[
\Delta^n = \mathbb B^n
\]
the $n$-dimensional unit ball, with transition maps compatible with the ones in this topological model (which sits inside the $\mathbb R$-points).\footnote{The advantage of this model is that all algebras are monoid algebras, and all transition maps are induced by monoid maps. We will however not really be using this.} We write $A\langle \Delta^n\rangle\cong A\langle T_1,\ldots,T_n\rangle_{1,\ldots,1}$ for the corresponding Banach-$A$-algebra.

\begin{proposition} There is a left adjoint $L_{\mathbb B}$ to the inclusion
\[
\mathcal D_{\mathrm{mot}}^{\mathrm{eff}}(X)\subset \mathcal D_{\mathrm{fin}}(X,\mathbb Z)
\]
and it takes any $M\in \mathcal D_{\mathrm{fin}}(X,\mathbb Z)$ to $L_{\mathbb B}(M)$ with
\[
L_{\mathbb B}(M)(A) = \mathrm{colim}_{[n]\in \Delta^{\mathrm{op}}} M(A\langle\Delta^n\rangle)
\]
for any strictly totally disconnected $A$.
\end{proposition}

\begin{proof} Using Theorem~\ref{thm:properpushforwardfinitary}, the formula defines an endofunctor $L_{\mathbb B}$ of $\mathcal D_{\mathrm{fin}}(X,\mathbb Z)$ with a transformation $\mathrm{id}\to L_{\mathbb B}$ that is an equivalence on objects of $D_{\mathrm{mot}}^{\mathrm{eff}}(X,\mathbb Z)$. To see that it is indeed the left adjoint, it suffices to see that $L_{\mathbb B}$ takes values in ball-invariant sheaves, and that the natural transformation $L_{\mathbb B}\to L_{\mathbb B}^2$ obtained by applying $L_{\mathbb B}$ to $\mathrm{id}\to L_{\mathbb B}$ is an equivalence; see \cite[Proposition 5.2.7.4]{LurieHTT}.

Note that
\[
L_{\mathbb B}(M)(A) = \mathrm{colim}_{[n]\in \Delta^{\mathrm{op}}} M(A\langle\Delta^n\rangle)
\]
holds also for any analytic Banach ring $A$ of finite cohomological dimension, as then evaluation at $A$ commutes with all colimits in $\mathcal D_{\mathrm{fin}}(-,\mathbb Z)$. To see that $L_{\mathbb B}(M)$ is ball-invariant, we first observe that the two maps
\[
i_0^\ast,i_1^\ast: L_{\mathbb B}(M)(\Delta^1_A)\to L_{\mathbb B}(M)(A)
\]
agree. Indeed, those two maps admit a homotopy constructed from a standard homotopy between the two maps $i_0,i_1: \Delta^\bullet_A\to \Delta^\bullet_{\Delta^1_A}$ of cosimplicial objects. But $i_0^\ast=i_1^\ast$ implies that a sheaf is ball-invariant: Indeed, it implies that for any homotopy $f: \Delta^1_A\times X\to Y$ of maps $X\to Y$, the maps $f_0: X\to Y$ and $f_1: X\to Y$ induce the same map on sections of the sheaf, and this applies in particular to $f_0: \Delta^1_A\to \Delta^1_A$ projecting everything to $0$, and $f_1=\mathrm{id}: \Delta^1_A\to \Delta^1_A$, taking $f$ to be multiplication.

Finally, the map $L_{\mathbb B}\to L_{\mathbb B}^2$ is an equivalence by another homotopy between cosimplicial objects.
\end{proof}

\begin{remark}\label{rem:ballinvariantpresheaves} More generally, on the $\infty$-category of all arc-presheaves of anima (defined on the test category of all uniform Banach rings $A$ over our fixed $X$), one has the subcategory of ball-invariant presheaves, where ball-invariance means that the values on $A$ and $A\langle T\rangle_1$ agree, for all $A$. Then there is a left adjoint $L_{\mathbb B}^{\mathrm{pre}}$ on it, given by the same formula
\[
L_{\mathbb B}^{\mathrm{pre}}(M)(A) = \mathrm{colim}_{[n]\in \Delta^{\mathrm{op}}} M(A\langle\Delta^n\rangle).
\]
The previous proposition then follows by noting that if $M$ is a finitary arc-presheaf, so is $L_{\mathbb B}^{\mathrm{pre}}(M)$; and $L_{\mathbb B}(M)$ is its arc-hypersheafification, which then does not change the value on strictly totally disconnected $A$.

Slightly mysteriously, the functor $L_{\mathbb B}^{\mathrm{pre}}$ defined on presheaves often turns out to have better properties than one might a priori expect; this is at the basis of Voevodsky's theory of presheaves with transfer.
\end{remark}

In the following proposition, we say that $M$ is $\mathbb B$-contractible if there is a map $\mathbb B\to \mathrm{End}(M)$ sending $1$ to $1$ and $0$ to $0$.

\begin{proposition} The functor
\[
L_{\mathbb B}: \mathcal D_{\mathrm{fin}}(X,\mathbb Z)\to \mathcal D_{\mathrm{mot}}^{\mathrm{eff}}(X)
\]
is a Verdier quotient whose kernel is the $\otimes$-ideal generated by all $\mathbb B$-contractible finitary sheaves.

In particular, there is a unique way to simultaneously endow $\mathcal D_{\mathrm{mot}}^{\mathrm{eff}}(X)$ and $L_{\mathbb B}$ with a symmetric monoidal structure.
\end{proposition}

\begin{proof} It is clear that any $\mathbb B$-contractible finitary sheaf is in the kernel, and that $\mathbb B$-contractible finitary sheaves are stable under tensor products with finitary sheaves. But the explicit description of $L_{\mathbb B}$ shows that the kernel is generated under colimits by the cone of the maps
\[
M\to \intHom(\mathbb B^n,M)
\]
for finitary $M$ and $n\geq 0$; and these are $\mathbb B$-contractible (via multiplication on $\mathbb B^n$).

The second part follows from \cite[Theorem I.3.6]{NikolausScholze}.
\end{proof}

With torsion coefficients prime to the characteristic, \'etale sheaves are ball-invariant. We did not define the \'etale site here, and instead use the condition ``invariance under change of algebraically closed field'' to single out \'etale sheaves inside all finitary arc-sheaves.

\begin{proposition}\label{prop:etalesheavesballinvariant} Let $M\in \mathcal D_{\mathrm{fin}}(X,\mathbb Z/n\mathbb Z)$ for some integer $n$ with $|n|\geq 1$ everywhere on $X$. Assume that $M$ is invariant under change of algebraically closed field, i.e. for any map $C\to C'$ of algebraically closed Banach fields over $X$, the map
\[
M(C)\to M(C')
\]
is an isomorphism. Then $M$ is ball-invariant.
\end{proposition}

Later we will see that the converse is also true: If $M\in \mathcal D_{\mathrm{fin}}(X,\mathbb Z/n\mathbb Z)$ and $M$ is ball-invariant, then $M$ is invariant under change of algebraically closed field. In particular, with torsion coefficients, the category of ball-invariant sheaves inherits a $t$-structure (which is the usual $t$-structure).

\begin{proof} Ball-invariance is a statement after pullback to algebraically closed Banach fields, so we can assume that $X=\mathcal M(C)$. The assumption on $M$ is then equivalent to the assertion that $M$ is the constant sheaf on the $\mathbb Z/n$-module $M_0=M(C)$. As taking cohomology on $C\langle T\rangle_1$ commutes with arbitrary direct sums, we can assume $M=\mathbb Z/n\mathbb Z$. We can now appeal to a theorem of Berkovich \cite[Theorems 6.3.9, 6.4.1]{BerkovichEtale}, or argue directly as follows. We compute the higher direct images towards $\mathcal M(C\langle T\rangle_1)$. We know that this is concentrated in degrees $0$ and $1$, and in degree $0$ we get the constant sheaf $\mathbb Z/n\mathbb Z$, whose cohomology in turn is just $\mathbb Z/n\mathbb Z$ in degree $0$, as the space is an infinite tree. We need to understand the sheaf in degree $1$, and show that it has no global sections or cohomology. First, we can understand the fibres of the sheaf: These are given by the Galois cohomology groups of the fields $K(x)$. At points of type (1) or (4), this is trivial; at points of type (3), given by $\mathbb Z/n\mathbb Z$; and at points of type (2), given by the reduction modulo $n$ of the group of degree $0$ divisors on $\mathbb P^1_k$, where $k$ is the residue field of $C$. As a sheaf, the description is the following: On any finite tree
\[
\{(x_a)_{a\in A}\mid \forall a,b: x_a\leq \mathrm{max}(x_b,|a-b|_C), |a-b|_C\leq \mathrm{max}(x_a,x_b)\}\subset \prod_{a\in A} [0,1]
\]
(for a finite subset $A\subset \mathcal O_C$), look at the sheaf which locally associates an element of $\mathbb Z/n\mathbb Z$ to any branch of the tree, so that at branching points the value at the upper branch is the sum of the values at the branches below; and the value at any branch meeting the boundary is zero. Now take the filtered colimit of the pullbacks of these sheaves, where pullback maps assign the value zero to any new branches. Using Kummer covers (and a choice of primitive $n$-th root of unity in $C$) one constructs a map from this candidate sheaf to the first pushforward, and then checks on stalks that it is an isomorphism. For the construction of the map, pick any such finite tree, given by some subset $A$ of $\mathcal O_C$. On any branch of the tree, the sheaf is the constant sheaf $\mathbb Z/n\mathbb Z$, and we send a class $k\in \mathbb Z/n\mathbb Z$ to the Kummer cover given by $n$-th roots of the function $(T-a)^k$ (where $a$ is chosen so that the branch of the tree is centered at $a$). At a branch point of the tree, we have the branches below, centered at $a_i$, and with assigned values $k_i\in \mathbb Z/n\mathbb Z$. We take this to the Kummer cover given by $\prod_i (T-a_i)^{k_i}$. On all the branches below, this gives the correct Kummer cover already defined, as when $T$ is close to $a_j$ for $j\neq i$ then
\[
\frac{T-a_i}{a_j-a_i} = 1 + \frac{T-a_j}{a_j-a_i}
\]
is close to $1$ and hence an $n$-th power; while $a_j-a_i\in C$ is also an $n$-th power, so the Kummer cover defined by $T-a_i$ becomes trivial. Thus, the Kummer cover for $\prod_i (T-a_i)^{k_i}$ agrees with the Kummer cover for $(T-a_j)^{k_j}$ when $T$ is close to $a_j$.

Also, at the branch above the branch point, the Kummer cover for $\prod_i (T-a_i)^{k_i}$ becomes equivalent to the one for $(T-a)^{\sum k_i}$ where $a$ is given by any of the $a_i$. Indeed, 
\[
\frac{T-a_i}{T-a} = 1 + \frac{a-a_i}{T-a_i}
\]
is close to $1$ in this case, and hence an $n$-th power. This computation reflects the condition that the value at the upper branch must be the sum of the values at the lower branches.

This constructs the map from this candidate sheaf towards the first pushforward, and one directly checks that at stalks, one gets an isomorphism, yielding thus an explicit description of this first pushforward.

To check vanishing of cohomology of this sheaf, reduce (by filtered colimits) to the sheaves on finite trees, and inductively project them to smaller trees. These projection maps are isomorphisms except at one point in the target, where the preimage is some closed interval $I$. This reduces to the vanishing of all cohomology groups of $j_! \mathbb Z/n\mathbb Z$ on $[0,1]$, for $j: [0,1)\hookrightarrow [0,1]$.
\end{proof}

With rational coefficients, the category of ball-invariant sheaves also inherits a $t$-structure (the so-called homotopy $t$-structure). However, this behaves differently from the $t$-structure with torsion coefficients; for example the Tate twist will sit in degree $-1$ and not in degree $0$.

\begin{proposition}\label{prop:homotopytstructure} The subcategory
\[
\mathcal D_{\mathrm{mot}}^{\mathrm{eff}}(X,\mathbb Q)\subset \mathcal D_{\mathrm{fin}}(X,\mathbb Q)
\]
is stable under truncations.
\end{proposition}

In Corollary~\ref{cor:homotopytstructure} below, we will see that this is also true with integral coefficients.

\begin{proof} This is a formal consequence of the functor $M\mapsto M(C\langle T\rangle_1)$ having cohomological dimension $1$. Indeed, there is a spectral sequence
\[
E_2^{ij} = H^i(C\langle T\rangle_1,\mathcal H^j M)\Rightarrow H^{i+j}(C\langle T\rangle_1,M)
\]
As $E_2^{ij}$ can be nonzero only for $i=0,1$, this necessarily degenerates, yielding short exact sequences
\[
0\to H^1(C\langle T\rangle_1,\mathcal H^{i-1} M)\to H^i(C\langle T\rangle_1,M)\to H^0(C\langle T\rangle_1,\mathcal H^i M)\to 0.
\]
Ball-invariance of $M$ implies that the middle term is just $H^i(C,M) = H^0(C,\mathcal H^i M)$. Then it follows that the first term must vanish, and $\mathcal H^i M$ is also ball-invariant.
\end{proof}

\subsection{Free motivic sheaves} For ball-invariant sheaves, pullback along smooth maps admits a left adjoint satisfying the projection formula. To prove this, we first handle the case of open subsets of affine space.

\begin{lemma}\label{lem:freemotivicsheafanima} Let $X$ be a small arc-stack and let $U\subset \mathbb A^n_X$ be an open subset for some $n$. Then the free ball-invariant arc-presheaf of anima
\[
L_{\mathbb B}^{\mathrm{pre}}(U): A\mapsto \mathrm{colim}_{[n]\in \Delta^{\mathrm{op}}} U(A\langle\Delta^n\rangle)
\]
is finitary.
\end{lemma}

\begin{proof} We can assume that $X=\mathcal M_{\mathrm{arc}}(A)$ is affine. Assume first that $X$ is nonarchimedean. Then the open balls $\mathbb D^n_X(0,<r)\subset \mathbb A^n_X$ around $0$ of radius $r$ are subgroups (with respect to coordinatewise addition). Moreover, $U$ is an increasing union of open subspaces $U_i$ that are invariant under $\mathbb D^n_X(0,<r_i)$ for some $r_i>0$. Then $U_i$ is homotopy-equivalent to the presheaf quotient $U_i/\mathbb D^n_X(0,<r_i)$ (more precisely, to each term in the simplicial presheaf computing it, and hence to its geometric realization), and this quotient is already finitary.

In general, we argue as follows to prove finitarity. Consider any filtered colimit of Banach rings $A=\mathrm{colim}_i A_i$ over $X$. We need to see that the map of simplicial sets
\[
\mathrm{colim}_i U(A_i\langle\Delta^\bullet\rangle)\to U(A\langle\Delta^\bullet\rangle)
\]
induces an equivalence of anima, under geometric realization. Using Kan's $\mathrm{Ex}^\infty$-functor, this can be checked by mapping simplicial subdivisions $\mathrm{sd}^\ell \partial\Delta^n$ of $\partial\Delta^n$'s into these simplicial sets, and checking that it induces homotopy equivalences. The key now is that any two maps that are sufficiently close are automatically $\mathbb B$-homotopic, via the straight-line interpolation in $U\subset \mathbb A^n$.

More precisely, given any finite simplicial set $S$ with a map $f: S\to U(A\langle \Delta^\bullet\rangle)$ as well as a sub-simplicial set $S'\subset S$ with a sequence of maps $g'_n: S'\to U(A_{i_n}\langle\Delta^\bullet\rangle)$ such that the $g'_n$ converge to $f|_{S'}$, we can always find some $j_n\geq i_n$ and an extension of $g'_n$ to a map $g_n: S\to U(A_{j_n}\langle\Delta^\bullet\rangle)$ such that the $g_n$ converge to $f$. Indeed, this can as usual be done by induction on simplices, reducing to $S=\Delta^k$ and $S'=\partial\Delta^k$, where it follows from the map
\[
\mathrm{colim}_i \mathrm{ker}(A_i\langle\Delta^k\rangle\to A_i\langle\partial\Delta^k\rangle)\to \mathrm{ker}(A\langle\Delta^k\rangle\to A\langle\partial\Delta^k\rangle)
\]
having dense image (and $U$ being open in affine space). Thus, given any map $f: \mathrm{sd}^\ell\partial\Delta^k\to U(A\langle \Delta^\bullet\rangle)$, we can find a sequence $g_n$ of maps $\mathrm{sd}^\ell\partial\Delta^k\to U(A_{i_n}\langle\Delta^\bullet\rangle)$ so that the $g_n$ converge to $f$. In particular, for sufficiently large $n$, the straight-line interpolation from $f$ to $g_n$ defines a $\mathbb B$-homotopy in maps from $\mathrm{sd}^\ell\partial\Delta^k$ to $U(A\langle \Delta^\bullet\rangle)$, showing surjectivity of
\[
\mathrm{colim}_i U(A_i\langle\Delta^\bullet\rangle)\to U(A\langle\Delta^\bullet\rangle)
\]
on homotopy groups. For injectivity on homotopy groups, we argue similarly, using maps from $\mathrm{sd}^\ell(\Delta^1\times \partial\Delta^k)$ with fixed restriction to $\mathrm{sd}^\ell(\{0,1\}\times \partial\Delta^k)$.
\end{proof}

\begin{proposition}\label{prop:freemotivicsheaves} Let $X$ be a small arc-stack and let $U\subset \mathbb A^n_X$ be an open subset for some $n$. Then the free ball-invariant arc-sheaf on $U$ with values in $\mathcal D(\mathbb Z)$ is finitary, defining an object
\[
\mathbb Z_{\mathrm{mot}}[U]\in \mathcal D_{\mathrm{mot}}^{\mathrm{eff}}(X)
\]
representing the functor $M\mapsto M(U)$. Moreover, the formation of $\mathbb Z_{\mathrm{mot}}[U]$ commutes with base change and sends fibre products over $X$ to tensor products.
\end{proposition}

We will later give a different, more explicit description.

\begin{proof} This follows from the previous lemma by passing from free arc-sheaves of anima to free arc-sheaves with values in $\mathcal D(\mathbb Z)$.

For the final statement, we note that if $\mathcal D_{\mathbb B}(X_{\mathrm{arc}},\mathbb Z)\subset \mathcal D(X_{\mathrm{arc}},\mathbb Z)$ denotes the full subcategory of ball-invariant objects (where ball-invariance means that the values on $Y$ and $Y\times \mathbb B$ agree, for all $Y$ over $X$), then the left adjoint to the inclusion defines a symmetric monoidal functor commuting with all pullbacks. Indeed, $\mathcal D_{\mathbb B}(X_{\mathrm{arc}},\mathbb Z)$ is the symmetric monoidal Verdier quotient of $\mathcal D(X_{\mathrm{arc}},\mathbb Z)$ by the $\otimes$-ideal generated by the cone of $\mathbb Z[\mathbb B]\to Z\mathbb Z$; and the commutation with pullbacks is left adjoint to the observation that pushforward preserves ball-invariance.

Now the previous lemma says that when applied to $U$ as in the statement, this left adjoint takes values in $\mathcal D_{\mathrm{mot}}^{\mathrm{eff}}(X)\subset \mathcal D_{\mathbb B}(X_{\mathrm{arc}},\mathbb Z)$. This subcategory is stable under tensor products and pullbacks, giving the desired statement.
\end{proof}

\begin{proposition}\label{prop:freemotivicsheavesgenerate} Let $X=\mathcal M_{\mathrm{arc}}(A)$ be analytic and of finite cohomological dimension. Then $\mathcal D_{\mathrm{mot}}^{\mathrm{eff}}(X)$ is generated by the objects $\mathbb Z_{\mathrm{mot}}[U]$ where $U$ ranges over open subsets of finite-dimensional affine spaces over $X$.
\end{proposition}

\begin{proof} We have to see that if $M\in \mathcal D_{\mathrm{mot}}^{\mathrm{eff}}(X)$ and $M(U)=0$ for all such $U$, then $M=0$. Note first that the condition implies that for all closed bounded subsets $Z\subset \mathbb A^n_A$, one has $M(Z)=0$, by passing to filtered colimits over open neighborhoods $U$ of $Z$ (noting that their cohomological dimension is uniformly bounded). We will be applying this to the particular case where $Z$ is a point: This case means that for any Banach field $K$ with map $A\to K$ such that the corresponding map $K(x)\to K$ is topologically finitely generated, one has $M(K)=0$. Here $A\to K(x)$ is the point of $\mathcal M(A)$ over which $K$ factors. But then the same statement applies to any finite extension of $K$, and passing to the filtered colimit, also to its completed algebraic closure (as again, the cohomological dimension is bounded in this system). To see that $M=0$, it suffices to evaluate at complete algebraically closed Banach fields over $A$, but writing any such as a completed filtered colimits of ones of finite topological transcendence degree, we reduce to the assertion already established.
\end{proof}

\begin{theorem} Let $f: Y\to \mathbb A^n_X\to X$ be a composite of a finitary map and projection from a finite-dimensional affine space over $X$. Then
\[
f^\ast: \mathcal D_{\mathrm{mot}}^{\mathrm{eff}}(X)\to \mathcal D_{\mathrm{mot}}^{\mathrm{eff}}(Y)
\]
admits a left adjoint
\[
f_\sharp: \mathcal D_{\mathrm{mot}}^{\mathrm{eff}}(Y)\to \mathcal D_{\mathrm{mot}}^{\mathrm{eff}}(X)
\]
that commutes with base change and satisfies the projection formula.
\end{theorem}

\begin{proof} Writing $f$ as a composite and using Proposition~\ref{prop:finitaryleftadjoint}, we can assume that $Y=\mathbb A^n_X$. We can assume that $X$ is strictly totally disconnected; indeed, if we can prove the left adjoint exists in that case, and commutes with any base change of strictly totally disconnected spaces, then we get existence in general by hyperdescent (using \cite[Proposition 4.7.4.18]{LurieHA}). If $X$ is strictly totally disconnected, $\mathcal D_{\mathrm{mot}}^{\mathrm{eff}}(\mathbb A^n_X)$ is generated by $\mathbb Z_{\mathrm{mot}}[V_{/Y}]$ where $V\subset \mathbb A^m_Y$ is open. But then $f_\sharp(\mathbb Z_{\mathrm{mot}}[V_{/Y}]) = \mathbb Z_{\mathrm{mot}}[V_{/X}]$, where $V\subset \mathbb A^{n+m}_X$ is also open. This description commutes with base change, and satisfies the projection formula (where one has to check the latter only on generators $\mathbb Z_{\mathrm{mot}}[U_{/X}]$, where it follows from $\mathbb Z_{\mathrm{mot}}$ taking fibre products over $XC$ to tensor products).
\end{proof}

\subsection{The Tate twist}

We work in the nonarchimedean locus $|2|\leq 1$. Then the strong triangle inequality holds on all uniform Banach rings, and in particular for any $r\in \mathbb R$, there is a subgroup $A_{<r}\subset A$ of elements of norm $<r$. Similarly, if $r\leq 1$, then $1+A_{<r}\subset A^\times$ is stable under multiplication. We call $1+A_{<1}\subset A^\times$ the group of $1$-units, and denote the corresponding sheaf by $1+\mathcal O_{<1}$.

\begin{definition} The reduced group of units is
\[
\overline{\mathbb G}_m = \mathbb G_m / (1 + \mathcal O_{<1}),
\]
where the quotient is taken in arc-sheaves of abelian groups.
\end{definition}

\begin{theorem}\label{thm:reducedunitsballinvariant} The sheaf $\overline{\mathbb G}_m$ is finitary and ball-invariant.
\end{theorem}

\begin{proof} It is clear that it is finitary. To see that it is ball-invariant, take any non-discrete algebraically closed nonarchimedean Banach field $C$. We first argue modulo $\ell$ for some prime $\ell$. Note that raising to the power $\ell$ is surjective, with kernel the quotient of $\mu_{\ell}$ by the $1$-units inside $\mu_\ell$. If $\ell$ is equal to the residue characteristic of $C$, this sheaf is trivial; while for $\ell$ different from the residue characteristic, this is isomorphic to $\mathbb Z/\ell$. Now the claim follows from Proposition~\ref{prop:etalesheavesballinvariant}.

It remains to handle the situation after rationalization. In this case, there is no Galois cohomology. Thus, it reduces to a computation of the cohomology of the naively defined sheaf $\overline{\mathbb G}_m\otimes \mathbb Q$ (the sheafification of $A\mapsto A^\times/(1 + A_{<1})\otimes \mathbb Q$) on the topological space $\mathcal M(C\langle T\rangle_1)$. As this is a tree, there can only be cohomology in degree $0$ and $1$. Moreover, by right-exactness, the vanishing of $H^1$ reduces to the vanishing of $H^1$ with coefficients in $\mathbb G_m$. Any class in this $H^1$ can be realized by a \v{C}ech $1$-cocycle, which can be used to build an invertible module over $C\langle T\rangle_1$ using analytic descent of line bundles on rigid-analytic spaces. Now $\mathrm{Pic}(C\langle T\rangle_1)=0$ as $C\langle T\rangle_1$ is a principal ideal domain. Picking a global invertible section of this line bundle then trivializes the \v{C}ech cocycle, giving the vanishing of $H^1$. It remains to see that
\[
H^0(\mathcal M(C\langle T\rangle_1),\overline{\mathbb G}_m\otimes \mathbb Q)=C^{\times}/(1+C_{<1})\otimes \mathbb Q.
\]
For this, we note that there is a norm map
\[
\overline{\mathbb G}_m\otimes \mathbb Q\to \mathbb R_{>0}\xrightarrow{\mathrm{log}} \mathbb R
\]
of arc-sheaves (in particular, of sheaves on $\mathcal M(C\langle T\rangle_1)$), where $\mathbb R$ denotes the arc-sheaf taking any $A$ to the set of continuous maps from $\mathcal M(A)$ to $\mathbb R$. We claim that any global section on $\mathcal M(C\langle T\rangle_1)$ maps to a constant. Indeed, representing a section of $\overline{\mathbb G}_m$ locally by a section of $\mathbb G_m$, and approximating these invertible functions by a function of the form $\prod_i (T-a_i)^{k_i}$, the resulting map $\mathcal M(C\langle T\rangle_1)\to \mathbb R$ factors locally over some quotient
\[
\mathcal M(C\langle T\rangle_1)\to \Big\{(x_a)_{a\in A}\mid \forall a,b: x_a\leq \mathrm{max}(x_b,|a-b|), |a-b|\leq \mathrm{max}(x_a,x_b)\Big\}\subset \prod_{a\in A} [0,1]
\]
for some finite subset $A\subset \mathcal O_C$, and then by quasicompactness even globally. This gives a map from a tree to $\mathbb R$, and this map is piecewise a linear combination of $\mathrm{log}(x_a)=\mathrm{log} |T-a|$, and at each branch point except possibly the central Gau\ss\ point the sums of all slopes is equal to $0$ -- the slopes of the rays can locally be computed in terms of the valuations of elements in the residue fields of type (2) points, and the product formula for $\mathbb P^1$ then ensures the harmonicity. But on each ray $[0,1]$ of the tree, the map is constant near $0$, and hence it must be constant on the whole tree.

Rescaling, we can assume that it is equal to $0$. Looking at the Gau\ss\ point, or any type (2) point, we get an element of $k(T)^\times\otimes \mathbb Q$, where $k$ is the residue field of $C$; but it must have trivial valuations, as the norm map is constant (and the valuations describe the derivatives along the rays). This works for all valuations except one valuation for the central Gau\ss\ point, but by the product formula, the valuation there is also trivial. Hence we get an element of $k^\times\otimes \mathbb Q$. At points of other types, the residue field is unchanged, and hence we necessarily get an element of $k^\times\otimes \mathbb Q$. By finitaryness, the resulting element of $k^\times\otimes \mathbb Q$ is locally constant, and as the space is connected, it is constant.
\end{proof}

\begin{proposition} The tautological section of $\overline{\mathbb G}_m$ over $\mathbb G_m$ induces an isomorphism
\[
\mathbb Z_{\mathrm{mot}}[\mathbb G_m]/\mathbb Z_{\mathrm{mot}}[\ast]\cong \overline{\mathbb G}_m.
\]
\end{proposition}

\begin{proof} The left-hand side is equivalently
\[
\mathbb Z_{\mathrm{mot}}[\overline{\mathbb G}_m]/\mathbb Z_{\mathrm{mot}}[\ast],
\]
by ball-invariance. Now $\overline{\mathbb G}_m$ is already a finitary-sheaf, so $\mathbb Z_{\mathrm{mot}}[\overline{\mathbb G}_m]$ is obtained by taking the free $\mathbb Z$-module on the sheaf $\overline{\mathbb G}_m$, and making it ball-invariant. The free $\mathbb Z$-module is obtained by taking the free commutative monoid, and group-completing. The free commutative monoid on any sheaf $X$ is the disjoint union $\bigsqcup_{n\geq 0} \mathrm{Sym}^n(X)$ where $\mathrm{Sym}^n(X) = X^n/\Sigma_n$ is the (naive set-theoretic) quotient of $X^n$ by the action of the symmetric group $\Sigma_n$. Quotienting by $\mathbb Z_{\mathrm{mot}}[\ast]$, we get
\[
\mathbb Z_{\mathrm{mot}}[\overline{\mathbb G}_m]/\mathbb Z_{\mathrm{mot}}[\ast] = \mathrm{colim}_n \mathrm{Sym}^n(\overline{\mathbb G}_m)
\]
where the transition maps add $1\in \overline{\mathbb G}_m$. In our case we have, for $n\geq 1$, a commutative diagram
\[\xymatrix{
\mathrm{Sym}^n(\mathbb G_m)\ar[r]\ar[d] & \mathbb A^{n-1}\times \mathbb G_m\ar[d]\\
\mathrm{Sym}^n(\overline{\mathbb G}_m)\ar[r] & \overline{\mathbb G}_m
}\]
where the upper map is given by the characteristic polynomial and is an isomorphism of arc-sheaves. The left vertical and right vertical maps are (colimits of) $\mathbb B$-homotopy equivalences (for the left vertical map, write $\overline{\mathbb G}_m$ as the geometric realization $\mathbb G_m\times (1+\mathcal O_{<1})^\bullet$, all of whose terms are $\mathbb B$-homotopy equivalent to $\mathbb G_m$; then apply $\mathrm{Sym}^n$ everywhere); thus, the lower map is an isomorphism when mapping to ball-invariant sheaves, which gives the desired result.
\end{proof}

In the following definition, we allow archimedean $X$ again.

\begin{definition} The first Tate twist is
\[
\mathbb Z(1) = (\mathbb Z_{\mathrm{mot}}[\mathbb G_m]/\mathbb Z_{\mathrm{mot}}[\ast])[-1],
\]
i.e.~a shift of the reduced homology of $\mathbb G_m$.
\end{definition}

Over $\mathbb C$, this is isomorphic to $\mathbb Z$ (over $\mathbb R$, complex conjugation acts nontrivially).

\begin{proposition}\label{prop:tatetwistcompact} If $X$ is strictly totally disconnected, the object $\mathbb Z(1)\in \mathcal D_{\mathrm{mot}}^{\mathrm{eff}}(X)$ is compact. For any open annulus $\mathbb T=\mathbb T(r_1,r_2)\subset \mathbb G_m$ (with $0<r_1<r_2<\infty$), the map
\[
\mathbb Z_{\mathrm{mot}}[\mathbb T]\to \mathbb Z_{\mathrm{mot}}[\mathbb G_m]
\]
is an isomorphism.
\end{proposition}

\begin{proof} We first prove that
\[
\mathbb Z_{\mathrm{mot}}[\mathbb T]\to \mathbb Z_{\mathrm{mot}}[\mathbb G_m]
\]
is an isomorphism. By rescaling, assume $r_1<1<r_2$, so $1\in \mathbb T$ and we can equivalently prove that
\[
\mathbb Z_{\mathrm{mot}}[\mathbb T]/\mathbb Z_{\mathrm{mot}}[\ast]\to \mathbb Z_{\mathrm{mot}}[\mathbb G_m]/\mathbb Z_{\mathrm{mot}}[\ast]
\]
is an isomorphism. This can be checked in case $X=\mathcal M_{\mathrm{arc}}(C)$ for a non-discrete algebraically closed Banach field $C$. It is clear when $C=\mathbb C$, so we can assume that $C$ is nonarchimedean. Then it follows from the argument of the previous proposition: We need to see that in the colimit over $n$, the maps of symmetric powers
\[
\mathrm{Sym}^n(\mathbb T)\subset \mathrm{Sym}^n(\mathbb G_m)\cong \mathbb A^{n-1}\times \mathbb G_m
\]
become isomorphisms when mapping into ball-invariant sheaves. But the projection to $\mathbb G_m$ yields a surjective map
\[
\mathrm{Sym}^n(\mathbb T)\to \mathbb T(r_1^n,r_2^n)
\]
whose fibres are open balls: Indeed, a polynomial
\[
X^n+a_{n-1}X^{n-1}+\ldots+a_1X+a_0
\]
with given $a_0\in \mathbb T(r_1^n,r_2^n)$ has all zeros in $\mathbb T$ if and only if all slopes of the Newton polygon defined by $a_{n-1},\ldots,a_1,a_0$ lie in the interval $(r_1,r_2)$. This means
\[
|a_i|<\mathrm{max}(r_2^{n-i},|a_0|r_1^{-i})
\]
for all $i=1,\ldots,n-1$, which is indeed an open ball. Thus, when mapping to ball-invariant sheaves, the map $\mathrm{Sym}^n(\mathbb T)\to \mathbb T(r_1^n,r_2^n)$ becomes an isomorphism. In the colimit over $n$, the $\mathbb T(r_1^n,r_2^n)$ become all of $\mathbb G_m$.

In particular, the functor $\mathrm{Hom}(\mathbb Z_{\mathrm{mot}}[\mathbb G_m],-)$ agrees with $\mathrm{Hom}(\mathbb Z_{\mathrm{mot}}[\mathbb T],-)$, i.e.~evaluation at $\mathbb T$, for all annuli $\mathbb T$. Writing a closed annulus as an intersection of open annuli, and using finitaryness, it follows that it also agrees with evaluation at a(ny) closed annulus. But a closed annulus has finite cohomological dimension and is quasicompact, so evaluation commutes with all colimits, so $\mathbb Z_{\mathrm{mot}}[\mathbb G_m]$ and hence $\mathbb Z(1)$ is compact.
\end{proof}

\section{Free motivic sheaves}

In the last section, we proved the existence of free motivic sheaves $\mathbb Z_{\mathrm{mot}}[U]$ where $U$ is an open subspace of affine space $\mathbb A^n_X$. Our first goal is to make these more explicit. We will restrict attention to the nonarchimedean locus. In that case, we have seen that these free motivic sheaves can be understood in terms of those of $U/\mathbb D^n(0,<r)$ for some $r>0$, which are even finitary over $X$.

\begin{theorem}\label{thm:freemotivicsheaves} Let $X$ be a small arc-stack and let $U$ be a finitary arc-sheaf of sets over $X$. Let $\mathbb N[U]$ be the free arc-sheaf of commutative monoids on $U$. Then, on strictly totally disconnected $A$ over $X$, and after inverting all primes $p$ such that $|p|<1$ somewhere on $\mathcal M(A)$, the $A$-valued points of the free motivic sheaf $\mathbb Z_{\mathrm{mot}}[U]$ are given by the group completion of
\[
L_{\mathbb B}^{\mathrm{pre}}(\mathbb N[U])(A).
\]
Moreover:
\begin{enumerate}
\item[{\rm (i)}] After rationalization, this holds true also for the values on any affine subset of $\mathbb A^1_A$, i.e.~any arc-subsheaf of the form $\mathcal M_{\mathrm{arc}}(B)$ for some Banach $A$-algebra $B$. 
\item[{\rm (ii)}] After reduction modulo $n$ for some $n$ with $|n|\geq 1$ everywhere on $X$, the sheaf $\mathbb Z_{\mathrm{mot}}[U]/n$ is invariant under change of algebraically closed field.
\end{enumerate}
\end{theorem}

\begin{proof} We can assume that $X$ is strictly totally disconnected. We need to see that the arc-sheaf associated to the finitary functor sending strictly totally disconnected $A$ over $X$ to the group completion of $(L_{\mathbb B}^{\mathrm{pre}}(\mathbb N[U]))(A)$ is ball-invariant. Note that all homology presheaves of the functor taking any Banach algebra $A$ over $X$ to the group completion of $(L_{\mathbb B}^{\mathrm{pre}}(\mathbb N[U]))(A)$ are finitary ball-invariant presheaves with transfer, as defined below. We are reduced to Theorem~\ref{thm:ballinvariantpresheaftransferarcsheafification} below.
\end{proof}

\begin{definition} Let $X$ be strictly totally disconnected. Endow the category of uniform Banach algebras $A$ over $X$ with a new category structure, where maps from $A$ to $B$ are given by maps of arc-sheaves of commutative monoids $\mathbb N[\mathcal M_{\mathrm{arc}}(A)_{/X}]\to \mathbb N[\mathcal M_{\mathrm{arc}}(B)_{/X}]$. A presheaf $M$ of abelian groups with transfer over $X$ is an additive functor $A\mapsto M(A)$ from this category towards abelian groups. The presheaf $M$ is finitary if it commutes with all filtered colimits of Banach algebras over $X$.
\end{definition}

\begin{theorem}\label{thm:ballinvariantpresheaftransferarcsheafification} Let $M$ be a finitary ball-invariant presheaf of abelian groups with transfer over a strictly totally disconnected $X$. Then the derived arc-sheafification of $M$ is also ball-invariant, after inverting all primes $p$ such that $|p|<1$ somewhere on $X$. After reduction modulo $n$ for some $n\geq 1$, it is invariant under change of algebraically closed field.

Moreover, the presheaf $M\otimes \mathbb Q$ is already a sheaf on the site of open subsets of $\mathbb A^1_A$ for strictly totally disconnected $A$, and the value of $M\otimes \mathbb Q$ on affine subsets of $\mathbb A^1_A$ is unchanged under arc-sheafification.
\end{theorem}

\begin{proof} The derived arc-sheafification can be computed by first restricting to strictly totally disconnected Banach rings $A$ over $X$, and then using Theorem~\ref{theorem:finitaryarc} and pro-\'etale sheafification. This is still finitary by Lemma~\ref{lem:quasiproetalecohomfiltcolim}. The ball-invariance is a statement about values on algebraically closed nonarchimedean Banach fields $C$, so we can assume that $X=\mathcal M_{\mathrm{arc}}(C)$. We can consider separately the cases where $M$ is torsion, and where $M$ is rational. In the torsion case, we have to see that $M$ is invariant under change of algebraically closed field. Thus, let $C'/C$ be an extension of algebraically closed Banach fields. By induction, we can assume that $C'$ has topological transcendence degree $1$ over $C$. Then $\mathcal M_{\mathrm{arc}}(C')$ can be written as a cofiltered limit of $U_i$ where $U_i/C$ are connected smooth Berkovich curves. By finitaryness, any $C'$-point is induced from some $U_i$-point; as $U_i$ has a $C$-point $x\in U_i$, this already implies injectivity of $M(C)\to M(C')$. For surjectivity, consider the $C$-point $x\in U_i$ just chosen as well as the tautological $C'$-point $x'$ of $U_i$. It suffices to see that the map $M(U_i)\to M(C')$ induced by $x'$ agrees with the composite map $M(U_i)\to M(C)\hookrightarrow M(C')$ induced by $x$. This follows from Lemma~\ref{lem:dividebypcurve}.

It remains to handle the case that $M$ is rational. Then $M$ satisfies finite \'etale descent as the transfers yield trace maps. It then suffices to see that the functor taking any open subset $U\subset \mathbb A^1_C$ to $M(U)$ is already a sheaf of complexes on the underlying Berkovich space of $\mathbb A^1_C$. As the Berkovich space of $\mathbb A^1_C$ is a tree, it suffices to see that for all connected open subsets $U\subset \mathbb A^1_C$ covered by two connected open subsets $U_1,U_2\subset U$ whose intersection $V=U_1\cap U_2$ is an open annulus, the sequence
\[
0\to M(U)\to M(U_1)\oplus M(U_2)\to M(V)\to 0
\]
is exact. By finitaryness, and the structure of $\mathbb A^1_C$, we can in fact assume that the complements of each of $U$, $U_1$, $U_2$ and $V$ in $\mathbb P^1_A$ is a finite disjoint union of closed discs. Let $\mathbb Q_{\mathrm{tr},\mathbb B}[U]$ denote the free ball-invariant presheaf of $\mathbb Q$-vector spaces with transfer on $U$; explicitly, this is the rationalization of (the group completion of) $\pi_0 L_{\mathbb B}^{\mathrm{pre}} \mathbb N[U]$. It suffices to see that
\[
0\to \mathbb Q_{\mathrm{tr},\mathbb B}[V]\to \mathbb Q_{\mathrm{tr},\mathbb B}[U_1]\oplus \mathbb Q_{\mathrm{tr},\mathbb B}[U_2]\to \mathbb Q_{\mathrm{tr},\mathbb B}[U]\to 0
\]
is a split short exact sequence in the category of ball-invariant presheaves of $\mathbb Q$-vector spaces with transfer. For this claim, which amounts only to the existence of certain maps satisfying certain identities, it is sufficient to understand the very small part of the category of ball-invariant presheaves of abelian groups with transfer that is spanned by the four objects
\[
\mathbb Q_{\mathrm{tr},\mathbb B}[V], \mathbb Q_{\mathrm{tr},\mathbb B}[U_1], \mathbb Q_{\mathrm{tr},\mathbb B}[U_2], \mathbb Q_{\mathrm{tr},\mathbb B}[U].
\]
By Lemma~\ref{lem:toycase} all maps between those are the same as the maps between their $\mathbb Q_{\mathrm{mot}}$-versions. But for the $\mathbb Q_{\mathrm{mot}}$-version, it is clear that one has an exact triangle, and in fact it is split as one of $U_1$ and $U_2$ is an open subset of an annulus $V'$ and the resulting map $\mathbb Q_{\mathrm{tr},\mathbb B}[V]\to \mathbb Q_{\mathrm{tr},\mathbb B}[V']$ is an isomorphism (again by comparison to $\mathbb Q_{\mathrm{mot}}$).
\end{proof}

\begin{lemma}\label{lem:dividebypcurve} Let $C$ be a non-discrete algebraically closed nonarchimedean field and let $U$ be a smooth connected curve over $C$. Let $p$ be a prime. For any two points $x,x'\in U(C)$ one can find some integer $n$ and points $x_1,\ldots,x_n\in U(C)$ and $x_1',\ldots,x_n'\in U(C)$ such that the points
\[
x+px_1+\ldots+px_n,x'+px_1'+\ldots+px_n'\in \mathbb N[U]
\]
are $\mathbb B$-homotopic.
\end{lemma}

\begin{proof} Potentially after shrinking $U$ (still containing $x$ and $x'$) we can find a smooth projective compactification $X$ of $U$ such that $Z=X\setminus U$ is a finite union of closed discs, see \cite[Corollary 1.3]{vanderPut}. Choose one point $z_i\in Z$ in each of those boundary discs, and consider the moduli space of degree $0$ line bundles $\mathcal L$ on $X$ with trivializations at all $z_i$. This is a semiabelian variety, and in particular multiplication by $p$ is surjective. We can then divide $\mathcal O([x]-[x'])$ by $p$ to get some line bundle $\mathcal L$, trivialized at the $z_i$, with $\mathcal L^{\otimes p}\cong \mathcal O([x]-[x'])$, compatibly with the trivializations at the $z_i$. We can then find a rational section $s$ of $\mathcal L$ without zeroes or poles on the boundary discs, and equal to $1$ at the $z_i$. This gives an isomorphism $\mathcal L\cong \mathcal O(-D)$ for some divisor $D$ (necessarily of degree $0$) supported on $U$, so we can write
\[
D=[x_1]+\ldots+[x_n]-[x_1']-\ldots-[x_n']
\]
with $x_i,x_i'\in U(C)$. We get an isomorphism $\mathcal O(-pD)\cong \mathcal O([x]-[x'])$, or, rearranging terms,
\[
\mathcal O([x]+p[x_1]+\ldots+p[x_n])\cong \mathcal O([x']+p[x_1']+\ldots+p[x_n'])
\]
that is moreover the identity at the fibres at $z_i$. Let $s_1$ and $s_2$ be the corresponding tautological sections of this line bundle; both are equal to $1$ at $z_i$. Taking a variable $t\in \mathbb B$, we get the section $ts_1 + (1-t) s_2$ which for $t=0$ is $s_2$ but for $t=1$ is $s_1$. The value at $z_i$ is constant $1$, and this implies that these sections have no zeroes on the boundary discs (as $\frac{s_1}{s_2}$ is an invertible function on the disc and equal to $1$ at $z_i$, thus $|\frac{s_1}{s_2}-1|<1$ and therefore $1 + t(\frac{s_1}{s_2}-1) = \frac{ts_1+(1-t)s_2}{s_2}$ is a $1$-unit), so their divisors are entirely supported on $U\subset X$. This gives the desired homotopy.
\end{proof}

The following lemma, which in the end is just a special case of the full theorem, was used in the proof.

\begin{lemma}\label{lem:toycase} Let $C$ be a non-discrete algebraically closed nonarchimedean Banach field, and let $U,V\subset \mathbb A^1_C$ be two open subsets whose complements in $\mathbb P^1_A$ are finite disjoint unions of closed discs. Then the value of $\mathbb Q_{\mathrm{mot}}[U]$ on $V$ is concentrated in degree $0$ and given by the cokernel of
\[
(\mathbb N[U])(V\times \mathbb B)\otimes \mathbb Q\xrightarrow{i_0^\ast-i_1^\ast}(\mathbb N[U])(V)\otimes \mathbb Q.
\]
\end{lemma}

\begin{proof} Assume that the complement of $U$ in $\mathbb P^1_A$ is the union of closed discs $Z_0,\ldots,Z_d$, where $Z_0$ is a disc around $\infty$. Then $\mathbb Q_{\mathrm{mot}}[U]\cong \mathbb Q\oplus \overline{\mathbb G}_m^d$. We may similarly describe $\mathbb Q_{\mathrm{mot}}[V]$, invoking $e$ finite boundary discs. There are no maps $\overline{\mathbb G}_m\to \mathbb Q$ as those correspond to the reduced $\mathbb Q$-cohomology of $\mathbb P^1$ which vanishes. The endomorphisms of $\overline{\mathbb G}_m\otimes \mathbb Q$ are given by $\mathbb Q$, concentrated in degree $0$ (they can be computed using the value of $\overline{\mathbb G}_m$ at a type (3) point, which we may assume exists by reducing to the case that $C$ is separable). We see that the maps between $\mathbb Q_{\mathrm{mot}}[U]$ and $\mathbb Q_{\mathrm{mot}}[V]$ are concentrated in degree $0$, and have a very explicit description as
\[
\mathbb Q\oplus (C^\times/(1+C_{<1}))^d\otimes \mathbb Q\oplus \mathbb Q^{de}.
\]
Concretely, the first two components are equivalent to a map $\mathbb Q\to \mathbb Q_{\mathrm{mot}}[U]$, corresponding to restriction along a chosen point $v\in V$; and the component $\mathbb Q^{de}$ measures degrees of maps on boundary tori.

Now we first prove surjectivity of the map
\[
(\mathbb N[U])(V)\otimes \mathbb Q\to \mathbb Q\oplus (C^\times/(1+C_{<1}))^d\otimes \mathbb Q\oplus \mathbb Q^{de}.
\]
It suffices to do this in case $e=1$: In general write $V=V_1\cap\ldots\cap V_e$ and use the commutative diagram
\[\xymatrix{
\bigoplus_i (\mathbb N[U])(V_i)\otimes \mathbb Q\ar[d]\ar[r] & \bigoplus_i (\mathbb Q\oplus (C^\times/(1+C_{<1}))^d\otimes \mathbb Q\oplus \mathbb Q^e) \ar[d] \\
(\mathbb N[U])(V)\otimes \mathbb Q \ar[r] & \mathbb Q\oplus (C^\times/(1+C_{<1}))^d\otimes \mathbb Q\oplus \mathbb Q^{de}
}\]
where the right-vertical map is surjective. Thus, shifting and rescaling $V$ if necessary, we can assume $V=\mathbb T(s_1,s_2)$ for some $0<s_1<1<s_2$. Similarly, we can assume that
\[
U=\mathbb D(0,<r_2)\setminus \bigcup_{i=1}^d \mathbb D(a_i,\leq r_1)
\]
for some $r_2>1>r_1$ and $a_i\in \mathcal O_C$, pairwise distinct in the residue field $k$. In general, there is an open immersion $\mathrm{Sym}^n(U)\subset \mathbb A^n_C$ given by the characteristic polynomial. We now claim that for any sufficiently large integers $k_1,\ldots,k_d\geq 1$ the map
\[
V\to \mathbb A^n_C: t\mapsto \prod_{i=1}^d (T-a_i)^{k_i} - t
\]
factors over $\mathrm{Sym}^n(U)$, where $n=k_1+\ldots+k_d$. Concretely, this means that any $T\in C$ satisfying
\[
\prod_{i=1}^d (T-a_i)^{k_i} = t
\]
must lie in $U$. Now if we had $|T|\geq r_1>1$, then as all $|a_i|\leq 1$, the left-hand side has absolute value $r_2^n$, which is $>s_2>|t|$ once $n$ is sufficiently large. On the other hand, if for some $i$ we have $|T-a_i|\leq r_1$, then for all other $j$ we have $|T-a_i|=|a_i|\leq 1$, and the absolute value of the left-hand side is at most $r_1^n$, which is $<s_1<|t|$ once $n$ is sufficiently large. Thus indeed, the map factors over $\mathrm{Sym}^n(U)$. This means that the composite map
\[
(\mathbb N[U])(V)\otimes \mathbb Q\to \mathbb Q\oplus (C^\times/(1+C_{<1}))^d\otimes \mathbb Q\oplus \mathbb Q^d\to \mathbb Q^d
\]
has $(k_1,\ldots,k_d)$ in its image, for all sufficiently large $k_1,\ldots,k_d$. This shows that this composite map is surjective. Now it suffices to see that the map
\[
(\mathbb N[U])(\ast)\otimes \mathbb Q\to \mathbb Q\oplus (C^\times/(1+C_{<1}))^d\otimes \mathbb Q
\]
is surjective. But this takes an element $T$ of $U(C)$ to $(1,(T-a_i)_i)$. Pick two elements $T_1,T_2$ with $|T_1-a_i|,|T_2-a_i|<1$ for some $i$. Then for $j\neq i$, $\frac{T_1-a_j}{T_2-a_j} = 1 + \frac{T_2-T_1}{T_2-a_j}$ is a $1$-unit, so the difference of the images of $T_1$ and $T_2$ is nonzero only in the $i$-th component, where it is the image of $\frac{T_1-a_i}{T_2-a_i}$. This shows that
\[
0\oplus (C^\times/(1+C_{<1}))^d\otimes \mathbb Q\subset \mathbb Q\to \mathbb Q\oplus (C^\times/(1+C_{<1}))^d\otimes \mathbb Q
\]
is in the image, but then the image is everything.

Finally, we need to prove injectivity. Thus consider two maps $V\to \mathrm{Sym}^n(U)$ yielding the same element of
\[
\mathbb Q\oplus (C^\times/(1+C_{<1}))^d\otimes \mathbb Q\oplus \mathbb Q^{de}.
\]
We regard them as maps $f_0,f_1: V\to \mathrm{Sym}^n(U)\subset \mathbb A^n_C$, and interpolate to a family of maps
\[
(f_t)_{t\in \mathbb B}: V\times \mathbb B\to \mathbb A^n_C
\]
via $f_t = (1-t)f_0 + tf_1$. We claim that this factors over $\mathrm{Sym}^n(U)\subset \mathbb A^n_C$. This can be checked pointwise, so we reduce to the case that $V$ is a point. Writing $U$ as an intersection (and translating it), we can also assume that $U = \mathbb T(r_1,r_2)$ is an annulus. Then it follows from the observation, in the proof of Proposition~\ref{prop:tatetwistcompact}, that
\[
\mathrm{Sym}^n(U)\to \mathbb T(r_1^n,r_2^n)
\]
is a fibration in open balls. Here, the assumption that $f_0$ and $f_1$ induce the same map on $\mathbb Q_{\mathrm{mot}}[-]$ ensures that they yield the same element of $\mathbb T(r_1^n,r_2^n)$ up to $1$-units.
\end{proof}

\begin{corollary} Let $n$ be an integer such that $|n|\geq 1$ everywhere on $X$, and let $M\in \mathcal D_{\mathrm{fin}}(X,\mathbb Z/n\mathbb Z)$. Then $M$ is ball-invariant if and only if $M$ is invariant under change of algebraically closed field.
\end{corollary}

\begin{proof} We have already seen that if $M$ is invariant under change of algebraically closed field, then $M$ is ball-invariant. For the converse, we can assume that $X$ is strictly totally disconnected. Using Proposition~\ref{prop:freemotivicsheavesgenerate}, it suffices to see that for all open $U\subset \mathbb A^m_X$, the sheaf $\mathbb Z_{\mathrm{mot}}[U]/n$ is invariant under change of algebraically closed field. This follows from Theorem~\ref{thm:freemotivicsheaves}.
\end{proof}

Summarizing, we have the following theorem.

\begin{theorem} Let $p$ be a prime and let $A$ be a Banach ring. If $|p|_A<1$, then $p$ is invertible in $\mathcal D_{\mathrm{mot}}^{\mathrm{eff}}(\mathcal M(A))$. If $|p|\geq 1$ everywhere on $\mathcal M(A)$, then the subcategory of $p$-power torsion objects on $\mathcal D_{\mathrm{mot}}^{\mathrm{eff}}(\mathcal M(A))$ is equivalent to the left-completed derived category of $p$-power torsion \'etale sheaves.
\end{theorem}

\begin{proof} For the first part, we can assume that $A$ is an algebraically closed Banach field. It is then perfectoid, and the arc-site is invariant under tilting, cf.~\cite[Theorem 3.12, 3.13]{ECoD}; moreover, ball-invariance over $A$ implies invariance under the perfectoid ball by passing to filtered colimits, and hence ball-invariance on the tilt. (The converse is also true, see below.) Thus, we can assume that $A$ is of characteristic $p$. Here, the Artin-Schreier sequence $0\to \mathbb F_p\to \mathbb B\xrightarrow{x^p-x} \mathbb B\to 0$ shows that $\mathbb F_p=0$ in ball-invariant sheaves.

Now assume that $|p|\geq 1$ everywhere on $\mathcal M(A)$. By descent, we can reduce to the case that $A$ is strictly totally disconnected, in which case the subcategory of $p$-power torsion objects of $\mathcal D_{\mathrm{mot}}^{\mathrm{eff}}(\mathcal M(A))$ is equivalent to $p$-power torsion sheaves on the profinite set $\mathcal M(A)$, by the previous corollary. Those are also the \'etale sheaves, as desired.
\end{proof}

\begin{proposition}\label{prop:effectivetilting} Let $p$ be a prime and let $X$ be a small arc-stack with $|p|<1$ everywhere on $X$, with tilt $X^\flat$. Then tilting induces an equivalence
\[
\mathcal D_{\mathrm{mot}}^{\mathrm{eff}}(X)\cong \mathcal D_{\mathrm{mot}}^{\mathrm{eff}}(X^\flat)
\]
where both categories agree with their $\mathbb Z[\tfrac 1p]$-linear versions.
\end{proposition}

The invariance of rigid motives under tilting is originally due to Vezzani \cite{Vezzani}.

\begin{proof} The arc-sites are equivalent under tilting; thus, the only difference is a priori the condition of ball-invariance. As noted above, ball-invariance over $X$ implies invariance under the perfectoid ball, which is ball-invariance over $X^\flat$. Conversely, note that the map $f: \mathbb B_X\to \mathbb B_X$ sending the coordinate $T$ to $T^p$ induces a split injection on sections, using the trace map; more precisely, $\mathbb Z[\tfrac 1p]\to f_\ast \mathbb Z[\tfrac 1p]$ has a section (and $p$ is invertible automatically, as noted above), given by the obvious trace map divided by $p$. Passing up the perfectoid tower, one gets a sequential colimit of split injections which is itself a split injection; this then implies that perfectoid ball-invariance implies usual ball-invariance.
\end{proof}

As another corollary, we can now prove that even with integral coefficients, $\mathcal D_{\mathrm{mot}}^{\mathrm{eff}}\subset \mathcal D_{\mathrm{fin}}$ inherits a $t$-structure (the homotopy $t$-structure).

\begin{corollary}\label{cor:homotopytstructure} For any small arc-stack $X$, the subcategory
\[
\mathcal D_{\mathrm{mot}}^{\mathrm{eff}}(X)\subset \mathcal D_{\mathrm{fin}}(X,\mathbb Z)
\]
is stable under truncations.
\end{corollary}

\begin{proof} Take any $M\in \mathcal D_{\mathrm{mot}}^{\mathrm{eff}}(X)$. We have a cofiber sequence
\[
M\to M\otimes \mathbb Q\to M\otimes \mathbb Q/\mathbb Z.
\]
We already know that all truncations of $M\otimes \mathbb Q$ and $M\otimes \mathbb Q/\mathbb Z$ remain in $\mathcal D_{\mathrm{mot}}^{\mathrm{eff}}(X)$. To conclude, it suffices to control the transition map. Thus, take any $N,N'\in \mathcal D_{\mathrm{mot}}^{\mathrm{eff}}(X)$ that are concentrated in degree $0$ and such that $N$ is rational while $N'$ is torsion. We need to see that the image of any map $f: N\to N'$ is in $\mathcal D_{\mathrm{mot}}^{\mathrm{eff}}(X)$. As this image is necessarily torsion, we have to see invariance under change of algebraically closed field. Thus let $C\subset C'$ be an inclusion of non-discrete algebraically closed Banach fields over $X$. We want to see that the image of $N(C)\to N'(C)$ maps isomorphically to the image of $N(C')\to N'(C')$, noting that we know $N'(C)=N'(C')$. Thus, the map is necessarily injective. But if a class $\alpha\in N'(C)$ lies in the image of $N(C')\to N'(C')=N'(C)$, then by finitaryness there is some Banach-$C$-algebra $R$ of topologically finite type such that the image of $\alpha$ under $N'(C)\to N'(R)$ lies in the image of $N(R)\to N'(R)$. But then there is a section $R\to C$, showing that $\alpha$ is already in the image of $N(C)$.
\end{proof}

\section{The Cancellation Theorem}
 
The goal of this section is to prove the following result, due to Voevodsky \cite{VoevodskyCancellation} for algebraic motives over a field. Ayoub has previously proved this for rigid-analytic motives \cite{AyoubRigid}, for a version of rigid-analytic motives with transfer, and Vezzani \cite{VezzaniTransfer} has proved that the versions with and without transfers agree in the setting relevant to us.

\begin{theorem}\label{thm:cancellation} For any small arc-stack $X$, the functor
\[
-\otimes \mathbb Z(1): \mathcal D_{\mathrm{mot}}^{\mathrm{eff}}(X)\to \mathcal D_{\mathrm{mot}}^{\mathrm{eff}}(X)
\]
is fully faithful.
\end{theorem}

\begin{remark} Classically, this is only true over a field. In our setting, it is true over any base.
\end{remark}

\begin{proof} We have to show that for all $M\in \mathcal D_{\mathrm{mot}}^{\mathrm{eff}}(X)$, the map
\[
M\to \intHom(\mathbb Z(1),M\otimes \mathbb Z(1))
\]
is an isomorphism. The internal Hom here is taken in $\mathcal D_{\mathrm{mot}}^{\mathrm{eff}}(X)$, but it agrees with the internal Hom in the larger category $\mathcal D(X_{\mathrm{arc}},\mathbb Z)$. Indeed, rewriting $\mathbb Z(1)$ in terms of $\mathbb Z_{\mathrm{mot}}[\mathbb T]/\mathbb Z$ for an annulus $\mathbb T$, and hence the internal Hom in terms of pushforward along $\mathbb T\to \ast$, this follows from Theorem~\ref{thm:properpushforwardfinitary}. In particular, it now suffices to prove that for all strictly totally disconnected $X'$ over $X$, the map
\[
M(X')\to \mathrm{Hom}_{\mathcal D_{\mathrm{mot}}^{\mathrm{eff}}(X')}(\mathbb Z(1),M|_{X'}\otimes \mathbb Z(1))
\]
is an isomorphism. Both sides here are naturally global sections of a sheaf on the profinite set $\mathcal M(X')$, and as $\mathbb Z(1)$ is compact, we can even reduce to the case that $X'$ (and then also $X$) is a geometric point.

Thus, we are reduced to the case of algebraically closed Banach fields $X=\mathcal M_{\mathrm{arc}}(C)$. The case of $C=\mathbb C$ is clear, so we can assume that $C$ is nonarchimedean. Also, the result is clear with torsion coefficients, so we will work with $\mathbb Q$-coefficients below.

Reducing to generators, we have to prove that for any finitary sheaf of sets $U$ over $C$, the map
\[
\mathbb Q_{\mathrm{mot}}[U]\to \intHom(\mathbb Q_{\mathrm{mot}}[(\mathbb G_m,\ast)],\mathbb Q_{\mathrm{mot}}[U\times (\overline{\mathbb G}_m,\ast)]).
\]
is an isomorphism. Here $U\times (\overline{\mathbb G}_m,\ast) = U\times\overline{\mathbb G}_m/U\times \ast$ denotes the product in pointed sheaves.

The key of the proof is to construct, for all $i\geq 0$, a map
\[
\pi_i \intHom(\mathbb Q_{\mathrm{mot}}[\mathbb G_m],\mathbb Q_{\mathrm{mot}}[U\times \overline{\mathbb G}_m])\to \pi_i \mathbb Q_{\mathrm{mot}}[U]
\]
splitting the previous displayed map on $\pi_i$. Actually, we will rather have a map from the source to eventually defined sequences in the target, i.e.~the colimit over all $N$ of $\mathrm{Map}(\mathbb N_{\geq N},-)$. Still, it will be a splitting in the sense that the composite map is the map taking an element to the constant sequence.

By Theorem~\ref{thm:freemotivicsheaves}, the left-hand side can be computed as the $i$-th homotopy group of the rationalization of the animated commutative monoid
\[
\mathrm{colim}_{[n]\in \Delta^{\mathrm{op}}} \mathbb N[U\times \overline{\mathbb G}_m](\mathbb T\times \Delta^n_C)
\]
where $\mathbb T\subset \mathbb G_m$ denotes a closed annulus, say $\tfrac 12\leq |T|\leq 2$. By using simplicial subdivision $\mathrm{sd}$, more precisely Kan's $\mathrm{Ex}^\infty$-functor, any class in there can be represented by a map
\[
\mathbb T\times \mathrm{sd}^\ell\partial\Delta^{i+1}_C\to \mathrm{Sym}^k(U\times \overline{\mathbb G}_m)
\]
for some $k,\ell\geq 1$; and two such maps yield the same element if, possibly after refinement and addition of some section $\mathbb T\to \mathrm{Sym}^{k'}(U\times \overline{\mathbb G}_m)$, they admit a homotopy from $\mathbb T\times \mathrm{sd}^{\ell'}\Delta^{i+1}_C\times \mathbb B$.

Our goal is to construct, for all sufficiently large integers $N$, two spaces
\[
Y_1,Y_2\to \mathrm{sd}^\ell\partial\Delta^{i+1}_C
\]
both of which are fibrations in open balls over \'etale maps $\overline{Y}_1,\overline{Y}_2\to \mathrm{sd}^\ell\partial\Delta^{i+1}_C$ equipped with multiplicities, and maps
\[
Y_1,Y_2\to \mathbb N[U].
\]
After composing to $\mathbb Z_{\mathrm{mot}}[U]$, one can replace $Y_i$ by $\overline{Y}_i$, and then sum with multiplicities over fibres of the \'etale map to $\mathrm{sd}^\ell\partial\Delta^{i+1}_C$, yielding two maps
\[
\mathrm{sd}^\ell\partial\Delta^{i+1}_C\to \mathbb Z_{\mathrm{mot}}[U].
\]
As the source is $\mathbb B$-homotopic to the $i$-dimensional sphere as an anima, these are classes in $\pi_i \mathbb Z_{\mathrm{mot}}[U]$. The desired map is given by the difference of these two maps. We will show that this is invariant under homotopies, yielding the desired map to eventually defined sequences in $\pi_i \mathbb Z_{\mathrm{mot}}[U]$.

The construction of $Y_1$ and $Y_2$ (and the multiplicities and maps to $\mathbb N[U]$) will be arc-local on $\mathrm{sd}^\ell \partial\Delta^{i+1}_C$; thus we can replace it by $\mathcal M_{\mathrm{arc}}(A)$ for some strictly totally disconnected $A$. Once one has constructed everything in this generality, the agreement on homotopic maps is immediate. Thus, our task is to start with a map
\[
\mathbb T\to \mathrm{Sym}^k(U\times \overline{\mathbb G}_m)
\]
and construct, for all sufficiently large integers $N$, spaces $Y_1$ and $Y_2$, with a fibration in open balls over an \'etale map to $A$ equipped with multiplicities, and maps $Y_1,Y_2\to \mathbb N[U]$.

For $Y_1$, we simply take the disjoint union of the $N$ open balls of radius $1$ around the $N$-th roots of unity in $\mathbb T$. This comes with the map
\[
Y_1\to \mathbb T\to \mathrm{Sym}^k(U\times \overline{\mathbb G}_m)\to \mathrm{Sym}^k(U)\to \mathbb N[U].
\]

For the construction of $Y_2$, we need only the map $\mathbb T\to \mathrm{Sym}^k(\overline{\mathbb G}_m)$. This is given by a subspace $\tilde{\mathbb T}\subset \mathbb T\times \overline{\mathbb G}_m$ (fibrewise over $\mathbb T$, the union of the $k$ points parametrized by $\mathbb T\to \mathrm{Sym}^k(\overline{\mathbb G}_m)$) together with multiplicities of the fibres of $\tilde{\mathbb T}\to \mathbb T$. Note that $\tilde{\mathbb T}\to \mathbb T$ is \'etale, i.e.~finitary with finite geometric fibres (but not separated in general -- points can collide, in which case multiplicities add). Let $p_1: \tilde{\mathbb T}\to \mathbb T\to \overline{\mathbb G}_m$ and $p_2: \tilde{\mathbb T}\to \overline{\mathbb G}_m$ denote the two natural projections to $\overline{\mathbb G}_m$. We let
\[
Y_2\subset \tilde{\mathbb T}
\]
be the subspace where $p_1^N=p_2$. Note that on $Y_2$, the projection $p_2$ is determined by $p_1$, so the composite $Y_2\subset \tilde{\mathbb T}\to \mathbb T$ is still an injection, and it is also an open subset (as all maps to $\mathbb T$ are finitary). Moreover, once $N$ is sufficiently large, the image of $Y_2$ does not meet the central ray of Gau\ss\ norms. But, over each point of $\mathcal M(A)$, the complement of this central ray is covered by open unit discs, and those cannot map to $\overline{\mathbb G}_m$ by ball-invariance; it follows that each of those is either fully contained in $Y_2$ or not at all. From here, one easily sees that $Y_2$ admits a fibration in open balls over a space $\overline{Y}_2$ that is \'etale over $\mathcal M_{\mathrm{arc}}(A)$. We can also assign multiplicities. As before, one also gets a map $Y_2\to \mathbb N[U]$ by using the projection to $U$.

This finishes the construction. It is easy to check that for all $N$ for which the map is defined, it is a splitting. Indeed, in that case $p_1=p_2$ and so $p_1^N=p_2$ cuts out the open balls around the $N-1$-th roots of unity (counted with multiplicity in case $N-1$ is divisible by the residue characteristic of $C$). The map to $\mathbb N[U]$ is constant, so $Y_1$ yields $N$ times the original section of $\mathbb N[U]$, while $Y_2$ yields it $N-1$ times; the difference is thus the original section. This already proves injectivity of the map
\[
\pi_i\mathbb Q_{\mathrm{mot}}[U]\to \pi_i\intHom(\mathbb Q_{\mathrm{mot}}[(\mathbb G_m,\ast)],\mathbb Q_{\mathrm{mot}}[U\times (\overline{\mathbb G}_m,\ast)]).
\]
We note that we are implicitly writing $\mathbb Q_{\mathrm{mot}}[U\times (\overline{\mathbb G}_m,\ast)]$ as a retract of $\mathbb Q_{\mathrm{mot}}[U\times \overline{\mathbb G}_m]$, and we use that when we apply the above construction to maps to $\mathbb Q_{\mathrm{mot}}[U\times \ast]\subset \mathbb Q_{\mathrm{mot}}[U\times \overline{\mathbb G}_m]$, the construction yields zero. Indeed, in this case in the construction of $Y_2$ we have $p_2=1$ and hence we look at the locus $p_1^N=1$, which is precisely the locus of $N$-th roots of unity. Thus, $Y_2=Y_1$ in that case, and we get the zero map.

For surjectivity, it suffices to see that the other composite from $\pi_i\intHom(\mathbb Q_{\mathrm{mot}}[(\mathbb G_m,\ast)],\mathbb Q_{\mathrm{mot}}[U\times (\overline{\mathbb G}_m,\ast)])$ to eventually defined sequences in $\pi_i\mathbb Q_{\mathrm{mot}}[U]$, and back to eventually defined sequences in $\pi_i \intHom(\mathbb Q_{\mathrm{mot}}[(\mathbb G_m,\ast)],\mathbb Q_{\mathrm{mot}}[U\times (\overline{\mathbb G}_m,\ast)])$, is the identity.

For this, we use that the same construction also gives, for every affine $V$, a map
\[
\pi_i\intHom(\mathbb Q_{\mathrm{mot}}[V\times (\mathbb G_m,\ast)],\mathbb Q_{\mathrm{mot}}[U\times (\overline{\mathbb G}_m,\ast)])\dashrightarrow \pi_i \intHom(\mathbb Q_{\mathrm{mot}}[V],\mathbb Q_{\mathrm{mot}}[U])
\]
where the dashed right arrow means that it is actually a map to eventually defined sequences. Now we use the commutative diagram
\[\xymatrix@C=2em{
\intHom(\mathbb Q_{\mathrm{mot}}[(\mathbb G_m,\ast)],\mathbb Q_{\mathrm{mot}}[U\times (\overline{\mathbb G}_m,\ast)])\ar[r]\ar[d] & \mathbb Q_{\mathrm{mot}}[U]\ar[d]\\
\intHom(\mathbb Q_{\mathrm{mot}}[(\mathbb G_m,\ast)\times (\mathbb G_m,\ast)],\mathbb Q_{\mathrm{mot}}[(\overline{\mathbb G}_m,\ast)\times U\times (\overline{\mathbb G}_m,\ast)])\ar[r] & \intHom(\mathbb Q_{\mathrm{mot}}[(\mathbb G_m,\ast)],\mathbb Q_{\mathrm{mot}}[(\overline{\mathbb G}_m,\ast)\times U])
}\]
where the maps are really only defined on $\pi_i$, and on the right we only get eventually defined sequences. We want to see that the composite along the upper right corner is the identity. If one inserts the swap automorphism of $\mathbb G_m\times \mathbb G_m$ and $\overline{\mathbb G}_m\times \overline{\mathbb G}_m$ on both terms in the lower left corner, then the composite is indeed the identity: This reduces to the easy composite being the identity. But the swap automorphism of $\overline{\mathbb G}_m\otimes \overline{\mathbb G}_m$ is minus the identity by Lemma~\ref{lem:symmetry} (note that $\mathbb Z(1)$ differs from $\overline{\mathbb G}_m$ by a shift, which introduces the sign), so using the swap automorphisms twice yields the identity, giving the desired result.
\end{proof}

We used the following lemma at the end of the proof.

\begin{lemma}\label{lem:symmetry} For any $\sigma\in \Sigma_n$, the action of $\sigma$ on $\mathbb Z(1)^{\otimes n}$ is trivial.
\end{lemma}

Of course, once this lemma was used to finish the proof of the cancellation theorem, we see that $\mathrm{End}(\mathbb Z(n))=\mathrm{End}(\mathbb Z)$ which is concentrated in degree $0$ and hence the triviality of the action of each $\sigma$ implies that the action of $\Sigma_n$ on $\mathbb Z(n)$ is homotopy-coherently trivial. For the moment, the weaker assertion is sufficient.

\begin{proof} Writing any element as a product of transpositions, we reduce to $n=2$. Now $\mathbb Z(2)[3]$ can be identified with the reduced homology of $\mathbb A^2\setminus \{0\}$, i.e.~the fibre of
\[
\mathbb Z_{\mathrm{mot}}[\mathbb A^2\setminus \{0\}]\to \mathbb Z_{\mathrm{mot}}[\ast].
\]
Indeed, this easily follows from the resolution
\[
\mathrm{coeq}(\mathbb G_m^2\rightrightarrows \mathbb G_m\times \mathbb A^1\sqcup \mathbb A^1\times \mathbb G_m)= \mathbb A^2\setminus \{0\}.
\]
There is an action of $\mathrm{GL}_2$ on $\mathbb A^2\setminus \{0\}$. We want to see that the antidiagonal matrix, switching the two coordinates, acts trivially on the reduced homology. But the subgroup generated by unipotent matrices acts trivially, by homotopy-invariance. This subgroup is $\mathrm{SL}_2$. Up the action of $\mathrm{SL}_2$, the antidiagonal matrix is equivalent to the diagonal matrix with entries $(-1,1)$. It remains to see that multiplication by $-1$ (or any element of $\mathbb G_m$) acts trivially on the reduced homology of $\mathbb G_m$, which is a direct verification using the identification with $\overline{\mathbb G}_m$. Indeed, the reduced homology of $\mathbb G_m$ is the kernel of $\mathbb Z_{\mathrm{mot}}[\mathbb G_m]\to \mathbb Z_{\mathrm{mot}}[\ast]$, and the isomorphism with $\overline{\mathbb G}_m$ is given by sending $\sum n_i [x_i]$ to $\prod \overline{x}_i^{n_i}$. The action of $-1$ sends $\sum n_i [x_i]$ to $\sum n_i [-x_i]$ which is mapped to $(-1)^{\sum n_i} \prod \overline{x}_i^{n_i}$. But $\sum n_i=0$ for any element in the kernel of $\mathbb Z_{\mathrm{mot}}[\mathbb G_m]\to \mathbb Z_{\mathrm{mot}}[\ast]$ and so $(-1)^{\sum n_i} = 1$.
\end{proof}

\section{$K$-theory}

In order to determine higher weight motivic cohomology rationally, we use the relation to algebraic $K$-theory. The argument comparing algebraic $K$-theory and motivic cohomology rationally that we give below can also be used in the usual setting of algebraic motives and seems more direct than other arguments known to the author.

For $n\geq 0$, set $\mathbb Z(n) = \mathbb Z(1)^{\otimes n}$. The following statement is a corollary of Lemma~\ref{lem:symmetry}.

\begin{corollary} For $n\geq 0$, the symmetric power $\mathrm{Sym}^n(\mathbb Q(1))$ agrees with $\mathbb Q(n)$.
\end{corollary}

\begin{proof} Indeed, rationally symmetric powers are the direct summand of tensor powers on which $\Sigma_n$ acts trivially.
\end{proof}

In the next corollary, we use that free motivic sheaves also exist for smooth Artin stacks, by left Kan extension.

\begin{corollary} There is a canonical direct sum decomposition
\[
\mathbb Q_{\mathrm{mot}}[\ast/\mathbb G_m] = \bigoplus_{n\geq 0} \mathbb Q(n)[2n]
\]
and for every $n\geq 0$ the map
\[
\mathbb Q_{\mathrm{mot}}[\mathbb P^n]\to \mathbb Q_{\mathrm{mot}}[\ast/\mathbb G_m]
\]
induced by the map $\mathbb P^n\to \ast/\mathbb G_m$, classifying the line bundle $\mathcal O_{\mathbb P^n}(1)$, maps isomorphically to the summand
\[
\bigoplus_{i=0}^n \mathbb Q(i)[2i].
\]
\end{corollary}

Of course, these results are also true with $\mathbb Z$-coefficients. We state it with $\mathbb Q$-coefficients here as our method of producing the splitting works only rationally.

\begin{proof} For any animated abelian group $M$, the group algebra $\mathbb Z[M]$ has a descending filtration whose graded pieces are given by $\mathrm{Sym}^n(M)$ (via animating the filtration coming from powers of the augmentation ideal). If $M$ is $1$-connective, this filtration converges. In this case, this yields a complete descending filtration of $\mathbb Q_{\mathrm{mot}}[\ast/\mathbb G_m]$ with graded pieces $\mathbb Q(n)[2n]$ (using the previous corollary). Note that $\mathbb Q(n)[2n]$ sits in degrees $\geq n$.

Moreover, the filtration is split by the increasing filtration via cells, yielding the desired result. (Note that the cofiber of
\[
\mathbb Z_{\mathrm{mot}}[\mathbb P^{n-1}]\to \mathbb Z_{\mathrm{mot}}[\mathbb P^n]
\]
can be shown to be $\mathbb Z(n)[2n]$ as follows. Use the covering of $\mathbb P^n$ by $\mathbb A^n$ and $\mathbb P^n\setminus \{0\}$; the first piece is contractible, the second piece is an $\mathbb A^1$-fibration over $\mathbb P^{n-1}$ and hence cancels with $\mathbb P^{n-1}$, so the remainder is a shift of the reduced homology of the intersection $\mathbb A^n\setminus \{0\}$, which gives $\mathbb Z(n)[2n]$.)
\end{proof}

Finally, we start our study of $K$-theory.

\begin{definition} Let $K_{\mathrm{cn}}$ be the presheaf of spectra taking any Banach ring $A$ to the connective $K$-theory spectrum $K_{\mathrm{cn}}(A)\in \mathrm{Sp}$.
\end{definition}

We recall that $K_{\mathrm{cn}}(A)$ is (for us, by definition) the group completion of the $E_\infty$-monoid of finite projective $A$-modules. This yields an $E_\infty$-group, or equivalently a connective spectrum, cf.~e.g.~the appendix of \cite{BhattScholzeWitt}, especially \cite[Definition 12.13]{BhattScholzeWitt}, for this perspective on $K$-theory.

We note that we are departing from our usual conventions, and consider a functor with values in spectra, not $\mathcal D(\mathbb Z)$. As we will be working rationally momentarily, the difference will disappear. Note however that everything we did with $\mathcal D(\mathbb Z)$-coefficients works also with $\mathrm{Sp}$-coefficients.

\begin{proposition} The universal ball-invariant arc-sheaf $\overline{K}$ under $K_{\mathrm{cn}}$ is finitary. In particular, $\overline{K}$ inherits an $E_\infty$-ring structure from $K_{\mathrm{cn}}$.
\end{proposition}

\begin{proof} After sheafification, the functor $K_{\mathrm{cn}}$ is the group completion of $\bigsqcup_{n\geq 0} [\ast/\mathrm{GL}_n]$. Group completion can be computed by a Breen--Deligne resolution (see \cite[Theorem 4.4]{AokiSemitopological} for the precise assertion we use), so $K_{\mathrm{cn}}$ is resolved by sheaves of the form
\[
\mathbb S[\ast/G]
\]
where $G$ is a product of $\mathrm{GL}_n$'s. Those, in turn, are resolved by sheaves of the form $\mathbb S[G]$, and those admit the finitary sheaves $\mathbb S_{\mathrm{mot}}[G]$ as their universal ball-invariant arc-sheaves.
\end{proof}

There is a multiplicative map $\ast/\mathbb G_m\to K_{\mathrm{cn}}$ coming from line bundles, inducing a map
\[
\mathbb S[\ast/\mathbb G_m]\to K_{\mathrm{cn}}
\]
and hence
\[
\mathbb S_{\mathrm{mot}}[\ast/\mathbb G_m]\to \overline{K}.
\]

\begin{proposition} After rationalization, the map
\[
\mathbb Q_{\mathrm{mot}}[\ast/\mathbb G_m]\to \overline{K}\otimes \mathbb Q
\]
is an isomorphism.
\end{proposition}

\begin{proof} Identifying connective spectra and $E_\infty$-groups in anima, $\mathbb S[\ast/\mathbb G_m]$ is the group completion of
\[
\bigsqcup_{n\geq 0} \ast/(\mathbb G_m^n\rtimes \Sigma_n);
\]
indeed, the displayed object is the free $E_\infty$-monoid on $\ast/\mathbb G_m$. Explicitly, this can be described as the $E_\infty$-monoid of finite projective modules $P$ equipped with lines $(L_i)_i$, $L_i\subset P$, such that $\bigoplus_i L_i\to P$ is an isomorphism. We may then compute the rationalization of the group completion in terms of the Breen--Deligne resolution involving tensor products of terms
\[
\mathbb Q[\ast/(\mathbb G_m^n\rtimes \Sigma_n)].
\]
The comparison map to $\overline{K}$ comes from the maps
\[
\ast/(\mathbb G_m^n\rtimes \Sigma_n)\to \ast/\mathrm{GL}_n
\]
embedding the normalizer of the diagonal torus (indeed, on the level of $E_\infty$-monoids, this takes $(P,(L_i)_i)$ to $P$), and it suffices to see that they induce isomorphisms
\[
\mathbb Q_{\mathrm{mot}}[\ast/(\mathbb G_m^n\rtimes \Sigma_n)]\to \mathbb Q_{\mathrm{mot}}[\ast/\mathrm{GL}_n].
\]
Using transitivity for $f_\sharp$, it suffices to prove that for the map
\[
f: \ast/(\mathbb G_m^n\rtimes \Sigma_n)\to \ast/\mathrm{GL}_n,
\]
one has $f_\sharp \mathbb Q = \mathbb Q$, which by base change reduces to showing
\[
\mathbb Q_{\mathrm{mot}}[\mathrm{GL}_n/\mathbb G_m^n\rtimes \Sigma_n] = \mathbb Q.
\]
Equivalently, the direct summand of
\[
\mathbb Q_{\mathrm{mot}}[\mathrm{GL}_n/\mathbb G_m^n]
\]
on which $\Sigma_n$ acts trivially is just $\mathbb Q$. Note that $\mathrm{GL}_n/\mathbb G_m^n$ is homotopy-equivalent to the flag variety for $\mathrm{GL}_n$, so in realizations this is the standard assertion that for the natural $\Sigma_n$-action on the homology of the flag variety, the invariants are concentrated in degree zero.

We argue by induction on $n$. There is a $\Sigma_{n-1}$-equivariant smooth projection
\[
\mathrm{GL}_n/\mathbb G_m^n\to \mathrm{GL}_n/\mathrm{GL}_{n-1}\times \mathbb G_m
\]
whose fibres are given ($\Sigma_{n-1}$-equivariantly) by $\mathrm{GL}_{n-1}/\mathbb G_m^{n-1}$. It follows that the $\Sigma_{n-1}$-invariants in
\[
\mathbb Q_{\mathrm{mot}}[\mathrm{GL}_n/\mathbb G_m^n]
\]
are given by
\[
\mathbb Q_{\mathrm{mot}}[\mathrm{GL}_n/\mathrm{GL}_{n-1}\times \mathbb G_m]\cong \mathbb Q_{\mathrm{mot}}[\mathbb P^{n-1}]\cong \bigoplus_{i=0}^{n-1} \mathbb Q(i)[2i].
\]
The $\Sigma_n$-invariants are then a direct summand of this. By the cancellation theorem, we know that all endomorphisms of $\bigoplus_{i=0}^{n-1} \mathbb Q(i)[2i]$ are upper-triangular, and on the diagonal given by rational numbers. To determine the image of an idempotent, it suffices to determine which of these numbers are zero or one. This can be done in multiple ways by using classical realizations (for example, by interpolating to the case over the complex numbers, or by using \'etale realizations), and one sees that the $\Sigma_n$-invariants are just $\mathbb Q[0]$.
\end{proof}

\begin{corollary} There is a separated descending filtration $\mathrm{Fil}_{\mathrm{mot}}^\ast$ on $\overline{K}\otimes \mathbb Q$ whose graded pieces are $\mathbb Q(n)[2n]$ for $n\geq 0$. In fact, via Adams operations, this filtration is canonically split.
\end{corollary}

The same argument applies in classical motivic homotopy theory.

\begin{proof} For the canonical splitting, note that multiplication by $2$ induces an endomorphism of $\ast/\overline{\mathbb G}_m$ and hence $\mathbb Q_{\mathrm{mot}}[\ast/\overline{\mathbb G}_m] = \overline{K}\otimes \mathbb Q$ which acts by multiplication through $2^n$ on the $n$-th graded piece. As these numbers are distinct, we can isolate the eigenspaces.
\end{proof}

We want to make $\overline{K}$ more explicit. For this, we first have to replace $K_{\mathrm{cn}}$ by a finitary approximation. We assume that our base ring $A$ is equipped with a unit $T\in A$ with $|T|<1$ and such that $T$ admits $n$-th roots for infinitely many $n$. In that case the non-unital ring $A_{<1}\subset A$ is Tor-unital: If $A_{<1}$ is an ideal in a unital ring $B$, then it is automatically the colimit of $B$ along multiplication by $T^{1/n_i-1/n_{i+1}}$ (where $n_i$ is a strictly increasing sequence of integers for which $T$ admits $n_i$-th roots) and hence flat and idempotent. This implies that the fiber of $K(B)\to K(B/A_{<1})$ is independent of $B$. Classically, this is due to Suslin \cite{SuslinExcision}. But even better, using Efimov's $K$-theory of dualizable categories, we note that there is the Verdier quotient
\[
\mathcal D(B)\to \mathcal D(B/A_{<1})
\]
as $B/A_{<1}$ is an idempotent $B$-algebra. The kernel of this Verdier quotient is the dualizable $\infty$-category $\mathcal D(B^a)$ of almost $B$-modules with respect to the ideal $A_{<1}$. As $K$-theory takes Verdier sequences to fibre sequences, we see that
\[
K(\mathcal D(B^a)) = \mathrm{fib}(K(B)\to K(B/A_{<1})).
\]
But even $\mathcal D(B^a)$ depends only on $A_{<1}$: Indeed, this category of almost modules is also the category of modules $M$ over the non-unital ring $A_{<1}$ such that the natural map
\[
M\dotimes_{A_{<1}} A_{<1}\to M
\]
is an isomorphism. Here, $M\dotimes_{A_{<1}} A_{<1}$ can be computed as the colimit of copies of $M$ with transition maps $T^{1/n_i-1/n_{i+1}}$ as above. We denote this $K$-theory by $K(A_{<1})$. Let $K_{\mathrm{cn}}(A_{<1})$ denote its connective cover. There is a natural map $K_{\mathrm{cn}}(A_{<1})\to K_{\mathrm{cn}}(A)$.

\begin{proposition} The universal ball-invariant arc-sheaf under the functor $A\mapsto K_{\mathrm{cn}}(A_{<1})$ is trivial.
\end{proposition}

\begin{proof} Let us choose $B=\mathbb Z\oplus A_{<1}$, a functorial choice of unital ring with ideal $A_{<1}$. We have to understand the fiber of $K_{\mathrm{cn}}(B)\to K(\mathbb Z)$. Now $K_{\mathrm{cn}}(B)$ is the group completion of $\bigsqcup_n \ast/\mathrm{GL}_n(\mathbb Z\oplus A_{<1})$. It suffices to observe that the functor $A\mapsto \mathrm{GL}_n(\mathbb Z\oplus A_{<1})$ maps to $\mathrm{GL}_n(\mathbb Z)$ with fibres that are ball-contractible.
\end{proof}

Thus, we can alternatively work with the cofibre of $K_{\mathrm{cn}}(A_{<1})\to K_{\mathrm{cn}}(A)$.

\begin{proposition}\label{prop:Ktheoryfinitary} The functor $A\mapsto \mathrm{cofib}(K_{\mathrm{cn}}(A_{<1})\to K_{\mathrm{cn}}(A))$, as well as its nonconnective version $\mathrm{cofib}(K(A_{<1})\to K(A))$, commutes with filtered colimits of Banach rings admitting a topologically nilpotent unit $T$ with roots as above. When evaluated on totally disconnected $A$, the connective and nonconnective versions agree.
\end{proposition}

\begin{proof} Commutation with (uncompleted) filtered colimits is clear. For invariance under completion, take any seminormed ring $A$ with $T\in A^\times$ with $|T|<1$ and admitting all roots. We have a cofiber sequence
\[
K(\mathbb Z\oplus A_{<1}\ \mathrm{on}\ (\mathbb Z\oplus A_{<1})/T)\to K(\mathbb Z\oplus A_{<1})\to K(A)
\]
as $A=(\mathbb Z\oplus A_{<1})[T^{-1}]$. Passing to the cofiber of $K(A_{<1})$ the middle term becomes simply $K(\mathbb Z)$ which is certainly invariant under changing $A$; and the left-most term is unchanged under replacing $A_{<1}$ by its $T$-adic completion, by the results of Thomason--Trobaugh \cite[Proposition 3.19]{ThomasonTrobaugh}. This gives the result for nonconnective $K$-theory. The case of connective $K$-theory follows.

To show the agreement for totally disconnected $A$, we can (by finitaryness) reduce to Banach fields $K$. Then $K$, $\mathcal O_K$ and $k$ are all valuation rings and hence $K$-theory agrees with nonconnective $K$-theory.
\end{proof}

Our goal now is to prove the following theorem, describing $\overline{K}$ explicitly at least on strictly totally disconnected $A$; in fact, below, we will prove stronger results describing more values of $\overline{K}$, at least rationally.

\begin{theorem}\label{thm:explicitKbar} Fix some set of primes $S$. On strictly totally disconnected Banach rings $A$ in which $|p|\geq 1$ everywhere on $\mathcal M(A)$ for all $p\not\in S$, the functor $A\mapsto \mathrm{cofib}(K(A_{<1})\to K(A))[1/S]$ defines a finitary arc-sheaf of spectra that is ball-invariant, and hence agrees with $\overline{K}[1/S]$.
\end{theorem}

A curious consequence is the following.

\begin{corollary} Let $C$ be an algebraically closed nonarchimedean field. Then there is a natural $E_\infty$-$K(C)$-algebra structure on $\mathrm{cofib}(K(C_{<1})\to K(C))$.
\end{corollary}

\begin{proof} After inverting the residue characteristic $p$ of $C$, this is a direct consequence of the identification with $\overline{K}(C)$. But after $p$-completion, $K(\mathcal O_C)$ and $K(C)$ agree (cf.~\cite[Lemma 1.3.7]{HesselholtNikolaus}), so this cofiber agrees with $K(k)$ after $p$-completion, and this is just $\mathbb Z_p$ as $k$ is a perfect field (cf.~\cite[Corollaire 5.5]{Kratzer}). This has a unique $E_\infty$-$K(C)$-algebra structure. To glue the two pieces, we need to see that that the map
\[
\mathrm{cofib}(K(C_{<1})\to K(C))[1/p]\to \mathrm{cofib}(K(C_{<1})\to K(C))^\wedge_p[1/p]\cong \mathbb Q_p
\]
admits the structure of a map of $E_\infty$-$K(C)$-algebras. But the map is simply the projection $K(C)\to \mathbb Z$ given by the rank (which is trivial on $K(C_{<1})$), composed with inverting $p$ and embedding into $\mathbb Q_p$.
\end{proof}

From the categorical perspective, it is quite unclear why this ring structure should exist -- it is like saying that almost $\mathcal O_C$-modules form an ideal in $C$-vector spaces.

To prove Theorem~\ref{thm:explicitKbar}, note first that with torsion coefficients, it is easy to prove the required invariance under change of algebraically closed field (using Gabber--Suslin rigidity). Thus, we concentrate on the case of rational coefficients.

We will prove Theorem~\ref{thm:explicitKbar} first when $|p|<1$. In that case, slightly surprisingly, the functor is already a sheaf for the analytic topology, at least on perfectoid rings.

\begin{theorem} On perfectoid Banach rings $A$ (with $|p|_A<1$), the functor $A\mapsto \mathrm{cofib}(K(A_{<1})\to K(A))$ is a sheaf for the analytic topology.
\end{theorem}

\begin{proof} We can work with the fiber instead. This is the $K$-theory of the dualizable stable $\infty$-category $\mathcal D_{\mathrm{tor}}(A_{\leq 1})^a$ of torsion almost modules over $A_{\leq 1}$. Indeed, there are two Verdier sequences
\[
0\to \mathcal D_{\mathrm{tor}}(A_{\leq 1})^a\to \mathcal D(A_{\leq 1})^a\to \mathcal D(A)\to 0
\]
and
\[
0\to \mathcal D(A_{\leq 1})^a\to \mathcal D(A_{\leq 1})\to \mathcal D(A_{\leq 1}/A_{<1})\to 0
\]
The second sequence shows $K(A_{<1})=K(\mathcal D(A_{\leq 1})^a)$, and the first one then gives
\[
K(\mathcal D_{\mathrm{tor}}(A_{\leq 1})^a) = \mathrm{fib}(K(A_{<1})\to K(A)).
\]

But $\mathcal D_{\mathrm{tor}}(A_{\leq 1})^a$ localizes on $\mathcal M(A)$ by the usual properties of the structure sheaf of perfectoid algebras, which in particular yields a sheaf of idempotent algebras with values in $\mathcal D_{\mathrm{tor}}(A_{\leq 1})^a$ (when interpreting torsion modules as complete modules). In other words, for any closed subset $Z\subset \mathcal M(A)$, the algebra $\mathcal O^+(Z)^a$ is an idempotent complete $A_{\leq 1}^a$-algebra, giving a Verdier sequence
\[
0\to \mathcal D_{\mathrm{tor}}(A_{\leq 1}^a\ \mathrm{on}\ U)\to \mathcal D_{\mathrm{tor}}(A_{\leq 1})^a\to \mathcal D_{\mathrm{tor}}(\mathcal O^+(Z)^a)\to 0,
\]
where $U$ is the complement of $Z$ (and the term $\mathcal D(\ldots\ \mathrm{on}\ U)$ is defined as the fibre). If $Z$ and $Z'$ cover $\mathcal M(A)$, then there are compatible Verdier sequences
\[\xymatrix{
0\ar[r] & \mathcal D_{\mathrm{tor}}(A_{\leq 1}^a\ \mathrm{on}\ U)\ar[r]\ar[d] & \mathcal D_{\mathrm{tor}}(A_{\leq 1})^a\ar[r]\ar[d] &  \mathcal D_{\mathrm{tor}}(\mathcal O^+(Z)^a)\ar[r]\ar[d] & 0 \\
0\ar[r] & \mathcal D_{\mathrm{tor}}(\mathcal O^+(Z')^a\ \mathrm{on}\ U)\ar[r] & \mathcal D_{\mathrm{tor}}(\mathcal O^+(Z')^a)\ar[r] & \mathcal D_{\mathrm{tor}}(\mathcal O^+(Z\cap Z')^a)\ar[r] & 0
}\]
where the left vertical functor is an equivalence. Applying $K$-theory then shows the desired sheaf property.
\end{proof}

\begin{corollary} On perfectoid Banach rings $A$ (with $|p|_A<1$) such that all sheaves on the compact Hausdorff space $\mathcal M(A)$ are Postnikov complete, the functor $A\mapsto \mathrm{cofib}(K(A_{<1})\to K(A))_{\mathbb Q}$ is a finitary ball-invariant hypercomplete arc-sheaf. In particular, on such $A$, we have
\[
\overline{K}(A)_{\mathbb Q} = \mathrm{cofib}(K(A_{<1})\to K(A))_{\mathbb Q}.
\]
\end{corollary}

The condition that all sheaves are Postnikov complete is satisfied, for example, if $\mathcal M(A)$ has finite covering dimension \cite[Proposition 7.2.1.10]{LurieHTT}.

\begin{proof} The commutation with filtered colimits comes from Proposition~\ref{prop:Ktheoryfinitary} and does not require any assumptions on $A$ and the $A_i$ with $A=\mathrm{colim}_i A_i$.

Rationally, finite \'etale descent holds because of trace maps; it even gives pro-finite \'etale descent. This combines with the analytic descent we already have to give a hypercomplete arc-sheaf. Indeed, from Theorem~\ref{theorem:finitaryarc} we can produce a hypercomplete arc-sheaf preserving the values at strictly totally disconnected spaces. We need to see that this leaves the value at $A$ unchanged. But using pro-finite \'etale descent, we know that the value at totally disconnected $A$ stays unchanged, and then one uses descent along a profinite cover of $\mathcal M(A)$ (from analytic descent and Postnikov completeness) to see that the value at $A$ is also unchanged, under the assumption on $A$.

Finally, we check invariance under perfectoid balls (which is equivalent to usual ball-invariance). We know explicitly the value on $C\langle T^{1/p^\infty}\rangle_1$ (as $\mathcal M(C\langle T^{1/p^\infty}\rangle_1)$ is a tree on which all sheaves are Postnikov complete): This is the cofiber of
\[
K((C\langle T^{1/p^\infty}\rangle_1)_{<1})\to K(C\langle T^{1/p^\infty}\rangle_1).
\]
By excision, this is the same as for
\[
K(C[T^{1/p^\infty}]_{<1})\to K(C[T^{1/p^\infty}])
\]
and the first term is the fiber of
\[
K(\mathcal O_C[T^{1/p^\infty}])\to K(k[T^{1/p^\infty}]).
\]
Now usual $\mathbb A^1$-invariance for $K$-theory (over valuation rings) yields the desired result.
\end{proof}

Unfortunately, this argument does not seem to work in equal characteristic zero. However, the following slightly weaker result is still true. Namely, for totally disconnected $A$, the cofiber $\mathrm{cofib}(K(A_{<1})\to K(A))$ also agrees with its homotopy-invariant version $\mathrm{cofib}(KH(A_{<1})\to KH(A))$, where this is also simply the result of applying $L_{\mathbb B}^{\mathrm{pre}}$. This functor is still finitary. The advantage is that as $KH$ is truncating (i.e.~does not see quotients by idempotent ideals), this is now well-defined even if $A$ does not have a pseudo-uniformizer with enough roots.

\begin{theorem} On all analytic nonarchimedean Banach rings $A$, the functor
\[
A\mapsto \mathrm{cofib}(KH(A_{<1})\to KH(A))
\]
is a sheaf for the analytic topology.

In particular, restricting to $A$ such that all sheaves on the compact Hausdorff space $\mathcal M(A)$ are Postnikov complete, the functor
\[
A\mapsto \mathrm{cofib}(KH(A_{<1})\to KH(A))_{\mathbb Q}
\]
is a finitary ball-invariant arc-hypersheaf. In particular, on such $A$, we have
\[
\overline{K}(A)_{\mathbb Q} = \mathrm{cofib}(KH(A_{<1})\to KH(A))_{\mathbb Q}.
\]
If $A$ is totally disconnected, this simplifies to
\[
\overline{K}(A)_{\mathbb Q} = \mathrm{cofib}(K(A_{<1})\to K(A))_{\mathbb Q}.
\]
\end{theorem}

\begin{proof} For the first statement, we can use the finitary nature of the functor to do a noetherian approximation to assume that $A_{\leq 1}$ is noetherian. The idea now is to use formal models and admissible blowups to relate analytic descent to (pro-)cdh descent results. We can either cite directly Cisinski's cdh-descent for $KH$, \cite{Cisinski}. Alternatively, in order to give an argument independent of the existing motivic formalism, it follows from pro-cdh descent \cite[Theorem 3.7]{MorrowProCdh} (for the equal characteristic $0$ case that we really need) and \cite[Theorem A]{ProCdh} (in the general case) upon making things ball-invariant. Namely, pro-cdh descent implies that the pro-spectrum $\{\mathrm{cofib}(K(A_{<r})\to K(A))\}_r$ satisfies descent, but after making it ball-invariant, all terms agree. Here $K(A_{<r})$ is notation for the fiber of $K(A_{\leq 1})\to K(A_{\leq 1}/A_{<r})$. The rest follows as before.
\end{proof}

\section{$\mathcal D_{\mathrm{mot}}(X)$}

At this point, we can finally define $\mathcal D_{\mathrm{mot}}(X)$. Recall that $\mathbb Z(1)$ is a symmetric object in $\mathcal D_{\mathrm{mot}}^{\mathrm{eff}}(X)$, so making it $\otimes$-invertible is a simple operation. We refer to Robalo \cite[Section 2.1]{Robalo} for a general discussion of adjoining tensor-inverses of objects.

\begin{definition} Let $X$ be a small arc-stack. The symmetric monoidal presentable stable $\infty$-category of motivic sheaves on $X$ is
\[
\mathcal D_{\mathrm{mot}}(X) = \mathcal D_{\mathrm{mot}}^{\mathrm{eff}}(X)[\mathbb Z(1)^{\otimes -1}].
\]
\end{definition}

This can be described equivalently as the sequential colimit
\[
\mathrm{colim}_{-\otimes \mathbb Z(1)} \mathcal D_{\mathrm{mot}}^{\mathrm{eff}}(X)
\]
taken in presentable $\infty$-categories, or as the limit
\[
\mathrm{lim}_{\intHom(\mathbb Z(1),-)} \mathcal D_{\mathrm{mot}}^{\mathrm{eff}}(X)
\]
in $\infty$-categories. In the second description, an object of $\mathcal D_{\mathrm{mot}}(X)$ consists of a collection $M_n\in \mathcal D_{\mathrm{mot}}^{\mathrm{eff}}(X)$ for $n\geq 0$, with equivalences
\[
M_n\cong \intHom(\mathbb Z(1),M_{n+1}).
\]

There is a fully faithful inclusion
\[
\mathcal D_{\mathrm{mot}}^{\mathrm{eff}}(X)\subset \mathcal D_{\mathrm{mot}}(X)
\]
sending any $M$ to the sequence $M_n = M(n)=M\otimes \mathbb Z(n)$ (with the equivalences coming from the cancellation theorem).

\begin{theorem} For any map $f: Y=\mathcal M_{\mathrm{arc}}(B)\to X=\mathcal M_{\mathrm{arc}}(A)$ of finite cohomological dimension of affine Berkovich spaces such that $Y$ admits an embedding into a finite-dimensional ball over $X$, the functor
\[
f_\ast: \mathcal D_{\mathrm{mot}}(Y)\to \mathcal D_{\mathrm{mot}}(X),
\]
right adjoint to pullback $f^\ast$, commutes with all colimits, base change, and satisfies the projection formula.
\end{theorem}

Using descent (and \cite[Corollary 4.7.4.18]{LurieHA}), one can formally deduce that the similar result holds true more generally when $f$ is only arc-locally (on the target) of this form. An even more general form of the result will be proved in Corollary~\ref{cor:generalprojectionformula} below.

\begin{proof} We already had colimit-preserving right adjoints $f_\ast$ on $\mathcal D_{\mathrm{mot}}^{\mathrm{eff}}$ satisfying base change by Theorem~\ref{thm:properpushforwardfinitary}. They pass on to $\mathcal D_{\mathrm{mot}}$, taking a collection
\[
N=(N_n)_n\in \mathrm{lim}_{\intHom(\mathbb Z(1),-)} \mathcal D_{\mathrm{mot}}^{\mathrm{eff}}(Y)
\]
to the collection $f_\ast N=(M_n)_n$ with $M_n = f_\ast N_n$; this obviously still commutes with colimits and satisfies base change. We have to see that it satisfies the projection formula. Checking isomorphisms at stalks, we can assume that $X=\mathcal M_{\mathrm{arc}}(C)$ with $C$ a non-discrete algebraically closed Banach field. The case $C=\mathbb C$ is trivial, using Theorem~\ref{thm:archimedeanfinitary} and the usual projection formula in topology, so we can assume that $C$ is nonarchimedean.

One can then further reduce to the case that $Y$ is a finite-dimensional ball (as the result is clear for closed immersions, e.g.~by the final part of Theorem~\ref{thm:properpushforwardfinitary}). By induction, we can reduce to the case of a $1$-dimensional ball. This can be embedded into $\mathbb P^1_C$. Now the result follows from Proposition~\ref{prop:projectionformulaP1} below.
\end{proof}

\begin{proposition}\label{prop:projectionformulaP1} Let $f: \mathbb P^1_X\to X$ be the projection from the projective line. Then there is a natural isomorphism
\[
f_\sharp\cong f_\ast\otimes \mathbb Z(1)[2]: \mathcal D_{\mathrm{mot}}(\mathbb P^1_X)\to \mathcal D_{\mathrm{mot}}(X).
\]
In particular, using that $f_\sharp$ satisfies the projection formula, we see that $f_\ast$ satisfies the projection formula on $\mathcal D_{\mathrm{mot}}$.
\end{proposition}

\begin{proof} Note that $\mathbb Z(1)[2]$ is $\otimes$-invertible in $\mathcal D_{\mathrm{mot}}$, so the projection formula automatically holds when tensoring with $\mathbb Z(1)[2]$. We will use this repeatedly below.

For this, we will construct the unit and counit transformations, and check the required composites are isomorphisms. For the unit, we need a map
\[
\alpha: \mathrm{id}\to f^\ast(f_\ast\otimes \mathbb Z(1)[2])\cong p_{1\ast}(p_2^\ast\otimes \mathbb Z(1)[2]).
\]
Here $p_1,p_2: \mathbb P^1_X\times_X\mathbb P^1_X\to \mathbb P^1_X$ are the two projections, with diagonal $\Delta$. This will arise via applying $p_{1\ast}$ to a map
\[
\Delta_\ast\to p_2^\ast\otimes \mathbb Z(1)[2].
\]
As $\Delta$ is a closed immersion, $\Delta_\ast$ satisfies the projection formula, and it suffices to construct a map
\[
\Delta_\ast \mathbb Z\to \mathbb Z(1)[2].
\]
Via excision, this amounts to giving a $\overline{\mathbb G}_m$-torsor over $\mathbb P^1_X\times_X \mathbb P^1_X$ that is trivialized outside the diagonal; and for this, we take the one induced from the $\mathbb G_m$-torsor given by the line bundle $\mathcal O(\Delta)$.

For the counit, we need a map
\[
\beta: f_\ast(f^\ast\otimes \mathbb Z(1)[2])\to \mathrm{id}.
\]
In fact, we claim that the reduced cohomology on $\mathbb P^1$ of $f^\ast\otimes \mathbb Z(1)[2]$ agrees with the identity functor. This reduced cohomology of $\mathbb P^1$ agrees with the reduced cohomology of $\mathbb G_m$ (by ball-invariance) up to shift, and the latter can be computed by taking homomorphisms from $\mathbb Z_{\mathrm{mot}}[\mathbb G_m]/\mathbb Z_{\mathrm{mot}}[\ast]$, which agrees with $\overline{\mathbb G}_m=\mathbb Z(1)[1]$, so this follows from the $\otimes$-invertibility of $\mathbb Z(1)$ in $\mathcal D_{\mathrm{mot}}$.

We need to check that certain induced endomorphisms of the functors $f^\ast$ and $f_\ast$ are isomorphisms. Consider first $f_\ast$. This is a composite
\[
f_\ast\xrightarrow{f_\ast \alpha} f_\ast f^\ast(f_\ast\otimes\mathbb Z(1)[2]) = f_\ast(f^\ast f_\ast\otimes\mathbb Z(1)[2])\xrightarrow{\beta f_\ast} f_\ast.
\]
It suffices to check the result for objects in the image of $j_!$ where $j: \mathbb A^1_X\to \mathbb P^1_X$ is the open immersion (it is then also the case fo other points at infinity, and these generate). We write $f^\circ_!$ for $f_\ast j_!$; after we completed this proof, this is the $!$-functor for $f^\circ: \mathbb A^1_X\to X$. It suffices to see that the resulting composite
\[
f^\circ_!\xrightarrow{f^\circ_! \alpha^\circ} f^\circ_! f^{\circ\ast}(f^\circ_!\otimes \mathbb Z(1)[2])\xrightarrow{\beta^\circ f^\circ_!} f^\circ_!
\]
is an isomorphism, where $\alpha^\circ$ and $\beta^\circ$ are the obvious variants. But $\beta^\circ$ is already an isomorphism, by the construction of the counit. It remains to see that $f^\circ_! \alpha^\circ$ is an isomorphism. This is the result of applying $(f^\circ\times f^\circ)_! = f^\circ_! p_{2!}^\circ$ to the map
\[
\Delta^\circ_\ast M\to p_2^{\circ\ast} M\otimes \mathbb Z(1)[2]
\]
in $\mathcal D_{\mathrm{mot}}(\mathbb A^2_X)$, for any $M\in \mathcal D_{\mathrm{mot}}(\mathbb A^1_X)$. In fact, we claim that this is already an isomorphism after applying $p_{2!}^\circ$. This can be checked on geometric fibres. Now we have some $M'\in \mathcal D_{\mathrm{mot}}(X)$ (a fibre of $M$), where $X=\mathcal M_{\mathrm{arc}}(C)$, and a section $s: X\to \mathbb A^1_X$, and apply $f^\circ_!$ to the map
\[
s_\ast M'\to f^{\circ\ast} M'\otimes \mathbb Z(1)[2].
\]
By another application of the computation in the construction of the counit, we know that the global sections of both sides evaluate to $M'(C)$, and it remains to identify the map. This can be done for $M'=\mathbb Z$, where it is a simple computation.

On the other hand, for $f^\ast$, checking that an endomorphism of $f^\ast M$ is an isomorphism can also be done on geometric fibres, and we arrive at a certain natural endomorphism of $M(C)$. Identifying this can again be done for $M=\mathbb Z$. Identifying the endomorphism of $f^\ast \mathbb Z$ can be done after applying $f^\circ_!$. But then $f^\circ_! \alpha^\circ f^{\circ\ast}$ and $f^\circ_! f^{\circ\ast} \beta^\circ$ are individually isomorphisms, by what we have already proved.
\end{proof}

For the following discussion of six-functor formalisms, we make extensive use of the references \cite{ScholzeSixFunctors} and \cite{HeyerMann}. Starting with the category of affine arc-sheaves, and the class of maps of finite cohomological dimension as the $!$-able (and proper) maps, we get a $6$-functor formalism $X\mapsto \mathcal D_{\mathrm{mot}}(X)$. By general extension procedures, cf.~\cite[Theorem 5.19]{ScholzeSixFunctors} and \cite[Theorem 3.4.11]{HeyerMann}, we get a $6$-functor formalism $X\mapsto \mathcal D_{\mathrm{mot}}(X)$ on all small arc-stacks, for a certain class of $!$-able maps of small arc-stacks.

\begin{proposition} Open immersions $j: U\hookrightarrow X$ are $!$-able and $j_!$ is the left adjoint of pullback $j^\ast$.
\end{proposition}

\begin{proof} We can check this for $X=\mathcal M_{\mathrm{arc}}(A)$ affine. It suffices to cover $U$ by $!$-able open subsets with the desired property; then it follows for $U$ by gluing. So we can assume $U=\{|f|<1\}$ for some $f\in A$. This is the union of the affine subsets $\{|f|\leq c\}$ over $c<1$, and these form a $!$-cover of $U$. It follows that $U$ is $!$-able, and it is a standard exercise to compute the resulting $j_!$. (Indeed, it is colimit-preserving and has the correct restriction to all $\{|f|\leq c\}$.)
\end{proof}

\begin{proposition} All \'etale maps (i.e.~finitary maps with finite geometric fibres) are $!$-able and cohomologically \'etale.
\end{proposition}

\begin{proof} \'Etale maps are \'etale locally on the target covered by disjoint unions of open subsets.
\end{proof}

In the following corollary, when working over a Banach field $K$, we use Berkovich's notion of smooth $K$-analytic spaces. These correspond to smooth separated rigid-analytic spaces over $K$.

\begin{corollary}\label{cor:smoothcohomsmooth} Let $K$ be a Banach field and let $X$ be a smooth $K$-analytic space without boundary. Then $X$ is cohomologically smooth.
\end{corollary}

\begin{proof} Any smooth space is locally \'etale over $\mathbb A^n$, reducing us to the previous result (noting that boundaryless \'etale maps of $K$-Berkovich spaces are \'etale in our sense; this reduces to the case of finite \'etale maps and open immersions, which are clear).
\end{proof}

Also note that as under tilting, the first Tate twist is preserved, we get:

\begin{proposition}\label{prop:fulltilting} Let $X$ be a small arc-stack on which $|p|<1$, and let $X^\flat$ be its tilt. Then tilting induces a natural symmetric monoidal equivalence
\[
\mathcal D_{\mathrm{mot}}(X)\cong \mathcal D_{\mathrm{mot}}(X^\flat),
\]
functorial in $X$.
\end{proposition}

\begin{proof} We already proved the effective version in Proposition~\ref{prop:effectivetilting}. Inverting the Tate twist (which is preserved under tilting), the result follows.
\end{proof}

\section{$\mathcal D_{\mathrm{mot}}(C)$}

Next, we analyze the category $\mathcal D_{\mathrm{mot}}(C)$ when $C$ is a non-discrete algebraically closed Banach field. When $C=\mathbb C$, this is just $\mathcal D(\mathbb Z)$.

\begin{proposition}\label{prop:compactgenerationpoint} For any algebraically closed Banach field $C$, the category $\mathcal D_{\mathrm{mot}}(C)$ is compactly generated, the unit is compact, and all compact objects are dualizable. If $C$ has equal characteristic and we pick a splitting $k\to C$ of its residue field $k$, then a generating family of compact objects is given by $\mathbb Z_{\mathrm{mot}}[X_C](-j)$ where $X$ is a smooth projective variety over $k$, and $j\geq 0$.

In particular, $\mathcal D_{\mathrm{mot}}(C)$ is rigid.
\end{proposition}

\begin{proof} By Proposition~\ref{prop:fulltilting}, we can assume that $C$ has equal characteristic. Pick a splitting $k\to C$ of its residue field. As in any six-functor formalism, pushforward along cohomologically smooth and proper maps preserves dualizable objects, it is clear that $\mathbb Z_{\mathrm{mot}}[Y](-j)$ is dualizable when $Y$ is smooth projective over $C$, in particular this happens for $Y=X_C$. Also, the unit is compact, so all dualizable objects are compact. For the rest, it suffices to see that the $\mathbb Z_{\mathrm{mot}}[X_C](-j)$ generate $\mathcal D_{\mathrm{mot}}(C)$. Already $\mathbb Z$ itself generates the torsion part, so we can work with rational coefficients. Assume that $M\in \mathcal D_{\mathrm{mot}}(C)$ has the property that it has no nonzero maps from $\mathbb Z_{\mathrm{mot}}[X_C](-j)$ for all such $X$ and $j$; we need to show that $M=0$. It suffices to show that $M(C')=0$ for all complete algebraically closed extensions $C'/C$. We argue by induction on the topological transcendence degree. If $C_d/C$ is some extension of topological transcendence degree $d$, pick some $C_{d-1}\subset C_d$ of topological transcendence degree $d-1$ such that $C_d/C_{d-1}$ is of transcendence degree $1$. It then corresponds to a completed algebraic closure of one of the types of points (2), (3) or (4) on the Berkovich line over $C_{d-1}$. The residue field at the point may not be algebraically closed, but any of its finite extensions can also be realized by a point of the same type in the Berkovich line over $C_{d-1}$. Type (4) points are intersections of balls, so evaluation at $C_d$ agrees with evaluation at $C_{d-1}$, giving the result by induction. Type (3) points are intersections of annuli, so evaluation at $C_d$ agrees with evaluation at $C_{d-1}$ of Tate twists, so again we get the result by induction. Thus, only type (2) remains. But if the residue $k_d$ of $C_d$ has transcendence degree $<d$ over the residue field $k$ of $C$, we could always choose $C_{d-1}$ to already have residue field $k_{d-1}=k_d$, thereby avoiding type (2). Thus, only the case that $k_d$ has transcendence degree $d$ is relevant. Write $\mathcal O_{C_d}$ as a completed filtered colimit of algebras topologically of finite type $R_i$ over $\mathcal O_C$. These algebras are generically smooth, so we can assume that the $R_i$ are smooth. We can then find an isomorphism $R_i\cong \overline{R}_i\hat{\otimes}_k \mathcal O_C$ for smooth $k$-algebras $\overline{R}_i$. Using alterations, we can arrange that $\mathrm{Spec}(\overline{R}_i)\subset X_i$ has a smooth projective compactification $X_i$ whose boundary is a normal crossing divisor. Now the motive of $\overline{R}_i\hat{\otimes}_k C$ can be described in terms of the motive of $X_{i,C}$, and lower-dimensional varieties, by induction; and the completed colimit of the $\overline{R}_i\hat{\otimes}_k C$ is $C_d$ by construction.
\end{proof}

Let us note that once we know this, we can prove the projection formula in a slightly more general setting.

\begin{corollary}\label{cor:generalprojectionformula} Let $f: Y\to X$ be a proper $0$-truncated map of locally finite cohomological dimension between small arc-stacks. Then
\[
f_\ast: \mathcal D_{\mathrm{mot}}(Y)\to \mathcal D_{\mathrm{mot}}(X)
\]
commutes with colimits and satisfies base change and the projection formula.
\end{corollary}

\begin{proof} We already know commutation with colimits and base change. For the projection formula, we can then reduce to the case that $X=\mathcal M_{\mathrm{arc}}(C)$ is a geometric point. But then $\mathcal D_{\mathrm{mot}}(X)$ is compactly generated by dualizables; in this case, the projection formula is automatic.
\end{proof}

We can then rerun the construction of $6$-functor formalisms using this class of proper maps (still between affine arc-stacks)\footnote{In \cite[Theorem 3.4.11]{HeyerMann}, the resulting class of $!$-able maps depends on the choice of basic geometric objects, in this case affine arc-stacks: Namely, to check that a map of stacks is $!$-able, it suffices to check this after pullback to each affine arc-stack. If we removed the affineness condition here, the resulting condition would be harder to check in practice.}, to get a more general class of $!$-able maps.

To understand maps in $\mathcal D_{\mathrm{mot}}(C)$, one can reduce to understanding maps from smooth projective $X$ towards Tate twists in $\mathcal D_{\mathrm{mot}}^{\mathrm{eff}}(C)$. These are all recorded in $\overline{K}(X)$. This is easier to compute when $X$ has good reduction, and fortunately the above result says that these are enough.

In the relative setting, we have the following result.

\begin{proposition}\label{prop:rigidity} Let $A$ be an analytic Banach ring of finite cohomological dimension. Then $\mathcal D_{\mathrm{mot}}(\mathcal M_{\mathrm{arc}}(A))$ is rigid.

More generally, assume that $X$ is a qcqs arc-sheaf of finite cohomological dimension such that for each $x\in |X|$ there is some non-discrete algebraically closed Banach field $C(x)$ and a quasi-pro-\'etale map $\mathcal M_{\mathrm{arc}}(C(x))\to X$ with image $x$. Then $\mathcal D_{\mathrm{mot}}(X)$ is rigid.
\end{proposition}

The underlying topological space $|X|$ is defined via left Kan extension from the affine case. If $X$ is qcqs, then choosing a surjection $\mathcal M_{\mathrm{arc}}(A)\to X$ and $\mathcal M_{\mathrm{arc}}(B)$ surjecting onto the equivalence relation, we can write $|X|$ as the quotient $\mathcal M(A)/\mathcal M(B)$. This is a quotient of a compact Hausdorff space by a closed equivalence relation and hence itself compact Hausdorff. Similarly, the underlying topological space of $\mathbb A^n_X$ is a locally compact Hausdorff space.

\begin{proof} The unit is compact by assumption on finite cohomological dimension. First we check dualizability; equivalently, being compactly assembled. For this, it suffices to exhibit sufficiently many compact maps. But whenever $U\subset V\subset \mathbb A^n_X$ is an inclusion of open subsets that factors over a quasicompact subset $K\subset \mathbb A^n_X$, then the map $\mathbb Z_{\mathrm{mot}}[U]\to \mathbb Z_{\mathrm{mot}}[V]$ is compact, as the induced map $\mathrm{Hom}(V,-)\to \mathrm{Hom}(U,-)$ factors over evaluation at $K$, and evaluation at $K$ commutes with filtered colimits. As any $V\subset \mathbb A^n_X$ can be written as a union of such $U$ strictly contained in $V$, we see that $\mathbb Z_{\mathrm{mot}}[V]$ can be written as a filtered colimit with compact transition maps. To deduce dualizability of $\mathcal D_{\mathrm{mot}}(X)$, it remains to see that those objects generate $\mathcal D_{\mathrm{mot}}^{\mathrm{eff}}(X)$ (then their negative Tate twists generate all of $\mathcal D_{\mathrm{mot}}(X)$).

Let $\mathcal C\subset \mathcal D_{\mathrm{mot}}^{\mathrm{eff}}(X)$ be the subcategory generated $\mathbb Z_{\mathrm{mot}}[U]$ where $U\subset \mathbb A^n_X$ is open. By the preceding paragraph, we see that $\mathcal C$ is compactly assembled and that the inclusion preserves compact morphisms, and hence its right adjoint preserves colimits. In particular, its formation commutes with passing to stalks of $|X|$, and so the claim $\mathcal C=\mathcal D_{\mathrm{mot}}^{\mathrm{eff}}(X)$ can be reduced to the case that $|X|$ is a point. Then $X=\mathcal M_{\mathrm{arc}}(C)/G$ for some profinite group $G$ acting continuously on an algebraically closed non-discrete Banach field $C$ (see \cite[Proposition 21.9]{ECoD} for the argument). Now the result can be proved in the same way as Proposition~\ref{prop:freemotivicsheavesgenerate}, using $G$-equivariant open subsets of $\mathbb A^n_C$. The points of $\mathbb A^n_C/G$ are then of the form $\mathcal M_{\mathrm{arc}}(C')/H$ for a profinite group $H$ acting faithfully on complete algebraically closed nonarchimedean fields $C'$, and the arguments still apply in this case, replacing finite extension fields of Banach fields by open subgroups of $H$.

We note that our proof of dualizability shows at the same time that the tensor product functor has the property that the tensor product of two compact morphisms is compact (as this is evident for the generating class of compact morphisms). In particular, a tensor product of two compact objects must again be compact.

To prove rigidity, let us first stay in the case that $X=\mathcal M_{\mathrm{arc}}(C)/G$ is one point. Then we claim that $\mathcal D_{\mathrm{mot}}(X)$ is generated by dualizable objects. Picking a cofinal system of open subgroups $G_i\subset G$, set $X_i=\mathcal M_{\mathrm{arc}}(C)/G_i$. Then the cofiltered limit of the $X_i$ is $\mathcal M_{\mathrm{arc}}(C)$. By Lemma~\ref{lem:spreadingcompactobjects} below, any compact object of $\mathcal M_{\mathrm{arc}}(C)$ descends to some $X_i$, as does its dual, as well as the duality maps; thus, it descends to a compact dualizable object of some $X_i$. Then also the pushforward along the finite \'etale map $X_i\to X$ is compact dualizable, and these objects generate.

In general, it suffices to see that there is a generating family of objects that are sequential colimits along trace-class maps, cf.~\cite[Corollary 4.57]{RamziRigid}. Pick any $x\in X$ and write it as a cofiltered limit of closed neighborhoods $Z_i\subset X$. Using Lemma~\ref{lem:spreadingcompactobjects} again, the compact dualizable objects in the fibre at $x$ spread to some $Z_i$. If $M\in \mathcal D_{\mathrm{mot}}(Z_i)$ is compact dualizable and $U_i\subset Z_i$ is an open neighborhood of $x$ in $X$, without loss of generality a countable union of closed subsets, then the extension by $0$ from $U_i$ to $X$ of $M|_{U_i}$ is such a sequential colimit along trace-class maps, and these objects generate.
\end{proof}

\begin{lemma}\label{lem:quasiproetalebarrbeck} Let $f: Y\to X$ be a quasi-pro-\'etale proper map. Then the natural functor
\[
\mathcal D_{\mathrm{mot}}(Y)\to \mathrm{Mod}_{f_\ast \mathbb Z}(\mathcal D_{\mathrm{mot}}(X))
\]
induced by $f_\ast$ is an equivalence.

In particular, if $Y=\varprojlim_i Y_i$ is a cofiltered limit of quasi-pro-\'etale proper $f_i: Y_i\to X$, then as presentable stable $\infty$-categories,
\[
\mathcal D_{\mathrm{mot}}(Y) = \mathrm{colim}_i \mathcal D_{\mathrm{mot}}(Y_i).
\]
\end{lemma}

\begin{proof} The map $f$ has cohomological dimension $0$, so $f_!=f_\ast$ commutes with colimits and satisfies the projection formula. Moreover, $f_\ast$ is conservative: This can be checked after passing to geometric fibres, but when $X$ is a geometric point, then $Y$ is a profinite set times a geometric point, where global sections are conservative. Now the first result follows from Barr--Beck--Lurie \cite[Theorem 4.7.3.5]{LurieHA}. The colimit statement is a consequence of $f_\ast \mathbb Z= \mathrm{colim}_i f_{i\ast} \mathbb Z$ and the behaviour of module categories under colimits.
\end{proof}

\begin{lemma}\label{lem:spreadingcompactobjects} Let $(\mathcal C_i)_i$ be a filtered diagram of dualizable stable $\infty$-categories, with colimit-preserving functors whose right adjoints also commute with colimits. Let $\mathcal C=\mathrm{colim}_i \mathcal C_i$, itself dualizable. Passing to compact objects, the natural functor
\[
\mathrm{colim}_i \mathcal C_i^\omega\to \mathcal C^\omega
\]
of small stable $\infty$-categories is an equivalence.
\end{lemma}

\begin{proof} See \cite[Proposition 1.72]{EfimovDualizable}.
\end{proof}

As a consequence, we get a weak form of the categorical K\"unneth formula.

\begin{corollary}\label{cor:weakkunneth} Let $X\xrightarrow{f} S\xleftarrow{g} Y$ be a diagram of qcqs arc-stacks as in Proposition~\ref{prop:rigidity}. Then the exterior tensor product defines a fully faithful functor
\[
\mathcal D_{\mathrm{mot}}(X)\otimes_{\mathcal D_{\mathrm{mot}}(S)} \mathcal D_{\mathrm{mot}}(Y)\to \mathcal D_{\mathrm{mot}}(X\times_S Y).
\]
\end{corollary}

\begin{proof} Relative tensor products of rigid categories remain rigid, and for any symmetric monoidal functor of rigid categories, the right adjoint functor is colimit-preserving and linear over the base category. It follows that it suffices to check that the unit of the adjunction becomes an equivalence after evaluation at the unit, or the slightly more refined statement that it holds after pushforward to $\mathcal D_{\mathrm{mot}}(S)$. This means that we have to check that for all $M\in \mathcal D_{\mathrm{mot}}(X)$ and $N\in \mathcal D_{\mathrm{mot}}(Y)$ with associated $M\boxtimes N\in \mathcal D_{\mathrm{mot}}(X\times_S Y)$, the natural map
\[
f_\ast M\otimes g_\ast N\to (f\times_S g)_\ast(M\boxtimes N)
\]
is an isomorphism. But this is the usual K\"unneth formula which holds in any six-functor formalism (as a consequence of the projection formula).
\end{proof}

\subsection{Mixed Tate motives} For any small arc-stack $X$, we define a subcategory of mixed Tate motives $\mathcal D_{MT}(X)\subset \mathcal D_{\mathrm{mot}}(X)$, as follows. (This might more properly be called mixed Artin--Tate, as we check the condition only at geometric points.)

\begin{definition} An object $M\in \mathcal D_{\mathrm{mot}}(X)$ lies in $\mathcal D_{MT}(X)$ if its fibre at all geometric points $\mathcal M_{\mathrm{arc}}(C)\to X$, $C$ a nondiscrete algebraically closed Banach field, lies in the subcategory generated by all Tate twists $\mathbb Z(n)$, $n\in \mathbb Z$.
\end{definition}

Then $\mathcal D_{MT}(X)\subset \mathcal D_{\mathrm{mot}}(X)$ is stable under tensor products and pullbacks, but not in general under $f_!$.

\begin{proposition} The functor $X\mapsto \mathcal D_{MT}(X)$ defines an arc-sheaf of (symmetric monoidal presentable stable) $\infty$-categories.
\end{proposition}

\begin{proof} It suffices to see that if an object $M\in \mathcal D_{\mathrm{mot}}(X)$ lies in $\mathcal D_{MT}(X)$ arc-locally, then it lies in $\mathcal D_{MT}(X)$. From the definition, it suffices to check this when $X=\mathcal M_{\mathrm{arc}}(C)$ and for an arc-cover by $X'=\mathcal M_{\mathrm{arc}}(C')$. This is nontrivial only when $C$ is nonarchimedean. Take any $M\in \mathcal D_{\mathrm{mot}}(X)$. As $\mathcal D_{\mathrm{mot}}(X)$ is compactly generated, we can write $M=\mathrm{colim}_i M_i$ as a filtered colimit of compact $M_i$. Assume that $M|_{C'}$ is mixed-Tate. Then there is another presentation $M|_{C'} = \mathrm{colim}_j M'_j$ where all $M'_j$ are compact and mixed-Tate. The two presentations $\mathrm{colim}_i M_i|_{C'} = M|_{C'} = \mathrm{colim}_j M'_j$ must be Ind-isomorphic. It follows that for any $i$, there is some $i'\geq i$ so that the map $M_i|_{C'}\to M_{i'}|_{C'}$ factors over a mixed-Tate object. It suffices to see that in this case also $M_i\to M_{i'}$ factors over a mixed-Tate object. To see this, write $C'$ as a completed filtered colimit of algebraically closed Banach fields $C_k$ that have finite topological transcendence degree over $C$. Both the mixed-Tate object and the factorization of $M_i|_{C'}\to M_{i'}|_{C'}$ over it are then already defined over some $C_k$; so we may assume that $C'$ has finite topological transcendence degree. We can then write $C'$ as a completed filtered colimit of smooth $C$-algebras topologically of finite type $B_k$, of uniformly bounded dimension. Then again, the mixed-Tate object and the factorization $M_i|_{C'}\to M_{i'}|_{C'}$ over it are already defined over some $B_k$. But $B_k$ admits a section $B_k\to C$, thus yielding the factorization already over $C$.
\end{proof}

For strictly totally disconnected $X$, the point-wise condition in the definition becomes global.

\begin{proposition} For strictly totally disconnected $X$, the subcategory
\[
\mathcal D_{MT}(X)\subset \mathcal D_{\mathrm{mot}}(X)
\]
is the full subcategory generated by all Tate twists $\mathbb Z(n)$, $n\in \mathbb Z$.
\end{proposition}

\begin{proof} Let $\mathcal D_{MT}'(X)\subset \mathcal D_{\mathrm{mot}}(X)$ be the subcategory generated by all Tate twists. The inclusion $\mathcal D_{\mathrm{MT}}'(X)\subset \mathcal D_{\mathrm{mot}}(X)$ has a right adjoint $R$ commuting with all colimits (as the inclusion preserves compact objects). Assume that $M\in \mathcal D_{MT}(X)$. The map $R(M)\to M$ in $\mathcal D_{\mathrm{mot}}(X)$ becomes an isomorphism at all stalks by the assumption and the commutation of $R$ with colimits. Thus, it is an isomorphism. Therefore $\mathcal D_{MT}(X)\subset \mathcal D_{MT}'(X)$, and the other inclusion is clear.
\end{proof}

\begin{proposition} If $X$ is a qcqs arc-sheaf as in Proposition~\ref{prop:rigidity}, then $\mathcal D_{MT}(X)$ is rigid.
\end{proposition}

\begin{proof} The same proof as before works.
\end{proof}

\section{$\mathcal D_{\mathrm{mot}}(k)$}

Using the results of the previous section, we can now also understand the structure of $\mathcal D_{\mathrm{mot}}(k)$, and in fact compare it to Voevodsky's theory, for any (discrete) field $k$, considered as a Banach ring with the trivial norm. We consider the case that $k$ is algebraically closed; one can reduce the general case to this case by Galois descent (but may have to be careful with descent versus hyperdescent questions, depending on the precise definitions used).

\begin{theorem} For any algebraically closed discrete field $k$, the category $\mathcal D_{\mathrm{mot}}(k)$ is generated by compact dualizable objects $\mathbb Z_{\mathrm{mot}}[X](-j)$ where $X$ is a smooth projective variety over $k$, and $j\in \mathbb Z$. Moreover,
\[
\mathrm{Hom}_{\mathcal D_{\mathrm{mot}}(k)}(\mathbb Z_{\mathrm{mot}}[X](-j),\mathbb Z)
\]
is, after rationalization, given by the weight $j$ Adams piece of $K(X)$; while after profinite completion (prime to the characteristic $p$ of $k$) it is given by pro\'etale cohomology of $X$ with coefficients in $\hat{\mathbb Z}^p(j)$.
\end{theorem}

\begin{remark} Recall that Voevodsky's category of \'etale motives over $k$ is the free presentable stable $\infty$-category generated by smooth schemes over $k$, subject to \'etale descent, $\mathbb A^1$-invariance, and $\mathbb G_m$-stability. This induces a comparison functor to $\mathcal D_{\mathrm{mot}}(k)$, and then the previous theorem and the similar description of Voevodsky's category imply that this comparison functor is an equivalence.
\end{remark}

\begin{proof} Let $K=k((T))_{1/2}$ be the Laurent series on a free variable $T$ with exact norm $\tfrac 12$, with $f: \mathcal M_{\mathrm{arc}}(K)\to \mathcal M_{\mathrm{arc}}(k)$. This is a closed subset of $g: \mathbb G_{m,k}\to \mathcal M_{\mathrm{arc}}(k)$, and the map $g_\ast \mathbb Z\to f_\ast \mathbb Z$ is an isomorphism (use Proposition~\ref{prop:tatetwistcompact}, pass to duals, and pass to limits resp.~colimits over increasing resp.~shrinking annuli, using finitaryness). The cone of $\mathbb Z\to f_\ast \mathbb Z$ is isomorphic to $\mathbb Z(-1)[-1]$. In particular, $f_\ast\mathbb Z$ is dualizable. We claim that its dual $(f_\ast \mathbb Z)^\vee$ is representing the functor $M\mapsto \mathrm{Hom}(1,f^\ast M)$. Indeed,
\[
\mathrm{Hom}((f_\ast \mathbb Z)^\vee,M) = \mathrm{Hom}(1,M\otimes f_\ast \mathbb Z) = \mathrm{Hom}(1,f_\ast f^\ast M)=\mathrm{Hom}(1,f^\ast M).
\]
In particular, the left adjoint to $f^\ast$ exists on the unit, with value $(f_\ast \mathbb Z)^\vee$. By Proposition~\ref{prop:compactgenerationpoint}, $\mathcal D_{\mathrm{mot}}(K)$ is generated by the image of $f^\ast$ and Artin motives (which can be handled by passing to finite extensions of $K$, which are also Laurent series fields), it follows that the left adjoint $f_\sharp$ exists in general.

It also follows that $f_\sharp \mathbb Z = (f_\ast \mathbb Z)^\vee$ is a compact object. But $f_\ast \mathbb Z\cong g_\ast \mathbb Z$ has $\mathbb Z$ as a summand, using say the unit section $1\in \mathbb G_{m,k}$. Thus $\mathbb Z$ is also compact. Moreover, $\mathcal D_{\mathrm{mot}}(k)$ is generated by $f_\sharp \mathcal D_{\mathrm{mot}}(K)$ as $f^\ast$ is conservative, which in turn is generated by
\[
f_\sharp f^\ast \mathbb Z_{\mathrm{mot}}[X](-j) = \mathbb Z_{\mathrm{mot}}[X](-j)\otimes f_\sharp \mathbb Z
\]
for smooth projective $X$ over $k$ and $j\in \mathbb Z$, as well as similar objects arising from finite extensions of $K$. Using the splitting of $f_\sharp \mathbb Z$, we see that the compact dualizable objects $\mathbb Z_{\mathrm{mot}}[X](-j)$ form generators.

The explicit description of morphisms comes from the explicit description of higher weight motivic cohomology in terms of $\overline{K}$. When $X$ is smooth projective over $k$, the base change to the completed algebraic closure $C$ of $K$ has
\[
\overline{K}(X_C)_{\mathbb Q} = \mathrm{cofib}(K(X_{\mathfrak m_C})\to K(X_C))_{\mathbb Q}
\]
(where we write $K(X_{\mathfrak m_C})$ for the fiber of $K(X_{\mathcal O_C})\to K(X)$), which can also be written as the pushout
\[\xymatrix{
K(X_{\mathcal O_C})_{\mathbb Q}\ar[r]\ar[d] & K(X_C)_{\mathbb Q}\ar[d]\\
K(X)_{\mathbb Q}\ar[r] & \overline{K}(X_C)_{\mathbb Q}
}\]
or as the cofiber of
\[
K(X_{\mathcal O_C}\ \mathrm{on}\ X)_{\mathbb Q}\to K(X)_{\mathbb Q}.
\]
But writing $\mathcal O_C$ as a completed filtered colimit of finite extensions of $k[[T]]$, and using excision on regular schemes, the term $K(X_{\mathcal O_C}\ \mathrm{on}\ X)_{\mathbb Q}$ is isomorphic to $K(X)_{\mathbb Q}$; with Adams weight spaces shifted by $1$. This extra piece matches the second summand of $f_\sharp \mathbb Z$ (argue by induction on Adams weight).
\end{proof}

We can now also go back and describe $\mathcal D_{\mathrm{mot}}(C)$ in terms of $\mathcal D_{\mathrm{mot}}(k)$ for the completed algebraic closure $C$ of $K=k((T))_{\tfrac 12}$. Indeed, let $f: \mathcal M_{\mathrm{arc}}(C)\to \mathcal M_{\mathrm{arc}}(k)$ be the projection. Then
\[
f^\ast: \mathcal D_{\mathrm{mot}}(k)\to \mathcal D_{\mathrm{mot}}(C)
\]
is a symmetric monoidal functor of rigid compactly generated presentable stable $\infty$-categories, where the image of $f^\ast$ generates the target. It follows formally (from Barr--Beck--Lurie) that
\[
\mathcal D_{\mathrm{mot}}(C) = \mathrm{Mod}_{f_\ast \mathbb Z}(\mathcal D_{\mathrm{mot}}(k)).
\]
If we pick a pseudouniformizer $\pi$ with compatible roots $\pi^{1/n}$ for all $n$ prime to the residue characteristic $p$ of $C$, we get a map $C\to \tilde{\mathbb G}_{m,k}$ to $\tilde{\mathbb G}_{m,k} = \mathrm{lim}_{x\mapsto x^n} \mathbb G_{m,k}$. Let $g: \tilde{\mathbb G}_{m,k}\to \mathcal M_{\mathrm{arc}}(k)$ be the projection. The induced map
\[
f_\ast \mathbb Z\to g_\ast \mathbb Z
\]
is an isomorphism, by direct computation. (Indeed, by adjunction both sides have compatible maps from $\mathbb Z$, and the cofibre is $\mathbb Q(-1)[-1]$. This follows by passing to colimits in the description at finite levels; for $\tilde{\mathbb G}_m$ this has to be combined with a limit over increasing annuli, noting that all transition maps in the limit are isomorphisms. The induced endomorphism of $\mathbb Q(-1)[-1]$ is nonzero, by comparison to any finite level, so must be an isomorphism.) But $g_\ast \mathbb Z$ has a natural map back to $\mathbb Z$ given by $1\in \tilde{\mathbb G}_{m,k}$.

\begin{definition}\label{def:motivicnearby} The motivic nearby cycles functor is the symmetric monoidal functor
\[
\Psi = \Psi_{(\pi^{1/n})_n}: \mathcal D_{\mathrm{mot}}(C) = \mathrm{Mod}_{f_\ast \mathbb Z}(\mathcal D_{\mathrm{mot}}(k))\to \mathcal D_{\mathrm{mot}}(k)
\]
given by tensoring along the map
\[
f_\ast \mathbb Z\cong g_\ast \mathbb Z\to \mathbb Z
\]
induced by the choice of roots $\pi^{1/n}$.
\end{definition}

\begin{proposition} The map $f_\ast \mathbb Z\cong g_\ast \mathbb Z\to \mathbb Z$ is descendable in the sense of \cite[Definition 3.18]{MathewDescendable}. In particular, $\Psi$ is conservative, and we may regard $\mathcal D_{\mathrm{mot}}(C)$ as a ``neutral Tannakian stable $\infty$-category over $\mathcal D_{\mathrm{mot}}(k)$''.
\end{proposition}

\begin{proof} The fibre of $g_\ast \mathbb Z\to \mathbb Z$ is isomorphic to $\mathbb Q(-1)[-1]$ and the $g_\ast \mathbb Z$-module structure on it factors over $g_\ast \mathbb Z\to \mathbb Z$: Indeed, this can be checked rationally, and rationally $g_\ast \mathbb Q$ is the free $E_\infty$-algebra on $\mathbb Q(-1)[-1]$. Thus, $g_\ast \mathbb Z$ is filtered by modules inflated along $g_\ast \mathbb Z\to \mathbb Z$, which implies descendability.
\end{proof}

One can then describe $\mathcal D_{\mathrm{mot}}(C)$ as representations of an ``affine group scheme over $\mathrm{Spec}(\mathcal D_{\mathrm{mot}}(k))$''. We want to make this more explicit. As $\mathcal D_{\mathrm{mot}}(k)$ is inexplicit, this is a bit subtle. Fortunately, the relevant algebras lie in the subcategory of mixed Tate motives, and this can be made more explicit.

In the following, we specialize to $k=\overline{\mathbb F}_p$ for a prime $p$. This is in fact sufficient to describe $g_\ast \mathbb Z$ in general, at least when $k$ has positive characteristic, by noting that it is base changed from the case of $\overline{\mathbb F}_p$.

While describing the category $\mathcal D_{\mathrm{mot}}(\overline{\mathbb F}_p)$ explicitly is out of reach (it would entail proving the Tate conjecture), one can describe the subcategory generated by all Tate twists $\mathbb Z(j)$, $j\in \mathbb Z$. The key input is Quillen's computation of the $K$-theory of $\overline{\mathbb F}_p$; in fact, we need only the following simple statement.

\begin{proposition} After rationalization, $K(\overline{\mathbb F}_p)_{\mathbb Q} = \mathbb Q[0]$. In particular, rationally there are no maps between different Tate twists in $\mathcal D_{\mathrm{mot}}(\overline{\mathbb F}_p)$.
\end{proposition}

\begin{proof} It suffices to prove the same result for finite fields $\mathbb F_q$. Then it follows from $K(\mathbb F_q)$ being the group completion of $\bigsqcup_{n\geq 0} [\ast/\mathrm{GL}_n(\mathbb F_q)]$, and the vanishing of the rational group homology of finite groups.
\end{proof}

Let $\mathrm{MG}_k$ be the stack over $\mathbb Z[\tfrac 1p]$ classifying, over a ring $R$, a line bundle $L$ over $R$ together with compatible isomorphisms $L/^{\mathbb L} n\cong \mu_n(k)\dotimes_{\mathbb Z} R$ for all integers $n$ coprime to $p$.\footnote{The letters $\mathrm{MG}$ stand for ``motivic Galois'' -- the stack $\mathrm{MG}_k$ is roughly the classifying stack of a motivic Galois group of $k$.} Thus, $\mathrm{MG}_k$ maps to $\ast/\mathbb G_m$ and this map is a torsor under the functor taking any $R$ to $\varprojlim_n (R/^{\mathbb L} n)^\times$. A different description is that $\mathrm{MG}_k$ is isomorphic to the classifying space of the group $\mathbb G_m^{\mathrm{rat}}$ where
\[
\mathbb G_m^{\mathrm{rat}} = \mathrm{Spec}(\mathbb Q[T^{\pm 1}]\times_{\mathbb Q}\mathbb Z[\tfrac 1p])
\]
(where the map $\mathbb Q[T^{\pm 1}]\to \mathbb Q$ is evaluation at $T=1$). We note that $\mathcal D_{\mathrm{qc}}(\mathrm{MG}_k)$ is compactly generated by the powers $L^{\otimes n}$, $n\in \mathbb Z$. Indeed, to check that these dualizable objects are compact, it suffices to check that the unit is compact, for which it suffices to check that pushforward along $\mathrm{MG}_k\to \mathrm{Spec}(\mathbb Z[\tfrac 1p])$ preserves colimits. This can be done fibrewise, but the fibres of $\mathrm{MG}_k$ are $\mathrm{Spec}(\mathbb F_\ell)$ for finite $\ell\neq p$, and $\ast/\mathbb G_m$ over the generic point; in both cases, it is clear. Similarly, one proves that $L^{\otimes n}$ generate as they do in each fibre.

Note that $\mathcal D_{\mathrm{mot}}(k)$ is naturally linear over $\mathcal D_{\mathrm{qc}}(\mathrm{MG}_k)$. Indeed, the map to $\ast/\mathbb G_m$ is given by the Tate twist, while the lift to the torsor over it is given by the isomorphisms
\[
\mathbb Z(1)/n\cong \mu_n\cong \mu_n(\overline{\mathbb F}_q)\otimes \mathbb Z(0).
\]

\begin{proposition}\label{prop:tatemotives} The functor $\mathcal D_{\mathrm{qc}}(\mathrm{MG}_k)\to \mathcal D_{\mathrm{mot}}(k)$ is fully faithful, with essential image being the subcategory generated under colimits, shifts, and retracts by $\mathbb Z(n)$ for $n\in \mathbb Z$. This agrees with $\mathcal D_{MT}(k)\subset \mathcal D_{\mathrm{mot}}(k)$.
\end{proposition}

\begin{proof} Recall that $\mathcal D_{\mathrm{qc}}(\mathrm{MG}_k)$ is compactly generated by the tensor powers of the line bundle $L$. To prove the proposition, it remains to compare homomorphisms between generating objects, and this can be split up into the case of torsion coefficients, and rational coefficients. The case of modulo $n$ coefficients is clear as both categories become $\mathcal D(\mathbb Z/n\mathbb Z)$, while the rational case follows from the previous proposition.

It is clear that the subcategory generated by all $\mathbb Z(n)$ is contained in $\mathcal D_{MT}(k)$. For the converse, pull back to $\mathcal D_{MT}(C)$ and use that every object is a retract of pull-push along $\mathcal M_{\mathrm{arc}}(C)\to \mathcal M_{\mathrm{arc}}(k)$.
\end{proof}

We note that it follows that $\mathcal D_{MT}(k)$ has a natural $t$-structure coming from the equivalence with $\mathcal D_{\mathrm{qc}}(\mathrm{MG}_k)$. This is the ``motivic $t$-structure'', whose connective part is the one generated under colimits by $\mathbb Z(n)$, $n\in \mathbb Z$; after $\ell$-adic completion, it gives the usual $t$-structure on $\ell$-adic sheaves.

\begin{proposition} The natural functor
\[
\mathcal D_{MT}(C)\otimes_{\mathcal D_{MT}(k)} \mathcal D_{\mathrm{mot}}(k)\to \mathcal D_{\mathrm{mot}}(C)
\]
is an equivalence.
\end{proposition}

As before, $C$ denotes some completed algebraic closure of $k((T))_{1/2}$.

\begin{proof} Indeed, as $f_\ast \mathbb Z\in \mathcal D_{MT}(k)$, it follows formally that $\mathcal D_{MT}(C)\cong \mathrm{Mod}_{f_\ast \mathbb Z}(\mathcal D_{MT}(k))$, and then the statement follows from base change for module categories.
\end{proof}

Thus, it suffices to describe explicitly $\mathcal D_{MT}(C)$ over $\mathcal D_{MT}(k)$. We endow $\mathcal D_{MT}(C)$ with a $t$-structure, by taking the connective part as the part generated under colimits by all $\mathbb Z(n)$, $n\in \mathbb Z$.

\begin{proposition} The nearby cycles functor induces a conservative $t$-exact symmetric monoidal functor
\[
\Psi: \mathcal D_{MT}(C)\to \mathcal D_{MT}(k).
\]
\end{proposition}

\begin{proof} The functor is given by
\[
\mathcal D_{MT}(C)\cong \mathrm{Mod}_{f_\ast \mathbb Z}(\mathcal D_{MT}(k))\to \mathcal D_{MT}(k)
\]
induced by the map $f_\ast \mathbb Z\cong g_\ast \mathbb Z\to \mathbb Z$. It suffices to see that $f_\ast \mathbb Z\to \mathbb Z$ is coconnectively faithfully flat in $\mathcal D_{MT}(k)\cong \mathcal D_{\mathrm{qc}}(\mathrm{MG}_k)$, i.e.~tensoring preserves coconnective modules, and is conservative on them. Covering $\mathrm{MG}_k$ by $\mathrm{Spec}(\mathbb Z[\tfrac 1p])$, we have to see that the map from the free $E_\infty$-$\mathbb Z[\tfrac 1p]$-algebra on $\mathbb Q[-1]$ back to $\mathbb Z[\tfrac 1p]$ is coconnectively faithfully flat. This is a simple computation.
\end{proof}

Let $\mathrm{MG}_C\to \mathrm{MG}_k$ be the algebraic stack parametrizing, in addition to the line bundle $L/R$ with compatible isomorphisms $L/n\cong \mu_n(k)\otimes_\mathbb Z R$, a map
\[
C^\times/(1+C_{<1})[-1]\to L
\]
in $\mathcal D(\mathbb Z)$ that after reduction modulo $n$ yields the natural map
\[
\mu_n(k)\cong (C^\times/(1+C_{<1})[-1])/^\mathbb L n\to L/^{\mathbb L} n,
\]
compatibly for varying $n$. Equivalently, this is the datum of an extension
\[
L\to \tilde{L}\to C^\times/(1+C_{<1})\otimes_{\mathbb Z} R
\]
in $\mathcal D(\mathbb Z)$ which after pullback to $\mathcal O_C^\times/(1+C_{<1})\otimes_{\mathbb Z} R\cong L\otimes \mathbb Q/\mathbb Z[\tfrac 1p]$ is identified with
\[
L\to L\otimes \mathbb Q\to L\otimes \mathbb Q/\mathbb Z[\tfrac 1p].
\]
As $C^\times/\mathcal O_C^\times\cong \mathbb Q$, the map $\mathrm{MG}_C\to \mathrm{MG}_k$ is a gerbe banded by $\mathrm{Hom}(\mathbb Q,L)$. This is an affine unipotent group scheme, and so
\[
\mathcal D_{\mathrm{qc}}(\mathrm{MG}_C)\cong \mathrm{Mod}_{h_\ast \mathcal O} \mathcal D_{\mathrm{qc}}(\mathrm{MG}_k).
\]
Moreover, $h_\ast \mathcal O$ is a complex which in degree $0$ is given by $\mathcal O$, and in degree $1$ by $L^\vee\otimes \mathbb Q$, and $0$ otherwise.

The identification $\mathbb Z(1)(C)\cong C^\times/(1+C_{<1})[-1]$ yields a natural $\mathcal D_{\mathrm{qc}}(\mathrm{MG}_k)\cong \mathcal D_{MT}(k)$-linear symmetric monoidal functor
\[
\mathcal D_{\mathrm{qc}}(\mathrm{MG}_C)\to \mathcal D_{MT}(C).
\]

\begin{theorem} The functor $\mathcal D_{\mathrm{qc}}(\mathrm{MG}_C)\to \mathcal D_{\mathrm{MT}}(C)$ is an equivalence. Consequently,
\[
\mathcal D_{\mathrm{mot}}(C)\cong \mathcal D_{\mathrm{mot}}(k)\otimes_{\mathcal D_{\mathrm{qc}}(\mathrm{MG}_k)} \mathcal D_{\mathrm{qc}}(\mathrm{MG}_C).
\]
\end{theorem}

\begin{proof} By our preparations, it suffices to see that $h_\ast \mathcal O\cong f_\ast \mathbb Z$ as $E_\infty$-algebras under the equivalence $\mathcal D_{\mathrm{qc}}(\mathrm{MG}_k)\cong \mathcal D_{MT}(k)$. The constructions give a natural map, and we know both sides explicitly as complexes: They are $\mathbb Z[\tfrac 1p]\oplus \mathbb Q(-1)[-1]$. One only has to see that the map $\mathbb Q(-1)[-1]\to \mathbb Q(-1)[-1]$ is an isomorphism. But such endomorphisms are given by multiplication by some rational number, and it is easy to see that it must be nonzero. More precisely, the $\mathbb Q$ appearing is in both cases more canonically the value group of $C$, and if one uses this canonical description, the map is the identity.
\end{proof}

In other words, we translated the picture of the Tannakian category $\mathcal D_{\mathrm{mot}}(C)$ over $\mathcal D_{\mathrm{mot}}(k)$ into the picture of the gerbe $\mathrm{MG}_C$ over $\mathrm{MG}_k$, banded by the group scheme $\varprojlim_{n} L$ (the inverse limit over multiplication by positive integers $n$). A choice $\pi^{1/n}$ of roots of a uniformizer $\pi$ yields a splitting
\[
C^\times/(1+C_{<1})\cong \mathbb Q\times \mathcal O_C^\times/(1+C_{<1})
\]
and consequently a section of $\mathrm{MG}_C\to \mathrm{MG}_k$; this is responsible for the nearby cycles functor
\[
\Psi = \Psi_{(\pi^{1/n})_n}: \mathcal D_{\mathrm{mot}}(C)\to \mathcal D_{\mathrm{mot}}(k).
\]
Such a choice splits the gerbe, making quasicoherent sheaves on $\mathrm{MG}_C$ equivalent to representations of $\varprojlim_{n} L$. These are given by quasicoherent sheaves $M$ on $\mathrm{MG}_k$ together with a locally nilpotent map $M\to M\otimes L^\ast_{\mathbb Q}$.

\begin{corollary} The functor $\Psi$ induces an equivalence
\[
\mathcal D_{\mathrm{mot}}(C)\cong \{(A,N)\mid A\in \mathcal D_{\mathrm{mot}}(k), N: A\to A\otimes \mathbb Q(-1)\ \mathrm{locally}\ \mathrm{nilpotent}\}.
\]
Working with rational coefficients, the choice of $\pi$ alone yields an equivalence
\[
\mathcal D_{\mathrm{mot}}(C,\mathbb Q)\cong \{(A,N)\mid A\in \mathcal D_{\mathrm{mot}}(k,\mathbb Q), N: A\to A\otimes \mathbb Q(-1)\ \mathrm{locally}\ \mathrm{nilpotent}\}.
\]
\end{corollary}

This was previously proved by Binda--Gallauer--Vezzani \cite{BindaGallauerVezzani}.

\begin{proof} With rational coefficients, the choice of $\pi$ alone suffices to construct the section $f_\ast \mathbb Q\to \mathbb Q$, as one can replace $\tilde{\mathbb G}_m$ by $\mathbb G_m$ in the construction (as with $\mathbb Q$-coefficients they have the same cohomology).
\end{proof}

The previous corollary is very precious: In general, it is extremely difficult to describe $\mathcal D_{\mathrm{mot}}$, and any absolute description seems out of reach. Still, the previous corollary gives a complete relative description of $\mathcal D_{\mathrm{mot}}(C)$ in terms of $\mathcal D_{\mathrm{mot}}(k)$.

\section{Relation to v-stacks}

In this final section, we comment on the relation between the Berkovich-arc-stacks used in this paper to the adic v-stacks used in \cite{ECoD} and \cite{FarguesScholze}.

To match the setup of \cite{ECoD}, we focus on the characteristic $p$ situation, and only work with perfect analytic Banach rings $R$ over $\mathbb F_p$, i.e.~perfectoid $\mathbb F_p$-algebras equipped with a chosen norm. We already see that the first difference between the settings is that here we fixed a norm, which we did not do in \cite{ECoD}. This has the consequence that in the setting of this paper, the point $\ast$ is a qcqs arc-sheaf, while in \cite{ECoD} it is not quasiseparated. In the setting of \cite{ECoD}, one can however consider the v-sheaf $\mathcal N$ parametrizing norms; this is a canonical $\mathbb R_{>0}$-torsor $\mathcal N\to \ast$ over the point. We will thus compare the arc-stacks of this paper with the v-stacks over $\mathcal N$.

Now we have a natural forgetful functor $(R,R^+)\mapsto R$ from perfectoid Huber pairs $(R,R^+)$ with chosen norm towards Banach-$\mathbb F_p$-algebras. This has both adjoints. The right adjoint takes $R$ to $(R,R^\circ)$ while the left adjoint takes $R$ to $(R,\widetilde{\mathbb F_p+R^{\circ\circ}})$ where $\widetilde{\mathbb F_p+R^{\circ\circ}}\subset R^\circ$ denotes the integral closure of $\mathbb F_p+R^{\circ\circ}$.

In the following proposition, we ignore the set-theoretic issues, which are resolved as usual by first fixing a cutoff cardinal $\kappa$, and then taking a union over all $\kappa$.

\begin{proposition} The functors $(R,R^+)\mapsto R$ and $R\mapsto (R,\widetilde{\mathbb F_p+R^\circ\circ})$ induce adjoint continuous maps of sites, and hence induce adjoint morphisms of $\infty$-topoi
\[
a^\ast: \mathrm{vStack}_{/\mathcal N}\to \mathrm{arcStack},\ b^\ast: \mathrm{arcStack}\to \mathrm{vStack}_{/\mathcal N},
\]
with $b^\ast$ the right adjoint of $a^\ast$. Concretely, $a^\ast X$ is the sheafification of $R\mapsto X(R,R^\circ)$, and
\[
(b^\ast Y)(R,R^+)= Y(R).
\]

The composite $b^\ast a^\ast X$ is the sheafification of
\[
(R,R^+)\mapsto X(R,R^\circ)
\]
which for separated $X$ yields the canonical compactification $\overline{X}$ of $X$ of \cite[Proposition 18.6]{ECoD}.
\end{proposition}

\begin{proof} Both functors are left adjoints and hence commute with all colimits, in particular with finite colimits; i.e.~on the opposite categories they commute with finite limits. They also preserve the natural notion of covers, and so yield morphisms of sites. All assertions follow directly from the definitions.
\end{proof}

The following corollary is immediate.

\begin{corollary} The functor $b^\ast$ yields a fully faithful functor
\[
\mathrm{arcStack}\hookrightarrow \mathrm{vStack}_{/\mathcal N}
\]
whose image consists of all partially proper v-stacks; i.e.~all $X$ such that $X(R,R^+)=X(R,R^\circ)$ for all Huber pairs $(R,R^+)$.$\hfill \Box$
\end{corollary}

Many naturally occuring v-stacks are partially proper. This includes analyticifications of algebraic varieties, $\mathrm{Div}^1$, $\mathrm{Bun}_G$, Banach--Colmez spaces, etc. However, quasicompact open subsets are virtually never partially proper. Such quasicompact open subsets play a key technical role in \cite{ECoD} as they are responsible for the compact generation of many categories, and many arguments proceed by reduction to compact objects. In the Berkovich setting, compact generation is not true anymore, but the categories are usually still compactly assembled, i.e.~(in the stable setting) dualizable. This means that arguments involving reduction to compact objects have to be replaced by arguments involving reduction to compact morphisms.

Having to work over $\mathcal N$ seems a bit artificial, and in fact we can modify the Berkovich side to get rid of this. Namely, we can also work with the test category of perfectoid Tate $\mathbb F_p$-algebras, and consider small arc-stacks there. The whole theory developed above works in that setting, too; in fact, whenever we work over some base $R$, we can fix a norm on $R$ to put us in the situation considered above. This yields
\[
a^{\prime\ast}: \mathrm{vStack}\to \mathrm{arcStack}', b^{\prime\ast}: \mathrm{arcStack}'\to \mathrm{vStack}
\]
for this variant $\mathrm{arcStack}'$ of small arc-stacks over $\mathbb F_p$, and $\mathcal D_{\mathrm{mot}}(X)$ is defined for any $X\in \mathrm{arcStack}'$. We note that $\mathrm{arcStack}'$ can also be defined by using the same test category of perfectoid Huber pairs $(R,R^+)$ as used for v-stacks, but endowing them with the stronger arc-topology, where only surjectivity on rank $1$-points is required. This means that $\mathrm{Spa}(R,R^\circ)\to \mathrm{Spa}(R,R^+)$ is an arc-cover, and so the choice of $R^+$ is ignored after sheafification in this topology.

\begin{proposition} Let $X$ be a small v-stack with corresponding small arc-stack $Y = a^{\prime\ast} X$. Then for any $n$ prime to $p$, there is a fully faithful embedding
\[
\mathcal D_{\mathrm{mot}}(Y,\mathbb Z/n\mathbb Z)\cong \mathcal D_{\mathrm{et}}(Y,\mathbb Z/n\mathbb Z)\hookrightarrow \mathcal D_{\mathrm{et}}(X,\mathbb Z/n\mathbb Z)
\]
whose essential image consists of all overconvergent \'etale sheaves, i.e.~those whose geometric fibres at higher-rank geometirc points $\mathrm{Spa}(C,C^+)$ agree with the stalks at their generic point $\mathrm{Spa}(C,\mathcal O_C)$.

This full embedding is compatible with pullback, tensor products, and pushforward along proper maps of finite cohomological dimension.
\end{proposition}

\begin{proof} By descent, we can assume that $X$ is strictly totally disconnected. Then the functor is the fully faithful inclusion of $\mathcal D(\pi_0 X,\mathbb Z/n\mathbb Z)$ into $\mathcal D_{\mathrm{et}}(X,\mathbb Z/n\mathbb Z)$ of \cite[Proposition 13.13, Proposition 22.7]{ECoD} (noting that in this strictly totally disconnected case, we can pass to Postnikov limits and replace $D^+$ by $D$). The essential image consists of the overconvergent sheaves. It is clear that this full inclusion is compatible with pullback, tensor products and proper direct images.
\end{proof}

\begin{proposition} Let $X$ be a spatial diamond of finite cohomological dimension with $Y=a^{\prime\ast} X$. Then the full inclusion
\[
\mathcal D_{\mathrm{mot}}(Y,\mathbb Z/n\mathbb Z)\cong \mathcal D_{\mathrm{et}}(Y,\mathbb Z/n\mathbb Z)\hookrightarrow \mathcal D_{\mathrm{et}}(X,\mathbb Z/n\mathbb Z)
\]
has a colimit-preserving right adjoint, and $\mathcal D_{\mathrm{et}}(X,\mathbb Z/n\mathbb Z)$ is compactly generated.
\end{proposition}

This writes $\mathcal D_{\mathrm{et}}(Y,\mathbb Z/n\mathbb Z)$ explicitly as a retract of a compactly generated $\infty$-category.

\begin{proof} We know that $\mathcal D_{\mathrm{mot}}(Y)$ is compactly assembled, with a generating class of compact morphisms being given by $\mathbb Z_{\mathrm{mot}}[U]\to \mathbb Z_{\mathrm{mot}}[V]$ for strict inclusions $U\subset V\subset \mathbb A^n_Y$. The corresponding inclusion $b^\ast U\times_{b^\ast Y} X\subset b^\ast V\times_{b^\ast Y} X\subset \mathbb A^n_X$ factors over a quasicompact open subset $W\subset \mathbb A^n_X$. The full inclusion takes $\mathbb Z_{\mathrm{mot}}[U]$ (or rather its reduction modulo $n$) to the homology $b^\ast U\times_{b^\ast Y} X\to X$. This factors over the homology of the quasicompact smooth $W\to X$, and this defines a compact object of $\mathcal D_{\mathrm{et}}(X,\mathbb Z/n\mathbb Z)$. Thus, the full inclusion preserves compact morphisms; equivalently, the right adjoint preserves colimits. The compact generation of $\mathcal D_{\mathrm{et}}(X,\mathbb Z/n\mathbb Z)$ is \cite[Proposition 20.17]{ECoD}.
\end{proof}

\bibliographystyle{amsalpha}
\bibliography{BerkovichMotives}

\end{document}